\documentclass[12pt,twoside,leqno]{article}
\usepackage[T1]{fontenc}
\usepackage{amsfonts}
\usepackage{amsmath,amsthm}
\usepackage{amssymb,latexsym}

\theoremstyle{plain}
\newtheorem{theorem}{Theorem}[section]
\newtheorem{lemma}[theorem]{Lemma}
\newtheorem{proposition}[theorem]{Proposition}
\newtheorem{corollary}[theorem]{Corollary}

\theoremstyle{definition}
\newtheorem{example}[theorem]{Example}
\newtheorem{definition}[theorem]{Definition}
\theoremstyle{remark}
\newtheorem{remark}[theorem]{Remark}
\newtheorem{question}[theorem]{Question}
\newtheorem{problem}[theorem]{Problem}

\DeclareMathOperator{\inter}{{\rm int}}
\DeclareMathOperator{\cl}{{\rm cl}}
\DeclareMathOperator{\bd}{{\rm bd}}
\DeclareMathOperator{\osc}{{\rm osc}}
\DeclareMathOperator{\Max}{{\rm Max}}
\DeclareMathOperator{\BMax}{\bf{Max}}

\begin{document}
\title{$\mathbf{E}$-compact extensions in the absence of the Axiom of Choice}
\author{ AliReza Olfati \\
Department of Mathematics, Faculty of Basic Sciences,\\ 
Yasouj University, Daneshjoo St.,  Yasouj 75918-74934, Iran\\
E-mail: alireza.olfati@yu.ac.ir
\and 
Eliza Wajch\\
Institute of Mathematics,
Faculty of Exact and Natural Sciences,\\
Siedlce University,\\
ul. 3 Maja 54, 08-110 Siedlce, Poland\\
E-mail: eliza.wajch@gmail.com}
\maketitle
\begin{abstract}
The main aim of this work is to show, in the absence of the Axiom of Choice, fundamental results on $\mathbf{E}$-compact extensions of $\mathbf{E}$-completely regular spaces, in particular, on Hewitt realcompactifications and Banaschewski compactifications. Some original results concern a special subring of the ring of all continuous real functions on a given zero-dimensional $T_1$-space. New facts about $P$-spaces, Baire topologies and $G_{\delta}$-topologies are also shown. Not all statements investigated here have proofs in $\mathbf{ZF}$. Some statements are shown equivalent to the Boolean Prime ideal Theorem, some are consequences of the Axiom of Countable Multiple Choices. \\

 
\end{abstract}

\tableofcontents

\section{Introduction}
\label{Intro}
\subsection{Basic terminology and the main aim} 
\label{s1.1}
In this article, the intended context of reasoning and statements of theorems is the Zermelo-Fraenkel set theory $\mathbf{ZF}$ in which the Axiom of Choice (denoted by $\mathbf{AC}$) is deleted. To stress the fact that a result is proved in $\mathbf{ZF}$ or $\mathbf{ZF+\Psi}$ (where $\mathbf{\Psi}$ is a statement independent of $\mathbf{ZF}$), we shall write
at the beginning of the statements of the theorems and propositions [$%
\mathbf{ZF}$] or [$\mathbf{ZF+\Psi}$], respectively. The system $\mathbf{ZF+AC}$ is denoted by $\mathbf{ZFC}$. 

Before we describe in brief in Section \ref{s1.6} what our work is about, let us establish notation, recall several definitions and known facts.

We usually denote topological or metric spaces with boldface letters, and their underlying sets with lightface letters. For instance, if $\mathbf{X}$ is a topological or metric space, then $X$ is its underlying set.  

In what follows, we denote by $\tau_{nat}$ the natural topology on the real line $\mathbb{R}$, and by $d_e$ the metric on $\mathbb{R}$ defined by: for all $x,y\in\mathbb{R}$, $d_e(x,y)=|x-y|$. If this is not misleading, $\mathbb{R}$ will denote both the topological space $\langle \mathbb{R}, \tau_{nat}\rangle$ and the metric space $\langle \mathbb{R}, d_e\rangle$. The set of all positive integers is denoted by $\mathbb{N}$. We recall that a set $X$ is \emph{Dedekind-finite} if no proper subset of $X$ is equipotent to $X$; otherwise, $X$ is \emph{Dedekind-infinite}. We may assume that $\mathbb{N}=\omega\setminus\{0\}$ where $0$ is the empty set $\emptyset$, and $\omega$ is the set of all Dedekind-finite ordinal numbers of von Neumann. For $n\in\omega$, $n+1=n\cup\{n\}$. A set $X$ is \emph{finite} if there exists $n\in\omega$ such that $X$ is equipotent to $n$. A set equipotent to a subset of $\omega$ is called \emph{countable}. Sets equipotent to $\omega$ are called \emph{denumerable}.

Given a set $X$,  the power set of $X$ is denoted by $\mathcal{P}(X)$. The set of all finite subsets of $X$ is denoted by $[X]^{<\omega}$, and the set of all countable subsets of $X$ is denoted by $[X]^{\leq\omega}$.  The discrete space $\mathbf{X}=\langle X, \mathcal{P}(X)\rangle$ is denoted by $X_{disc}$.  For simplicity, the discrete space $\mathbb{N}_{disc}$ is denoted by $\mathbb{N}$.

Suppose that $\mathbf{X}=\langle X, \tau\rangle$ is a given topological space, and $A\subseteq X$. Then $\tau|_{A}=\{U\cap A: U\in\tau\}$ and $\mathbf{A}=\langle A, \tau|_{A}\rangle$ is the (topological) subspace of $\mathbf{X}$. If it is not stated otherwise, all subsets of $X$  will be considered as topological subspaces of $\mathbf{X}$ and, for simplicity, if $A\subseteq X$, then the subspace $\mathbf{A}$ of $\mathbf{X}$ will be also denoted by $A$. The closure of $A$ in $\mathbf{X}$ is denoted by $\cl_{\mathbf{X}}(A)$. The interior of $A$ in $\mathbf{X}$ is denoted by $\inter_{\mathbf{X}}(A)$, and $\bd_{\mathbf{X}}(A)=\cl_{\mathbf{X}}(A)\setminus\inter_{\mathbf{X}}(A)$.

For topological spaces $\mathbf{X}=\langle X, \tau_X\rangle$ and $\mathbf{E}=\langle E, \tau_E\rangle$, $C(\mathbf{X}, \mathbf{E})$ stands for the set of all $\langle \tau_X, \tau_E\rangle$-continuous mappings of $\mathbf{X}$ into $\mathbf{E}$, and $C^{\ast}(\mathbf{X}, \mathbf{E})$ is the set of all $f\in C(\mathbf{X}, \mathbf{E})$ such that $\cl_{\mathbf{E}}(f[X])$ is compact. If this is not misleading, instead of a $\langle \tau_X, \tau_E\rangle$-continuous mapping, we write simply a \emph{continuous mapping}. As usual, $C(\mathbf{X})=C(\mathbf{X}, \mathbb{R})$ and $C^{\ast}(\mathbf{X})=C^{\ast}(\mathbf{X}, \mathbb{R})$. In $\mathbf{ZFC}$, a systematic study of $\mathbf{E}$-completely regular and $\mathbf{E}$-compact spaces in the sense of the following definition was started by R. Engelking and S. Mr\'owka in \cite{enmr}.

\begin{definition}
	\label{s1:d1}
	(Cf. e.g., \cite{enmr}, \cite{mr1}--\cite{mr5}.) Let $\mathbf{E}$ be a given topological space. Then a topological space $\mathbf{X}$ is called  $\mathbf{E}$-\emph{completely regular} (respectively,  $\mathbf{E}$-\emph{compact}) if there exists a non-empty set $J$ such that $\mathbf{X}$ is homeomorphic with a subspace (respectively, a closed subspace) of $\mathbf{E}^J$. (The definition of $\mathbf{E}^J$ is recalled in Section \ref{s1.2} above Definition \ref{s1:d7}.)
\end{definition}

The general concept of $\mathbf{E}$-compactness was introduced as a common generalization of realcompactness and $\mathbb{N}$-compactness in the sense of items (i) and (ii) of the following definition:

\begin{definition}
	\label{s1:d2}
	A topological space $\mathbf{X}$ is called:
	\begin{enumerate}
		\item[(i)] \emph{realcompact} if $\mathbf{X}$ is $\langle \mathbb{R}, \tau_{nat}\rangle$-compact;
		\item[(ii)] $\mathbb{N}$-\emph{compact} if $\mathbf{X}$ is $\langle \mathbb{N}, \mathcal{P}(\mathbb{N})\rangle$-compact;
		\item[(iii)] \emph{Cantor-compact} if $\mathbf{X}$ is $\langle \{0, 1\}, \mathcal{P}(\{0, 1\})\rangle$-compact;
		\item[(iv)] \emph{Tychonoff-compact} if $\mathbf{X}$ is $\langle [0, 1], \tau_{nat}|_{[0, 1]}\rangle$-compact.
	\end{enumerate}
\end{definition}

\begin{definition}
	\label{s1:d3}
	Let $\mathbf{X}$, $\mathbf{Y}$ and $\mathbf{E}$ be  topological spaces.
	\begin{enumerate}
		\item[(a)] An \emph{extension} of $\mathbf{X}$ is an ordered pair $\langle\mathbf{Y}, h\rangle$ where  $h$ is a homeomorphic embedding of $\mathbf{X}$ into $\mathbf{Y}$ such that $h[X]$ is dense in $\mathbf{Y}$.
		\item[(b)] If, for $i\in\{1,2\}$,  $\langle\mathbf{Y}_i, h_i\rangle$ are extensions of $\mathbf{X}$, then:
		\begin{enumerate}
			\item[(i)] we write $\langle \mathbf{Y}_1, h_1\rangle\leq \langle \mathbf{Y}_2, h_2\rangle$ if there exists a continuous surjection $f: \mathbf{Y}_2\to \mathbf{Y}_1$ such that $f\circ h_2=h_1$;
			\item[(ii)] we write $\langle \mathbf{Y}_1, h_1\rangle\thickapprox\langle\mathbf{Y}_2, h_2\rangle$ and say that $\langle \mathbf{Y}_1, h_1\rangle$ and $\langle\mathbf{Y}_2, h_2\rangle$ are \emph{equivalent extensions} of $\mathbf{X}$ if there exists a homeomorphism $h: \mathbf{Y}_1\to\mathbf{Y}_2$ such that $h\circ h_1=h_2$.
		\end{enumerate}
		\item[(c)] If $\mathcal{P}$ is a topological property, then we say that an extension $\langle\mathbf{Y}, h\rangle$ has $\mathcal{P}$ if $\mathbf{Y}$ has $\mathcal{P}$. In particular, an extension $\langle \mathbf{Y}, h\rangle$ of $\mathbf{X}$ is called a \emph{Hausdorff} (respectively, \emph{compact}, $\mathbf{E}$-\emph{compact}) extension of $\mathbf{X}$ if $\mathbf{Y}$ is a Hausdorff (respectively, compact, $\mathbf{E}$-compact) space. Compact extensions of $\mathbf{X}$ are called \emph{compactifications} of $\mathbf{X}$.
	\end{enumerate}
\end{definition}

\begin{remark}
	\label{s1:r4} 
	Let $\mathbf{X}$ be a topological space.
	\begin{enumerate}
		\item[(i)]Suppose $\langle \mathbf{Y}_1, h_1\rangle$ and $\langle \mathbf{Y}_2, h_2\rangle$ are Hausdorff extensions of $\mathbf{X}$. Then $\langle \mathbf{Y}_1, h_1\rangle\thickapprox \langle\mathbf{Y}_2, h_2\rangle$ if and only if $\langle \mathbf{Y}_1, h_1\rangle\leq \langle\mathbf{Y}_2, h_2\rangle$ and $\langle \mathbf{Y_2}, h_2\rangle\leq\langle \mathbf{Y}_1, h_1\rangle$.  
		
		\item[(ii)] For an extension $\langle\mathbf{Y}, h\rangle$ of $\mathbf{X}$, the subspace $Y\setminus h[X]$ of $\mathbf{Y}$ is called the \emph{remainder} of $\langle\mathbf{Y}, h\rangle$, the space $\mathbf{X}$ is identified with the subspace $h[X]$ of $\mathbf{Y}$, and $h$ is identified  with the identity mapping $\text{id}_{X}$ on $X$. 
		
		\item[(iii)] If $\langle \alpha\mathbf{X}, \alpha\rangle$ is a compactification of $\mathbf{X}$, we denote this compactification by $\alpha\mathbf{X}$, its underlying set by $\alpha X$,  and its remainder by $\alpha X\setminus X$. Analogously, if $\langle \mathbf{Y}, h\rangle$ is an arbitrary extension of $\mathbf{X}$, the space $\mathbf{Y}$ and the extension $\langle \mathbf{Y}, h\rangle$ can be both denoted by $h\mathbf{X}$.
	\end{enumerate}
\end{remark} 

In the sequel, we will be concerned mainly with Hausdorff extensions of Hausdorff spaces, in particular, for a given Hausdorff space $\mathbf{E}$, we will be concerned  with $\mathbf{E}$-compact extensions of $\mathbf{E}$-completely regular spaces.
We need the following generalization of the concept of the Hewitt realcompactification of a Tychonoff space:

\begin{definition}
	\label{s1:d5}
	Suppose that $\mathbf{X}$ and $\mathbf{E}$ are Hausdorff spaces such that $\mathbf{X}$ is $\mathbf{E}$-completely regular.
	An $\mathbf{E}$-compact extension $\langle \mathbf{Y}, h\rangle$ of $\mathbf{X}$ is called  a \emph{Hewitt} $\mathbf{E}$-\emph{compact extension} of $\mathbf{X}$ if it satisfies  the following condition:
	$$(\forall f\in C(\mathbf{X}, \mathbf{E}))(\exists \tilde{f}\in C(\mathbf{Y}, \mathbf{E}))\text{ } \tilde{f}\circ h=f.$$
\end{definition}

\begin{remark}
	\label{s1:r6}
	Suppose that $\mathbf{X}$ and $\mathbf{E}$ are Hausdorff spaces such that $\mathbf{X}$ is $\mathbf{E}$-completely regular. Under this assumption, it will be clearly stated in Theorem\ref{s4:t8} that it holds in $\mathbf{ZF}$ that there exists a unique (up to the equivalence $\approx$) Hewitt $\mathbf{E}$-compact extension of $\mathbf{X}$. This is why we can denote the uniquely determined (up to $\approx$) Hewitt $\mathbf{E}$-compact extension of $\mathbf{X}$ by $v_{\mathbf{E}}\mathbf{X}$. In particular, $v\mathbf{X}=v_{\mathbb{R}}\mathbf{X}$ is the \emph{Hewitt realcompactification} of $\mathbf{X}$.
\end{remark}

Needless to say that a huge number of articles on $E$-compactness in $\mathbf{ZFC}$ have appeared (see, e.g., \cite{enmr}, \cite{mr0}--\cite{mr5}). Some general results on $\mathbf{E}$-compactness have been applied to an intensive investigation of $\mathbf{E}$-compact extensions of topological spaces, in particular, to  Hewitt realcompactifications of Tychonoff spaces in $\mathbf{ZFC}$. Basic results on realcompatness in $\mathbf{ZFC}$ are collected in \cite[Chapter 3.11]{en}.  However, they have not been investigated in deep in the absence of $\mathbf{AC}$ yet. 

The main aim of our present work is to start a systematic study of realcompact spaces, $\mathbb{N}$-compact spaces and related topics in $\mathbf{ZF}$ by having a deeper look at Hewitt $\mathbf{E}$-compact extensions of $\mathbf{E}$-completely regular Hausdorff spaces in $\mathbf{ZF}$. Clearly, to do it well, we need to apply also Hausdorff compactifications in $\mathbf{ZF}$.  We recommend \cite{kw1} as the most extensive introduction to Hausdorff compactifications in the absence of $\mathbf{AC}$ that has ever been written so far. 

Before we pass to the body of the paper, let us establish other notation and terminology, and recall several useful results for references in the sequel, also to illuminate the current knowledge about the main topics of this work in $\mathbf{ZF}$. 

\subsection{Products, evaluation maps and continuous extensions}
\label{s1.2}

A \emph{base} (or, equivalently, an \emph{open base}) of a topological space $\mathbf{X}=\langle X, \tau\rangle$ is a base for $\tau$ (that is, a base for the open sets of $\mathbf{X}$). A \emph{ base for the closed sets} or, equivalently, a \emph{closed base} of $\mathbf{X}$ is a family $\mathcal{C}$ of subsets of $X$ such that $\{X\setminus C: C\in\mathcal{C}\}$ is a base of $\mathbf{X}$. 

Given  a set $J$ and a family $\mathcal{A}=\{A_j: j\in J\}$ of sets, every function $a\in\prod\limits_{j\in J}A_j$ is called a \emph{choice function} of $\mathcal{A}$. If $I$ is an infinite subset of $J$, then every choice function of $\{A_j: j\in I\}$ is called a \emph{partial choice function} of $\mathcal{A}$. A \emph{multiple choice function} of $\mathcal{A}$ is every function $a\in\prod\limits_{j\in J}([A_j]^{<\omega}\setminus\{\emptyset\})$. For $j_0\in J$, the standard projection of $\prod\limits_{j\in J}A_j$ into $A_{j_0}$ is denoted by $\pi_{j_0}$. If $\mathcal{A}$ consists of non-empty sets but fails to have a choice function, then, for any $j_0\in J$, the projection $\pi_{j_0}:\prod\limits_{j\in J}A_j\to A_{j_0}$ is not a surjection. 

Given a family $\{ \mathbf{X}_j: j\in J\}$ of topological spaces such that, for every $j\in J$,  $\mathbf{X}_j=\langle X_j, \tau_j\rangle$, we equip the set $X=\prod\limits_{j\in J}X_j$ with the product topology $\tau^{prod}={\prod\limits_{j\in J}}^{prod}\tau_j$ whose base is the family 
$$\left\{\bigcap\limits_{j\in K}\pi_j^{-1}[V_j]: (K\in[J]^{<\omega}\setminus\{\emptyset\})\wedge ((\forall j\in K)\text{ } V_j\in\tau_j)\right\},$$
called  the \emph{canonical base} of $\mathbf{X}=\langle X, \tau^{prod}\rangle=\prod\limits_{j\in J}\mathbf{X}_j$.  If, for every $j\in J$, $\mathbf{X}_j=\mathbf{Y}$, then $\prod\limits_{j\in J}\mathbf{X}_j$ is denoted by $\mathbf{Y}^J$. In particular, if $J$ is infinite, $2=\{0,1\}$ and $\mathbf{2}=2_{disc}$, then $\mathbf{2}^J$ is called a \emph{Cantor cube}, and  $[0,1]^J$ is called a \emph{Tychonoff cube}.  

\begin{definition}
	\label{s1:d7}
	Let $J$ be a non-empty set.  Let $X$ be a set,  $\{Y_j: j\in J\}$ be a family of sets, and  $\mathcal{F}=\{f_j: j\in J\}$ a family of mappings such that, for every $j\in J$,  $f_j: X\to Y_j$. Then the mapping $e_{\mathcal{F}}: X\to\prod\limits_{j\in J}Y_j$ defined by:
	$$(\forall x\in X)(\forall j\in J)\text{ } e_{\mathcal{F}}(x)(j)=f_j(x)$$
	is called the \emph{evaluation mapping determined by} $\mathcal{F}$.
\end{definition}

One can easily check that the following theorem from \cite[Exercise 2.3.D]{en} is provable in $\mathbf{ZF}$:

\begin{theorem}
	\label{s1:t8}
	$[\mathbf{ZF}]$
	Let $J$ be a non-empty set.  Let $\mathbf{X}$ be a topological space,  $\{\mathbf{Y}_j: j\in J\}$ be a family of topological spaces, and  $\mathcal{F}=\{f_j: j\in J\}$ a family of mappings such that, for every $j\in J$,  $f_j\in C(\mathbf{X}_j, \mathbf{Y}_j)$. For $K\in [J]^{<\omega}\setminus\{\emptyset\}$, let $\mathcal{F}_K=\{f_j: j\in K\}$. Then the evaluation mapping $e_{\mathcal{F}}$ is a homeomorphic embedding of $\mathbf{X}$ into $\prod\limits_{j\in J}\mathbf{Y}_j$ if and only if the family $\{ e_{\mathcal{F}_K}: K\in [J]^{<\omega}\setminus\{\emptyset\}\}$ separates points of $X$ and separates points from closed sets of $\mathbf{X}$.
\end{theorem}

\begin{definition}
	\label{s1:d9}
	Let $\mathcal{A}$ be a family of subsets of a set $X$.
	\begin{enumerate}
		\item[(i)] $\mathcal{A}_{\delta}=\{ \bigcap \mathcal{C}: \mathcal{C}\in [\mathcal{A}]^{\leq\omega}\setminus\{\emptyset\}\}$. 
		\item[(ii)] $\mathcal{A}_{\sigma}=\{ \bigcup \mathcal{C}: \mathcal{C}\in [\mathcal{A}]^{\leq\omega}\}.$
		\item[(iii)] $\mathcal{A}$ is \emph{closed under countable intersections} if $\mathcal{A}_{\delta}\subseteq\mathcal{A}$. If, for all $A,B\in\mathcal{A}$, $A\cap B\in\mathcal{A}$, then it is said that $\mathcal{A}$ is \emph{closed under finite intersections}.
		\item[(iv)] $\mathcal{A}$ is \emph{closed under countable unions} if $\mathcal{A}_{\sigma}\subseteq\mathcal{A}$. If, for all $A, B\in\mathcal{A}$, $A\cup B\in\mathcal{A}$, then it is said that $\mathcal{A}$ is \emph{closed under finite unions}.
	\end{enumerate}
\end{definition}

It was noticed in \cite{pw} that the following very useful modification of the Taimanov Extension Theorem (see \cite[Theorem 3.2.1]{en}), proved true in $\mathbf{ZFC}$ independently in \cite{blasco} and \cite{w1}, is also provable in $\mathbf{ZF}$:

\begin{theorem}
	\label{s1:t10}
	$($Taimanov's Extension Theorem.$)$ $[\mathbf{ZF}]$ Let $\mathbf{X}$ be a dense subspace of a topological space $\mathbf{T}$. Let $\mathcal{C}$ be a closed under finite intersections base for the closed sets of a compact Hausdorff space $\mathbf{Y}$. Let $f\in C(\mathbf{X}, \mathbf{Y})$. Then there exists $\tilde{f}\in C(\mathbf{T}, \mathbf{Y})$ such that $\tilde{f}\upharpoonright X= f$  if and only if, for every pair $A,B$ of disjoint members of $\mathcal{C}$, the following holds: 
	$$\cl_{\mathbf{T}}(f^{-1}[A])\cap\cl_{\mathbf{T}}(f^{-1}[B])=\emptyset.$$
\end{theorem}

A relatively simple corollary to Taimanov's Extension Theorem is the following result due to Blasco (see \cite{blasco}):

\begin{theorem}
	\label{s1:t11}
	$($Blasco's Extension Theorem.$)$ $[\mathbf{ZF}]$ Let $\mathbf{X}$ be a dense subspace of a topological space $\mathbf{T}$, and let $f\in C^{\ast}(\mathbf{X})$. Then $f$ has a continuous extension $\tilde{f}:\mathbf{T}\to\mathbb{R}$ if and only if, for every pair $a,b$ of real numbers with $a<b$, the sets $f^{-1}[(-\infty, a]]$ and $f^{-1}[[b, +\infty)]$ have disjoint closures in $\mathbf{T}$.
\end{theorem}

It is obvious that the following well-known proposition holds in $\mathbf{ZF}$:
\begin{proposition}
	\label{s1:p12}
	$[\mathbf{ZF}]$ Let $\mathbf{X}$ be a dense subspace of a topological space $\mathbf{T}$ and let $\mathbf{Y}$  be a Hausdorff space. If $f,g\in C(\mathbf{T},\mathbf{Y})$ and $f\upharpoonright X=g\upharpoonright X$, then $f=g$.
\end{proposition}

Given an extension $h\mathbf{X}$ of a topological space $\mathbf{X}$, we let 
$$C_{h}(\mathbf{X})=\{ f\in C(\mathbf{X}): (\exists \tilde{f}\in C(h\mathbf{X})) \tilde{f}\circ h= f\}.$$
In view of Proposition \ref{s1:p12}, for every $f\in C_h(\mathbf{X})$, there exists a unique $\tilde{f}\in C(h\mathbf{X})$ such that $\tilde{f}\circ h=f$. If $f\in C_h(\mathbf{X})$, the unique $\tilde{f}\in C(h\mathbf{X})$ with $\tilde{f}\circ h=f$ will be denoted by $f^h$. 

\subsection{A list of weaker forms of $\mathbf{AC}$ }
\label{s1.3}

Working in $\mathbf{ZF}$, we also show statements provable in $\mathbf{ZFC}$ but independent of $\mathbf{ZF}$ by proving that they imply forms known to be false in some models of $\mathbf{ZF}$.  Given a statement $\mathbf{\Psi}$ which cannot be proved true in $\mathbf{ZF}$, if possible, we show a weaker form  $\mathbf{\Phi}$ of $\mathbf{AC}$ such that $\mathbf{\Psi}$ holds in every model of $\mathbf{ZF+\Phi}$. In particular, we apply the forms from \cite{hr}, denoted and  defined as follows:

\begin{definition}
	\label{s1:d13}
	\begin{enumerate}
		\item $\mathbf{CAC}$  (\cite[Form 8]{hr}): Every denumerable family of non-empty sets has a choice function.
		
		\item $\mathbf{CMC}$ (\cite[Form 126]{hr}): Every denumerable family of non-empty sets has a multiple choice function.
		
		\item $\mathbf{CMC}_{\omega}$ (\cite[Form 350]{hr}): Every denumerable family of denumerable sets has a multiple choice function.	
		
		\item $\mathbf{BPI}$ (the Boolean Prime Ideal Theorem, \cite[Form 14]{hr}): Every Boolean algebra has a prime ideal.
	\end{enumerate}
\end{definition}

\begin{remark}
	\label{s1:r14}
	We shall use the following equivalents of $\mathbf{BPI}$:
	\begin{enumerate}
		\item[(i)] $\mathbf{UFT}$ (the Ultrafilter Theorem, \cite[Form 14A]{hr}): For every non-empty set $S$, every filter in $\mathcal{P}(S)$ can be extended to an ultrafilter in $\mathcal{P}(S)$.
		\item[(ii)] Tychonoff's Compactness Theorem for Compact Hausdorff spaces (\cite[Form 14J]{hr}): Every product of compact Hausdorff spaces is compact.
		\item[(iii)] Compactness of Cantor cubes (\cite[Form 14K]{hr}): For every infinite set $J$, the Cantor cube $\mathbf{2}^J$ is compact.
	\end{enumerate} 
	The readers can find more detailed information about several other statements equivalent to $\mathbf{BPI}$ (so also to $\mathbf{UFT}$), for instance, in \cite{her}.
\end{remark}

\subsection{Function spaces, zero-sets, $P$-spaces,  $G_{\delta}$-topologies and Baire topologies}
\label{s1.4}

Let  $\mathbf{Y}=\langle Y, \rho\rangle$ be a given metric space. For every $y\in Y$ and a positive real number $\varepsilon$, 
$$B_{\rho}(y, \varepsilon)=\{z\in Y: \rho(y,z)<\varepsilon\}.$$
We denote by $\tau(\rho)$ the topology on $Y$ having the family 
$$\{B_{\rho}(y, \varepsilon): y\in Y\wedge \varepsilon\in (0, +\infty)\}$$
as a base. The topology $\tau(\rho)$ is usually called \emph{induced by} $\rho$. For $A\subseteq Y$, let 
$$\delta_{\rho}(A)=\begin{cases}\sup\{ \rho(x,y): x,y\in A\} &\text{ if } A\neq\emptyset;\\
	0 &\text{ if } A=\emptyset.\end{cases}$$

For a topological space $\mathbf{X}=\langle X, \tau\rangle$ and the metric space $\mathbf{Y}=\langle Y, \rho\rangle$, let $C(\mathbf{X},\mathbf{Y})=C(\mathbf{X}, \langle Y, \tau(\rho)\rangle)$.  We denote by $d_u$ the metric of uniform convergence on $C(\mathbf{X},\mathbf{Y})$ defined as follows:
$$(\forall f, g\in C(\mathbf{X}, \mathbf{Y})) \text{ } d_u(f,g)=\sup\{ \min\{\rho (f(x), g(x)), 1)\}: x\in X\}.$$
The topological space $\langle C(\mathbf{X},\mathbf{Y}), \tau(d_u)\rangle$ is denoted by $C_{u}(\mathbf{X},\mathbf{Y})$.  For $f\in C(\mathbf{X},\mathbf{Y})$ and $A\subseteq X$, the oscillation of $f$ on $A$ is defined by:
$$\osc_A(f)=\begin{cases} \sup\{\rho(f(x), f(y)): x,y\in A\} &\text{ if } A\neq\emptyset;\\
	0 &\text{ if } A=\emptyset.\end{cases}$$

The space $C_{u}(\mathbf{X}, \langle \mathbb{R}, d_e\rangle)=C_{u}(\mathbf{X}, \mathbb{R})$ is denoted by $C_{u}(\mathbf{X})$. 

For $f\in C(\mathbf{X})$, let $Z(f)=f^{-1}[\{0\}]$. Then $Z(f)$ is the \emph{zero-set} determined by $f$. Let $\mathcal{Z}(\mathbf{X})=\{Z(f): f\in C(\mathbf{X})\}$. The members of $\mathcal{Z}(\mathbf{X})$ are called \emph{zero-sets} of $\mathbf{\mathbf{X}}$, their complements are called \emph{co-zero sets} of $\mathbf{\mathbf{X}}$. The family of all co-zero sets of $\mathbf{X}$ is denoted by $\mathcal{Z}^{c}(\mathbf{X})$. The family of all clopen subsets of a topological space $\mathbf{X}$ is denoted by $\mathcal{CO}(\mathbf{X})$. The family of all $G_{\delta}$-sets of $\mathbf{X}$ is denoted by $\mathcal{G}_{\delta}(\mathbf{X})$. The family of all $F_{\sigma}$-sets of $\mathbf{X}$ is denoted by $\mathcal{F}_{\sigma}(\mathbf{X})$.

\begin{definition}
	\label{s1:d15}
	A topological space $\mathbf{X}$ is called:
	\begin{enumerate}
		\item[(i)] \emph{completely regular} if $\mathcal{Z}^c(\mathbf{X})$ is a base of $\mathbf{X}$; a completely regular $T_1$-space is called \emph{Tychonoff} or a $T_{3\frac{1}{2}}$-space;
		\item[(ii)] \emph{zero-dimensional} if $\mathcal{CO}(\mathbf{X})$ is a base of $\mathbf{X}$;
		\item[(iii)] \emph{strongly zero-dimensional} if $\mathbf{X}$ is completely regular and, for every pair $Z_1,Z_2$ of disjoint zero-sets in $\mathbf{X}$, there exists $U\in\mathcal{CO}(\mathbf{X})$ such that $Z_1\subseteq U\subseteq X\setminus Z_2$;
		\item[(iv)] a \emph{P-space} if $\mathcal{G}_{\delta}(\mathbf{X})$ is the topology of $\mathbf{X}$ (cf. \cite{gh});
		\item[(v)] \emph{functionally Hausdorff} if, for every pair $x_0, x_1$ of distinct points of $\mathbf{X}$, there exists $f\in C(\mathbf{X})$ such that $f(x_0)=0$ and $f(x_1)=1$;
		\item[(vi)] \emph{rimcompact} (or, equivalently, \emph{peripherally compact}) if there exists a base $\mathcal{B}$ of $\mathbf{X}$ such that, for every $U\in\mathcal{B}$, $\bd_{\mathbf{X}}(U)$ is compact.
	\end{enumerate}
\end{definition}

One can easily check that, in $\mathbf{ZF}$, every rimcompact Hausdorff space is regular. Clearly, a topological space $\mathbf{X}$ is Tychonoff if and only if it is $\mathbb{R}$-completely regular.

\begin{definition}
	\label{s1:d16}
	For a topological space $\mathbf{X}$, we let $\mathcal{CO}_{\sigma}(\mathbf{X})=\mathcal{CO}(\mathbf{X})_{\sigma}$, $\mathcal{CO}_{\delta}(\mathbf{X})=\mathcal{CO}(\mathbf{X})_{\delta}$ and $\mathcal{Z}_{\delta}(\mathbf{X})=\mathcal{Z}(\mathbf{X})_{\delta}$. Every member of $\mathcal{CO}_{\delta}(\mathbf{X})$ is called a $c_{\delta}$-\emph{set} in $\mathbf{X}$. Every member of $\mathcal{Z}_{\delta}(\mathbf{X})$ is called a $z_{\delta}$-\emph{set} in $\mathbf{X}$.
\end{definition}

\begin{definition}
	\label{s1:d17} Let $\mathbf{X}=\langle X, \tau\rangle$ be a topological space. 
	\begin{enumerate}
		\item[(i)] (Cf., e.g., \cite{lr} and \cite{lorch}.) The \emph{Baire topology} of $\mathbf{X}$ is the topology $\tau_z$ on $X$ such that $\mathcal{Z}(\mathbf{X})$ is a base of $\langle X, \tau_z\rangle$. We denote the space $\langle X, \tau_z\rangle$ by $(\mathbf{X})_z$.
		\item[(ii)] (Cf. \cite{lr} and \cite{kow}.)  The $G_{\delta}$-\emph{topology} of $\mathbf{X}$  is the topology $\tau_{\delta}$ on $X$ such  that $\mathcal{G}_{\delta}(\mathbf{X})$ is a base of $\langle X, \tau_{\delta}\rangle$. We denote the space $\langle X, \tau_{\delta}\rangle$ by $(\mathbf{X})_{\delta}$.
	\end{enumerate}
\end{definition}

In the following theorem, we recall several facts established in \cite{kow}.

\begin{theorem}
	\label{s1:t18}
	$($Cf. \cite{kow}.$)$ $[\mathbf{ZF}]$
	\begin{enumerate}
		\item[(a)] For every topological space $\mathbf{X}$, the following inclusions hold: $$\mathcal{CO}(\mathbf{X})\subseteq\mathcal{CO}_{\delta}(\mathbf{X})\subseteq\mathcal{Z}(\mathbf{X})\subseteq\mathcal{G}_{\delta}(\mathbf{X}).$$
		\item[(b)] $\mathbf{CMC}$ implies the following: for every topological space $\mathbf{X}$, the families $\mathcal{CO}_{\delta}(\mathbf{X})$, $\mathcal{Z}(\mathbf{X})$ and $\mathcal{G}_{\delta}(\mathbf{X})$ are closed under countable intersections, and $(\mathbf{X})_{\delta}$ is a $P$-space.
		\item[(c)] Each of the following sentences (i)--(iii) implies $\mathbf{CMC}_{\omega}$:
		\begin{enumerate}
			\item[(i)] For every Tychonoff space $\mathbf{X}$ at least one of the families $\mathcal{Z}(\mathbf{X})$ and $\mathcal{G}_{\delta}(\mathbf{X})$ is closed under countable intersections.
			\item[(ii)] For every strongly zero-dimensional $T_1$-space $\mathbf{X}$, $\mathcal{CO}_{\delta}(\mathbf{X})$ is closed under countable intersections.
			\item[(iii)] For every Tychonoff space $\mathbf{X}$, $(\mathbf{X})_{\delta}$ is a $P$-space.
		\end{enumerate}
	\end{enumerate} 		
\end{theorem}

A lot of basic facts about $P$-spaces in the absence of the Axiom of Choice are given in \cite{kow}. As in \cite{kow}, for a topological space $\mathbf{X}$, we pay special attention to the subsets $A(\mathbf{X})$ and $U_{\aleph_0}(\mathbf{X})$ of $C(\mathbf{X})$, defined as follows:

\begin{definition} 
	\label{s1:d19}
	\begin{enumerate}
		\item[(i)] (Cf. \cite{kow}.)
		$U_{\aleph_0}(\mathbf{X})=\nonumber\\
		\{f\in C(\mathbf{X}): (\forall \varepsilon >0)(\exists \mathcal{A}\in [\mathcal{CO}(\mathbf{X})]^{\leq\omega})(\bigcup\mathcal{A}=X\wedge (\forall A\in\mathcal{A})\osc_A(f)\leq\varepsilon)\}.$
		\item[(ii)] $U_{\aleph_0}^{\ast}(\mathbf{X})=U_{\aleph_0}(\mathbf{X})\cap C^{\ast}(\mathbf{X})$.
		\item[(iii)] (Cf. \cite{bk} and \cite{kow}.) $A(\mathbf{X})=\{f\in C(\mathbf{X}): (\forall U\in\tau_{nat})f^{-1}[U]\in \mathcal{CO}_{\sigma}(\mathbf{X})\}$. 
	\end{enumerate}
\end{definition}

\begin{remark}
	\label{s1:r20}
	One can easily verify in $\mathbf{ZF}$ that, for every non-empty topological space $\mathbf{X}$, $U_{\aleph_0}(\mathbf{X})$, $U_{\aleph_0}^{\ast}(\mathbf{X})$ and $A(\mathbf{X})$ are subrings of the ring $C(\mathbf{X})$.  
	
	It was noticed in \cite{kow} that, given a topological space $\mathbf{X}$, the set $U_{\aleph_0}(\mathbf{X})$ is closed in $C_{u}(\mathbf{X})$. Since $C^{\ast}(\mathbf{X})$ is also closed in $C_{u}(\mathbf{X})$, it follows that $U_{\aleph_0}^{\ast}(\mathbf{X})$ is closed in  $C_{u}(\mathbf{X})$.
\end{remark}

In the following theorem, we summarize several results established in \cite{kow}.

\begin{theorem}
	\label{s1:t21}
	$($Cf. \cite{kow}.$)$ $[\mathbf{ZF}]$ For every topological space $\mathbf{X}$, the following conditions are satisfied:
	\begin{enumerate}
		\item[(a)] If $\mathbf{X}$ is completely regular and $A(\mathbf{X})=C(\mathbf{X})$, then $\mathbf{X}$ is strongly zero-dimensional.
		\item[(b)] If $\mathbf{X}$ is a $P$-space, then $C(\mathbf{X}, \mathbb{R}_{disc})=C(\mathbf{X})$.
		\item[(c)]  $\mathbf{CMC}$ implies the following:
		\begin{enumerate}
			\item[(i)] for every $f\in U_{\aleph_0}(\mathbf{X})$, the zero-set $Z(f)$ is a countable intersection of clopen sets of $\mathbf{X}$; 
			\item[(ii)] $A(\mathbf{X})=U_{\aleph_0}(\mathbf{X})=\cl_{C_{u}(\mathbf{X})}(C(\mathbf{X},\mathbb{R}_{disc}))$;
			\item[(iii)] if $\mathbf{X}$ is completely regular, then $\mathbf{X}$ is strongly zero-dimensional if and only if $A(\mathbf{X})=C(\mathbf{X})$;
			\item[(iv)] if $\mathbf{X}$ is completely regular, then $\mathbf{X}$ is strongly zero-dimensional if and only if $U_{\aleph_0}(\mathbf{X})=C(\mathbf{X})$;
			\item[(v)] if $\mathbf{X}$ is completely regular and $C(\mathbf{X}, \mathbb{R}_{disc})=C(\mathbf{X})$, then $\mathbf{X}$ is a $P$-space;
			\item[(vi)] if $\mathbf{X}$ is zero-dimensional, then $\mathbf{X}$ is a $P$-space if and only if the following equalities hold: $C(\mathbf{X})=U_{\aleph_0}(\mathbf{X})=C(\mathbf{X}, \mathbb{R}_{disc})$.
		\end{enumerate}
	\end{enumerate}
\end{theorem}

That a regular $P$-space may fail to be zero-dimensional in $\mathbf{ZF}$ was noticed in \cite{kow}.

\subsection{Special compactifications in $\mathbf{ZF}$, filters, ultrafilters and maximal ideals}
\label{s1.5}

In what follows, for a topological space $\mathbf{X}$, we denote by $\mathcal{E}(\mathbf{X})$ the set of all $\mathcal{F}\subseteq C(\mathbf{X})$ such that the evaluation map $e_{\mathcal{F}}$ is a homeomorphic embedding. Let $\mathcal{E}^{\ast}(\mathbf{X})=\{\mathcal{F}\in\mathcal{E}(\mathbf{X}): \mathcal{F}\subseteq C^{\ast}(\mathbf{X})\}$.  Clearly, a non-empty topological space $\mathbf{X}$ is Tychonoff if and only if $\mathcal{E}^{\ast}(\mathbf{X})\neq\emptyset$. 

\begin{definition}
	\label{s1:d22}
	For $\mathcal{F}\in\mathcal{E}(\mathbf{X})$, the extension $\langle \cl_{\mathbb{R}^{\mathcal{F}}}(e_{\mathcal{F}}[X]), e_{\mathcal{F}}\rangle$ of $\mathbf{X}$ is denoted by $e_{\mathcal{F}}\mathbf{X}$ and called the \emph{extension generated by} $\mathcal{F}$. If $e_{\mathcal{F}}$ is compact, we call it the \emph{compactification of} $\mathbf{X}$ \emph{generated by} $\mathcal{F}$.
\end{definition}

It is well known that the following useful theorem holds in $\mathbf{ZFC}$; however, we remark that its most elementary proof is also valid in $\mathbf{ZF}$.

\begin{theorem}
	\label{s1:t23}
	$[\mathbf{ZF}]$
	Let $\mathbf{X}$ be a non-empty Tychonoff space and let $\mathcal{F},\mathcal{G}\in\mathcal{E}( \mathbf{X})$. Then $e_{\mathcal{F}}\mathbf{X}\approx e_{\mathcal{G}}\mathbf{X}$ if and only if $\mathcal{F}\subseteq C_{e_{\mathcal{G}}}(\mathbf{X})$ and $\mathcal{G}\subseteq C_{e_{\mathcal{F}}}(\mathbf{X})$.
\end{theorem}

For a Tychonoff space $\mathbf{X}$, we denote the extension $e_{C^{\ast}(\mathbf{X})}\mathbf{X}$ by $\beta^{f}\mathbf{X}$ and, for simplicity, we put $e_{\beta^f}=e_{C^{\ast}(\mathbf{X})}$.

In $\mathbf{ZF}$, \v{C}ech-Stone compactifications, maximal Tychonoff compactifications and Banaschewski compactifications can be defined as follows:

\begin{definition}
	\label{s1:d24}
	\begin{enumerate}
		\item[(i)] A \emph{\v{C}ech-Stone compactification} of a Hausdorff space $\mathbf{X}$ is a Hausdorff
		compactification $\beta \mathbf{X}$ of $\mathbf{X}$ such that, for
		every compact Hausdorff space $\mathbf{K}$ and for each continuous mapping $%
		f:\mathbf{X}\rightarrow \mathbf{K}$, there exists a continuous mapping $\tilde{f}:\beta
		\mathbf{X}\rightarrow \mathbf{K}$ such that $f=\tilde{f}\circ \beta $.
		
		\item[(ii)]  A \emph{Banaschewski compactification} of a zero-dimensional $T_1$-space $\mathbf{X}$ is a zero-dimen\-sional
		$T_1$-compactification $ \beta_0 \mathbf{X} $ of $\mathbf{X}$ such that, for
		every compact zero-dimensional $T_1$-space $\mathbf{K}$ and for each continuous mapping $%
		f:\mathbf{X}\rightarrow \mathbf{K}$, there exists a continuous mapping $\tilde{f}:\beta_0
		\mathbf{X}\rightarrow \mathbf{K}$ such that $f=\tilde{f}\circ \beta_0 $.
		
		\item[(iii)] (Cf. \cite{kw1}.) If $\mathbf{X}$ is a Tychonoff space and its extension $\beta^{f}\mathbf{X}$ is compact, then the compactification $\beta^{f}\mathbf{X}$ is called the \emph{maximal functional compactification} of $\mathbf{X}$.
		
		\item[(iv)] A \emph{maximal Tychonoff compactification} of a Tychonoff space $\mathbf{X}$ is a Tychonoff compactification $\beta_{T}\mathbf{X}$ of $\mathbf{X}$ such that, for every compact Tychonoff space $\mathbf{K}$ and every $f\in C(\mathbf{X},\mathbf{K})$, there exists $\tilde{f}\in C(\beta_T\mathbf{X})$ such that $\tilde{f}\circ\beta_T=f$.
	\end{enumerate}
\end{definition}

It holds in $\mathbf{ZF}$ that a Hausdorff space $\mathbf{X}$ can have at most one (up to the equivalence $\thickapprox$) \v{C}ech-Stone compactification but even a Tychonoff space may fail to have its \v{C}ech-Stone compactification (see, e.g.,\cite{kw1}).  Similarly, in $\mathbf{ZF}$, a zero-dimensional $T_1$-space can have at most one (up to the equivalence $\thickapprox$) Banaschewski compactification but may fail to have its Banaschewski compactification. For instance, if $\mathcal{M}$ is any model of $\mathbf{ZF}$ in which there are non-compact Cantor cubes, every non-compact Cantor cube in $\mathcal{M}$ fails to have its Banaschewski compactification in $\mathcal{M}$ (see \cite[Proposition 3.17]{kw1}).

Definition \ref{s1:d24}(iv) was not stated explicitly in \cite{kw1}. By using \cite[Proposition 3.6]{kw1} and Theorem \ref{s1:t10}, one can show that the following theorem holds in $\mathbf{ZF}$:

\begin{theorem}
	\label{s1:t25}
	$[\mathbf{ZF}]$ Let $\mathbf{X}$ be a Tychonoff space. Then a maximal Tychonoff compactification of $\mathbf{X}$ exists if and only if the extension $\beta^{f}\mathbf{X}$ is compact. Furthermore, if $\beta^{f}\mathbf{X}$ is compact, then $\beta_{T}\mathbf{X}\approx\beta^{f}\mathbf{X}$ and, for every Tychonoff compactification $\gamma\mathbf{X}$ of $\mathbf{X}$, it holds that $\gamma\mathbf{X}\leq\beta_{T}\mathbf{X}$.
\end{theorem} 

One of the most familiar types of Hausdorff compactifications in $\mathbf{ZFC}$ are compactifications of Wallman type. They are constructed as spaces of ultrafilters of Wallman bases. We use the same slight modification of the concept of a Wallman base, as in \cite{pw} and \cite{kw1}. Let us recall it in the following definition.

\begin{definition}
	\label{s1:d26}
	A family $\mathcal{C}$ of subsets of a topological space $\mathbf{X}$ is called a \emph{normal} or \emph{Wallman base} for $\mathbf{X}$ if $\mathcal{C}$ satisfies the following conditions:
	\begin{enumerate} 
		\item[(i)] $\mathcal{C}$ is a base for closed sets of $\mathbf{X}$;
		\item[(ii)] if $C_1, C_2\in\mathcal{C}$, then $C_1\cap C_2\in\mathcal{C}$ and $C_1\cup C_2\in\mathcal{C}$;
		\item[(iii)] for each set $A\subseteq X$ and for each $x\in X\setminus A$, if $A$ is a singleton or $A$ is closed in $\mathbf{X}$, then there exists $C\in\mathcal{C}$ such that $x\in C\subseteq X\setminus A$;
		\item[(iv)] if $A_1, A_2\in\mathcal{C}$ and $A_1\cap A_2=\emptyset$, then there exist $C_1, C_2\in\mathcal{C}$ with $A_1\cap C_1=A_2\cap C_2=\emptyset$  and $C_1\cup C_2=X$.
	\end{enumerate}
\end{definition}

\begin{definition} 
	\label{s1:d27}
	Let $X$ be a set and let $\mathcal{R}$ be a non-empty family of subsets of $X$ such that $\mathcal{R}$ is closed under finite intersections. Let $\mathcal{F}$ be a filter in $\mathcal{R}$.
	\begin{enumerate}
		\item[(i)] $\mathcal{F}$ is called \emph{fixed} if $\bigcap\mathcal{F}\neq\emptyset$; otherwise, $\mathcal{F}$ is called \emph{free}.
		
		\item[(ii)] $\mathcal{F}$ has the \emph{countable intersection property} (c.i.p. in abbreviation) if, for every non-empty countable subcollection $\mathcal{H}$ of $\mathcal{F}$, $\bigcap\mathcal{H}\neq\emptyset$. 
		
		\item[(iii)] $\mathcal{F}$ is an \emph{ultrafilter} in $\mathcal{R}$ if, for every $A\in\mathcal{R}\setminus\mathcal{F}$, the family $\mathcal{F}\cup\{A\}$ is not contained in a filter in $\mathcal{R}$.
	\end{enumerate}
\end{definition}

\begin{definition}
	\label{s1:d28}
	Let $\mathbf{X}$ be a topological space.
	\begin{enumerate}
		\item[(i)]  Filters (respectively, ultrafilters) in $\mathcal{Z}(\mathbf{X})$ are called $z$-\emph{filters} (respectively, $z$-\emph{ultrafil\-ters}) in $\mathbf{X}$. 
		\item[(ii)] Filters (respectively, ultrafilters) in $\mathcal{CO}(\mathbf{X})$ are called \emph{clopen filters} (respectively, \emph{clopen ultrafilters}) in $\mathbf{X}$. 
	\end{enumerate}
\end{definition}

For Wallman-type extensions, we use the same notation as in \cite{pw} and \cite{kw1}. Namely, suppose that $\mathcal{C}$  is a Wallman base for a topological space $\mathbf{X}$. We denote by $\mathcal{W}(X, \mathcal{C})$ the set of all ultrafilters in $\mathcal{C}$. For $A\in\mathcal{C}$, we put
$$[A]_{\mathcal{C}}=\{p\in\mathcal{W}(X, \mathcal{C}): A\in p\}.$$
The  \emph{Wallman space} of $\mathbf{X}$ corresponding to $\mathcal{C}$ is denoted by $\mathcal{W}(X, \mathcal{C})$ and it is the set $\mathcal{W}(X, \mathcal{C})$ equipped with the topology having the collection $\{ [A]_{\mathcal{C}}: A\in\mathcal{C}\}$ as a base for closed sets. The canonical embedding of $\mathbf{X}$ into $\mathcal{W}(X, \mathcal{C})$ is the mapping $h_{\mathcal{C}}: X\to\mathcal{W}(X, \mathcal{C})$ defined as follows:
$$(\forall x\in X) \text{ } h_{\mathcal{C}}(x)=\{A\in \mathcal{C}: x\in A\}.$$ 
The pair $\langle\mathcal{W}(X, \mathcal{C}), h_{\mathcal{C}}\rangle$ is called \emph{the Wallman extension} of $\mathbf{X}$ corresponding to $\mathcal{C}$ and, for simplicity, this extension is denoted also by $\mathcal{W}(X, \mathcal{C})$. In the case when $\mathbf{X}=X_{disc}$, the Wallman space $\mathcal{W}(X, \mathcal{P}(X))$ corresponding to the power set $\mathcal{P}(X)$ is usually called the \emph{Stone space} of $X$.

In the following theorem, applying Theorem \ref{s1:t25}, we slightly reformulate several results established in  \cite[Section 3]{kw1}.

\begin{theorem}
	\label{s1:t29}
	$($Cf. \cite[Section 3]{kw1}.$)$ $[\mathbf{ZF}]$ Let $\mathbf{X}$ be a non-empty Tychonoff space. 
	\begin{enumerate}
		\item[(i)] If $\beta_{T}\mathbf{X}$ exists, then $\mathcal{W}(X, \mathcal{Z}(\mathbf{X}))$ is a Hausdorff compactification of $\mathbf{X}$ such that $\mathcal{W}(X, \mathcal{Z}(\mathbf{X}))\approx\beta_{T}\mathbf{X}$.
		\item[(ii)] $\beta_{T}\mathbf{X}$ exists if and only if $\mathcal{W}(X, \mathcal{Z}(\mathbf{X}))$ is compact.
		\item[(iii)] If $\beta\mathbf{X}$ exists, then $\beta_{T}\mathbf{X}$ also exists and $\beta_{T}\mathbf{X}\leq\beta\mathbf{X}$.
		\item[(iv)] If $\beta\mathbf{X}$ exists and is completely regular, then $\beta_{T}\mathbf{X}$ exists and $\beta_{T}\mathbf{X}\approx\beta\mathbf{X}$.
	\end{enumerate}
\end{theorem}

The following theorem can be easily deduced from \cite[Theorem 3.16]{kw1}.

\begin{theorem} 
	\label{s1:t30}
	The following conditions are all equivalent:
	\begin{enumerate}
		\item[(i)] $\mathbf{BPI}$;
		\item[(ii)] every Tychonoff-compact space is compact;
		\item[(iii)] every Cantor-compact space is compact;
		\item[(iv)] for every Tychonoff space $\mathbf{X}$, $\beta_{T}\mathbf{X}$ exists;
		\item[(v)] every Tychonoff space $\mathbf{X}$ has a Hausdorff compactification $\alpha\mathbf{X}$ such that every clopen set of $\mathbf{X}$ has a clopen closure in $\alpha\mathbf{X}$;
		\item[(vi)] every zero-dimensional $T_1$-space $\mathbf{X}$ has a Hausdorff compactification $\alpha\mathbf{X}$ such that every clopen set of $\mathbf{X}$ has a clopen closure in $\alpha\mathbf{X}$.
	\end{enumerate}
\end{theorem}

The following theorem was of significant importance in \cite{kw1} and we apply it in Section 2.

\begin{theorem}
	\label{s1:t31}
	$($Cf. \cite[Theorem 3.1]{kw1}.$)$ $[\mathbf{ZF}]$ Let $\mathbf{X}$ be a topological space and let $\mathcal{F}\in\mathcal{E}(\mathbf{X})$. Then $e_{\mathcal{F}}\mathbf{X}$ is compact if and only if there exists a  compactification $\alpha\mathbf{X}$ of $\mathbf{X}$ such that $\mathcal{F}\subseteq C_{\alpha}(\mathbf{X})$.
\end{theorem}

\begin{definition}
	\label{s1:d32}
	Suppose that $\mathbf{X}=\langle X, \tau_X\rangle$ is a non-compact locally compact Hausdorff space. For an element  $\infty\notin X$, let $X(\infty)=X\cup\{\infty\}$ and $\mathbf{X}(\infty)=\langle X(\infty), \tau_{X(\infty)}\rangle$ where 
	$$\tau_{X(\infty)}=\tau_X\cup\{ V\subseteq X(\infty): X\setminus V \text{ is compact in } \mathbf{X}\}.$$
	Then $\mathbf{X}(\infty)$ is called the \emph{Alexandroff compactification} of $\mathbf{X}$ (or, equivalently, the one-point Hausdorff compactification of $\mathbf{X}$).
\end{definition}

\begin{definition}
	\label{s1:d33}
	Let $R$ be a given commutative ring. 
	\begin{enumerate}
		\item[(i)] $\Max(R)$ is the set of all maximal ideals of $R$;
		\item[(ii)]  the \emph{hull-kernel topology} $\tau_{hk}$ on $\Max(R)$  is the topology whose base is the family $\{\{M\in\Max(R): r\notin M\}: r\in R\}$; 
		\item[(iii)] $\BMax(R)=\langle \Max(R), \tau_{hk}\rangle$.
	\end{enumerate}
\end{definition}

For a non-empty Tychonoff space $\mathbf{X}$, the \emph{canonical embedding} of $\mathbf{X}$ into $\BMax(C(\mathbf{X}))$ is the mapping $h_{max}$ defined as follows:
$$(\forall x\in X) \text{ } h_{max}(x)=\{f\in C(\mathbf{X}): f(x)=0\}.$$
The ordered pair $\langle\BMax(C(\mathbf{X})), h_{max}\rangle$ is an extension of $\mathbf{X}$ denoted, for simplicity, by $\BMax(C(\mathbf{X}))$. 

A well-known result of $\mathbf{ZFC}$ asserts that, for every non-empty Tychonoff space $\mathbf{X}$, the extension $\BMax(C(\mathbf{X}))$ is a Hausdorff compactification of $\mathbf{X}$ equivalent to the \v Cech-Stone compactification of $\mathbf{X}$ (see \cite[Chapter 7.11]{gj}). Since, in view of Theorems \ref{s1:t29} and \ref{s1:t30}, this cannot be a result of $\mathbf{ZF}$, let us state the following theorems concerning $\BMax(C(\mathbf{X}))$ in $\mathbf{ZF}$. 

\begin{theorem}
	\label{s1:t34}
	$[\mathbf{ZF}]$
	For every non-empty Tychonoff space $\mathbf{X}$, the following conditions are satisfied:
	\begin{enumerate}
		\item[(i)] the extensions $\mathcal{W}(X, \mathcal{Z}(\mathbf{X}))$ and $\BMax(C(\mathbf{X}))$ of $\mathbf{X}$ are equivalent;
		\item[(ii)] if $\beta_T\mathbf{X}$ exists, then $\BMax(C(\mathbf{X}))$ is a Hausdorff compactification of $\mathbf{X}$ such that $\beta_T\mathbf{X}\approx \BMax(C(\mathbf{X}))$;
		\item[(iii)] $\BMax(C(\mathbf{X}))$ is compact if and only if $\beta_T\mathbf{X}$ exists.
	\end{enumerate}
\end{theorem}
\begin{proof}
	(i) Given a non-empty Tychonoff space $\mathbf{X}$, we define a mapping $\psi_{max}$ from $\mathcal{W}(X,\mathcal{Z}(\mathbf{X}))$ onto $\Max(C(\mathbf{X}))$ as follows:
	$$(\forall p\in\mathcal{W}(X, \mathcal{Z}(\mathbf{X}))) \text{ }  \psi_{max}(p)=\{f\in C(\mathbf{X}): Z(f)\in p\}.$$
	In much the same way, as in \cite[Chapter 7.11]{gj}, one can check that $\psi_{max}$ is a homeomorphism of $\mathcal{W}(X,\mathcal{Z}(\mathbf{X}))$ onto $\BMax(C(\mathbf{X}))$ such that $\psi_{max}\circ h_{\mathcal{Z}(\mathbf{X})}=h_{max}$. Hence (i) holds, That (ii) and (iii) also hold follows from (i) and Theorem \ref{s1:t29}.
\end{proof}

\begin{theorem}
	\label{s1:t35}
	$[\mathbf{ZF}]$ $\mathbf{BPI}$ is equivalent to the following statement: For every non-empty Tychonoff space $\mathbf{X}$, the space $\BMax(C(\mathbf{X}))$ is compact.
\end{theorem}
\begin{proof}
	This is a consequence of Theorems \ref{s1:t30} and \ref{s1:t34}.
\end{proof}

All topological notions, if not defined here, can be found, for instance, in \cite{en}, \cite{gj}, \cite{nag} and \cite{w}.
\subsection{The content of the article in brief} 
\label{s1.6}

As we have already said in Section \ref{s1.1}, we work in $\mathbf{ZF}$, thus, in the absence of the Axiom of Choice. In Section \ref{s2.1}, we give necessary and sufficient conditions for a zero-dimensional $T_1$-space to have its Banaschewski compactification in the sense of Definition \ref{s1:d24}(ii) (see Theorem \ref{s2:t1}). We notice that the statement ``Every zero-dimensional $T_1$-space has its Banaschewski compactification'' is an equivalent of $\mathbf{BPI}$ (see Theorem \ref{s2:t4}). A Cantor cube $\mathbf{2}^J$ is compact if and only if its Banaschewski compactification exists (see Theorem \ref{s2:t5}). For a strongly zero-dimensional $T_1$-space $\mathbf{X}$, the Banaschewski compactification of $\mathbf{X}$ exists if and only if the maximal Tychonoff compactification of $\mathbf{X}$ exists (see Theorem \ref{s2:t6}).  In Section \ref{s2.2}, we show that it holds in $\mathbf{ZF}$ that a completely regular space $\mathbf{X}$ is strongly zero-dimensional if and only if $U_{\aleph_0}^{\ast}(\mathbf{X})=C^{\ast}(\mathbf{X})$ (see Theorem \ref{s2:t13}). Furthermore,  we prove in $\mathbf{ZF}$ that, for a zero-dimensional $T_1$- space $\mathbf{X}$, $\beta_0\mathbf{X}$ exists if and only if the extension of $\mathbf{X}$ generated by $U_{\aleph_0}^{\ast}(\mathbf{X})$ is compact (see Theorem \ref{s2:t16}); moreover, if $\beta_0\mathbf{X}$ exists, then it is equivalent to the compactification of $\mathbf{X}$ generated by $U_{\aleph_0}^{\ast}(\mathbf{X})$, and $C(\beta_0\mathbf{X})$ is the set of continuous extensions over $\beta_0\mathbf{X}$ of functions from $U_{\aleph_0}^{\ast}(\mathbf{X})$ (see Corollary \ref{s2:c17}).

In Section \ref{s2.3}, for every non-empty zero-dimensional $T_1$-space, we prove in $\mathbf{ZF}$ that the Banaschewski compactification $\beta_0\mathbf{X}$ exists if and only if the space $\BMax(U_{\aleph_0}^{\ast}(\mathbf{X}))$ is compact, and if $\beta_0\mathbf{X}$ exists, then it is homeomorphic with $\BMax(U_{\aleph_0}^{\ast}(\mathbf{X}))$ (see Theorem \ref{s2:t24}) and, moreover,  $\mathbf{CMC}$ implies that  $\beta_0\mathbf{X}$ and $\BMax(U_{\aleph_0}(\mathbf{X}))$ are homeomorphic (see Theorem \ref{s2:t22})).  We also notice that, for every non-empty Tychonoff space $\mathbf{X}$, it holds in $\mathbf{ZF}$ that $\beta_T\mathbf{X}$ exists if and only if $\BMax(C^{\ast}(\mathbf{X}))$ is compact (see Remark \ref{s2:r25}). This is a special case of a more general result shown in Remark \ref{s2:r25}. We deduce that, for every non-empty zero-dimensional $T_1$-space $\mathbf{X}$, it  holds in $\mathbf{ZF}$ that if $\BMax(U_{\aleph_0}(\mathbf{X}))$ is compact, then $\beta_0\mathbf{X}$ exists (see Remark \ref{s2:r26}). Theorem \ref{s2:t27} gives several equivalents of $\mathbf{BPI}$ in $\mathbf{ZF}$ in terms of spaces of maximal ideals with their hull-kernel topologies.

We show in Theorem \ref{s2:t29} that it holds in $\mathbf{ZF}$ that, given a Tychonoff space $\mathbf{Y}$ and a non-empty zero-dimensional $T_1$-space $\mathbf{X}$ for which $\beta_0\mathbf{X}$ exists, if the rings $C(\mathbf{Y})$ and  $U_{\aleph_0}^{\ast}(\mathbf{X})$ are isomorphic, then $\beta_T\mathbf{Y}$ exists and is homeomorphic with $\beta_0\mathbf{X}$; furthermore, if the rings $C(\mathbf{Y})$ and $U_{\aleph_0}(\mathbf{X})$ are isomorphic, then $\mathbf{CMC}$ implies that $\beta_T\mathbf{Y}$ exists and is homeomorphic with $\beta_0\mathbf{X}$ (see Theorem \ref{s2:t29}).

Section \ref{s3} concerns Baire and $G_{\delta}$-topologies (see Definition \ref{s1:d17}). Among other facts, it is noticed that $\mathbf{CMC}$ implies that the Baire topology and the $G_{\delta}$-topology of a completely regular space $\mathbf{X}$ coincide (see Proposition \ref{s3:p1}(v)), and $(\mathbf{X})_z$ is a $P$-space (see Corollary \ref{s3:c2}). The statement ``For every Tychonoff space $\mathbf{X}$, the space $(\mathbf{X})_z$ is a $P$-space'' is unprovable in $\mathbf{ZF}$ for it implies $\mathbf{CMC}_{\omega}$ (see Proposition \ref{s3:p3}). If $\mathbf{X}$ is a completely regular space whose every singleton is of type $G_{\delta}$ in $\mathbf{X}$, then $\mathbf{CMC}$ implies that $(\mathbf{X})_z$ is discrete (see Proposition \ref{s3:p6} and Definition \ref{s3:d5}).

In Section \ref{s4}, for a Hausdorff space $\mathbf{E}$, we check whether some basic statements on $\mathbf{E}$-compact spaces and Hewitt $\mathbf{E}$-compactifications, known to be true in $\mathbf{ZFC}$,  are also true in $\mathbf{ZF}$. We introduce a new notion of a compactly $\mathbf{E}$-Urysohn space (see Definition \ref{s4:d11}). Given a non-compact locally compact, compactly $\mathbf{E}$-Urysohn space $\mathbf{E}$, we describe a useful construction of the Hewitt $\mathbf{E}$-compact extension of a space $\mathbf{X}$ for which $C^{\ast}(\mathbf{X}, \mathbf{E})$ is a homeomorphic embedding (see Theorem \ref{s4:t14}). We apply this construction to a characterization of $\mathbf{E}$-compactness (see Corollary \ref{s4:c15}), in particular, to  $\mathbb{N}$-compactness (see Theorem \ref{s4:t16}(ii)).

In Section \ref{s5}, we notice that, for a Tychonoff space $\mathbf{E}$ such that $\mathbb{R}$ is $\mathbf{E}$-compact, it holds (in $\mathbf{ZF}$) that an $\mathbf{E}$-completely regular space $\mathbf{X}$ is strongly zero-dimensional if and only if the Hewitt $\mathbf{E}$-compact extension of $\mathbf{X}$ is strongly zero-dimensional (see Theorem \ref{s5:t2}). We show that if $\mathbf{E}$ is a Hausdorff space such that $\mathbb{R}_{disc}$ is $\mathbf{E}$-compact, then $\mathbf{CMC}$ implies that, for every $\mathbf{E}$-completely regular space $\mathbf{X}$, the rings $U_{\aleph_0}(\mathbf{X})$ and $U_{\aleph_0}(v_{\mathbf{E}}\mathbf{X})$ are isomorphic (see Theorem \ref{s5:t4}(i)), and, moreover, if $\mathbf{X}$ is a $P$-space, so is $v_{\mathbf{E}}\mathbf{X}$ (see Theorem \ref{s5:t7}). A characterization of strong zero-dimensionality in terms of rings of functions is also given in $\mathbf{ZF+CMC}$ (see Theorems \ref{s5:t4}(ii) and \ref{s5:t6}).

In Section \ref{s6}, we give a detailed $\mathbf{ZF}$-proof of Chew's theorem asserting that a zero-dimensional $T_1$-space is $\mathbb{N}$-compact if and only if every clopen ultrafilter in $\mathbf{X}$ with the countable intersection property is fixed (see Theorem \ref{s6:t1}). Chew's original proof was not in $\mathbf{ZF}$. We prove that it holds in $\mathbf{ZF+CMC}$ that a zero-dimensional $T_1$-space is $\mathbb{N}$-compact if and only if every ultrafilter in $\mathcal{CO}_{\delta}(\mathbf{X})$ with the countable intersection property is fixed (see Theorem \ref{s6:t5}). To get (in $\mathbf{ZF}$) a relevant characterization of realcompactness in terms of $z$-ultrafilters, we introduce the notion of a $z$-filter with the weak countable intersection property (see Definition \ref{s6:d8}). We prove that it holds in $\mathbf{ZF}$ that a Tychonoff space $\mathbf{X}$ is realcompact if and only if every $z$-ultrafilter in $\mathbf{X}$ with the weak countable intersection property is fixed (see Theorem \ref{s6:t10}). We deduce that the statement ``A Tychonoff space $\mathbf{X}$ is realcompact if and only if every $z$-ultrafilter in $\mathbf{X}$ with the countable intersection property is fixed'', well-known to be true in $\mathbf{ZFC}$, is also true in $\mathbf{ZF+CMC}$. This implies that it holds in $\mathbf{ZF+CMC}$ that if a Tychonoff space $\mathbf{X}$ is a countable union of $z$-embedded realcompact subspaces, then $\mathbf{X}$ is realcompact (see Theorem \ref{s6:t12}). We introduce the notion of a $c_{\delta}$-embedded subspace (see Definition \ref{s6:d13}) and prove that it holds in $\mathbf{ZF+CMC}$ that if a zero-dimensional $T_1$-space is a countable union of its $c_{\delta}$-embedded $\mathbb{N}$-compact subspaces, then $\mathbf{X}$ is $\mathbb{N}$-compact (see Theorem \ref{s6:t16}). Other relevant results and their applications are deduced (see Corollary \ref{s6:c17}, Theorem \ref{s6:t18}, Corollaries \ref{s6:c20}--\ref{s6:c22}). Notions of a $z$-perfect and a $c_{\delta}$-perfect mapping are introduced (see Definition \ref{s6:d24}) to show that it holds in $\mathbf{ZF}$ that if $\mathbf{X}=\langle X, \tau\rangle$ is a realcompact (respectively, $\mathbb{N}$-compact) space, and $\tau^{\ast}$ is a finer than $\tau$ completely regular (respectively, zero-dimensional) topology on $X$ such that $\tau^{\ast}$ is weaker than the Baire topology of $\mathbf{X}$ (respectively, weaker than the topology generated by $\mathcal{CO}_{\delta}(\mathbf{X})$), then the space $\langle X, \tau^{\ast}\rangle$ is also realcompact (respectively, $\mathbb{N}$-compact) (see Corollary \ref{s6:c27}).

In Section \ref{s7}, the concepts of an $r$-extension and an $r_{\mathbb{N}}$-extension of Definition \ref{s7:d1} are used in the following useful characterizations of realcompactness and $\mathbb{N}$-compactness, given in Proposition \ref{s7:p3}: it holds in $\mathbf{ZF}$ that a topological space $\mathbf{X}$ is realcompact (respectively, $\mathbb{N}$-compact) if and only if it has an $r$-extension (respectively, $r_{\mathbb{N}}$-extension). Proposition \ref{s7:p3} is applied to Baire sets (that is, to members of the $\sigma$-field generated by the family of zero-sets) in a realcompact space and to zero-Baire sets  (that is, to members of the $\sigma$-field generated by the family of $c_{\delta}$-sets) in an $\mathbb{N}$-compact space. Namely, by applying Proposition \ref{s7:p3}, we prove that it holds in $\mathbf{ZF+CMC}$ that every Baire subset of a realcompact space is realcompact, and every zero-Baire set in an  $\mathbb{N}$-compact space is also $\mathbb{N}$-compact (see Theorem \ref{s7:t10}).

The aim of Section \ref{s8} is to apply $z$-measures, $c$-measures and Dirac measures (see Definition \ref{s8:d2}) to characterizations of realcompactness and $\mathbb{N}$-compactness. We introduce the notion of a weakly $\sigma$-additive $z$-measure on a topological space (see Definition \ref{s8:d6}). We prove that it holds in $\mathbf{ZF}$ that a Tychonoff space $\mathbf{X}$ is realcompact if and only if every 2-valued weakly $\sigma$-additive $z$-measure on $\mathbf{X}$ is a Dirac $z$-measure (see Theorem \ref{s8:t8}). We also prove that it holds in $\mathbf{ZF}$ that a zero-dimensional $T_1$-space is $\mathbb{N}$-compact if and only if every 2-valued countably additive $c$-measure on this space is a Dirac $c$-measure (see Theorem \ref{s8:t12}). Some consequences of the results are also shown in Section \ref{s8} (see Corollaries \ref{s8:c9} and \ref{s8:c11}, Theorem \ref{s8:t13} and Remark \ref{s8:r14}).

Section \ref{s9} gives a deeper insight into characters and real ideals of the rings $C(\mathbf{X}, \mathbb{R}_{disc})$ and $U_{\aleph_0}(\mathbf{X})$ for a zero-dimensional $T_1$-space $\mathbf{X}$. For a given non-empty $\mathbb{N}$-compact space, we prove that it holds in $\mathbf{ZF}$ that every character of the ring $C(\mathbf{X}, \mathbb{R}_{disc})$ is determined by a unique point (see Theorem \ref{s9:t5}), and it holds in $\mathbf{ZF+CMC}$ that every character of $U_{\aleph_0}(\mathbf{X})$ is determined by a unique point (see Theorem \ref{s9:t6}). Moreover, we prove that, given a zero-dimensional $T_1$-space $\mathbf{X}$, it holds in $\mathbf{ZF}$ that $\mathbf{X}$ is $\mathbb{N}$-compact if and only if every real ideal of $C(\mathbf{X}, \mathbb{R}_{disc})$ is fixed, and it holds is $\mathbf{ZF+CMC}$ that $\mathbf{X}$ is $\mathbb{N}$-compact if and only if every real ideal of $U_{\aleph_0}(\mathbf{X})$ is fixed (see Theorem \ref{s9:t8}).

Our final goal is to prove in Section \ref{s10} that, assuming that $\mathbf{X}$ is a non-empty $\mathbb{N}$-compact space whose Banaschewski compactification exists, it holds in $\mathbf{ZF+CMC}$ that $\mathbf{X}$ is strongly zero-dimensional if and only if there exists a Tychonoff space $\mathbf{Y}$ such that the rings $C(\mathbf{Y})$ and $U_{\aleph_0}(\mathbf{X})$ are isomorphic (see Theorem \ref{s10:t3}).

Since many open problems have arisen while writing this paper, for the convenience of readers, we include a shortlist of open problems in Section \ref{s11}.

\section{Several new remarks on Banaschewski compactifications in $\mathbf{ZF}$}
\label{s2}

\subsection{Banaschewski compactifications via extensions generated by sets of functions}
\label{s2.1}

For topological spaces $\mathbf{X}$ and $\mathbf{E}$, let $\mathcal{E}(\mathbf{X}, \mathbf{E})$ stand for the collection of all sets $\mathcal{F}\subseteq C(\mathbf{X}, \mathbf{E})$ such that the evaluation map $e_{\mathcal{F}}$ is a homeomorphic embedding. 

Suppose $\mathbf{X}$ is a given zero-dimensional $T_1$-space, and $Y$ is a set consisting of at least two points. Then $C^{\ast}(\mathbf{X}, Y_{disc})\in \mathcal{E}(\mathbf{X}, Y_{disc})$. To simplify notation, we put $e_{Y}^{\ast}=e_{C^{\ast}(\mathbf{X}, Y_{disc})}$. In particular, $e^{\ast}_2=e_{C(\mathbf{X}, \mathbf{2})}$  is a homeomorphic embedding of $\mathbf{X}$ into the Cantor cube $\mathbf{2}^{C(\mathbf{X},\mathbf{2})}$; however, in $\mathbf{ZF}$, the Cantor cube  $\mathbf{2}^{C(\mathbf{X},\mathbf{2})}$ and the extension $e^{\ast}_{2}\mathbf{X}$ of $\mathbf{X}$ may fail to be compact. Therefore, it is reasonable to search for conditions under which, for a given zero-dimensional $T_1$-space $\mathbf{X}$, the Banaschewski compactification of $\mathbf{X}$ exists. The theorem given below can be proved by applying the methods shown in \cite{kw1}. Let us sketch a proof of this theorem for completeness.

\begin{theorem}
	\label{s2:t1}
	$[\mathbf{ZF}]$ For every zero-dimensional $T_1$-space $\mathbf{X}$, the following conditions are all equivalent:
	\begin{enumerate}
		\item[(i)] the extension $e^{\ast}_{2}\mathbf{X}$ is compact;
		\item[(ii)] $\beta_0\mathbf{X}$ exists;
		\item[(iii)] there exists a Hausdorff compactification $\alpha\mathbf{X}$ of $\mathbf{X}$ such that $C(\mathbf{X}, \mathbf{2})\subseteq C_{\alpha}(\mathbf{X})$; 
		\item[(iv)] $\mathcal{W}(X, \mathcal{CO}(\mathbf{X}))$ is compact.
	\end{enumerate}
\end{theorem}
\begin{proof}
	Assume (i). Let $\mathbf{K}$ be a compact zero-dimensional $T_1$-space and let $f\in C(\mathbf{X}, \mathbf{K})$. Let $A,B\in \mathcal{CO}(\mathbf{K})$ and $A\cap B=\emptyset$. Choose $g\in C(\mathbf{K},\mathbf{2})$ such that $A\subseteq g^{-1}[\{0\}]$ and $B\subseteq g^{-1}[\{1\}]$. Then $g\circ f\in C(\mathbf{X}, \mathbf{2})\subseteq C_{e^{\ast}_{2}}(\mathbf{X})$, so, by Theorem \ref{s1:t10}, $\cl_{e^{\ast}_{2}\mathbf{X}}(e^{\ast}_{2}[(g\circ f)^{-1}[\{0\}]])\cap \cl_{e^{\ast}_{2}\mathbf{X}}(e^{\ast}_{2}[(g\circ f)^{-1}[\{1\}]])=\emptyset$. Hence,  $\cl_{e^{\ast}_{2}\mathbf{X}}(e^{\ast}_{2}[f^{-1}[A]])\cap \cl_{e^{\ast}_{2}\mathbf{X}}(e^{\ast}_{2}[f^{-1}[B]])=\emptyset$. Thus, by Theorem \ref{s1:t10}, there exists $\tilde{f}\in C(e^{\ast}_{2}\mathbf{X}, \mathbf{K})$ with $\tilde{f}\circ e^{\ast}_{2}=f$. This proves that $e^{\ast}_{2}\mathbf{X}$ is the Banaschewski compactification of $\mathbf{X}$. Hence (i) implies (ii). 
	
	If (ii) holds, then every Hausdorff compactification $\alpha\mathbf{X}$ of $\mathbf{X}$ equivalent to $\beta_0\mathbf{X}$ is such that $C(\mathbf{X}, \mathbf{2})\subseteq C_{\alpha}(\mathbf{X})$. Hence (ii) implies (iii). That (iii) implies (i) follows directly from Theorem \ref{s1:t31}. It is obvious that (iv) implies (iii). To complete the proof, it suffices to show that (i) implies (iv).
	
	Assuming that (i) holds, we mimic and suitably modify the proof that (iii) implies (i) in Theorem 3.10 in \cite{kw1}. Namely, let $\mathcal{A}$ be any filter in $\mathcal{CO}(\mathbf{X})$. There exists $p\in e^{\ast}_{2}\mathbf{X}$ such that $p\in\bigcap\{ \cl_{e^{\ast}_{2}\mathbf{X}}e^{\ast}_{2}[A]: A\in\mathcal{A}\}$. We define
	$$\mathcal{F}=\{ C\in\mathcal{CO}(\mathbf{X}): p\in\cl_{e^{\ast}_{2}\mathbf{X}}(e^{\ast}_{2}[C])\}.$$ 
	Let us leave it as an exercise a relatively simple proof that $\mathcal{F}$ is an ultrafilter in $\mathcal{CO}(\mathbf{X})$ such that $\mathcal{A}\subseteq\mathcal{F}$. This shows that every filter in $\mathcal{CO}(\mathbf{X})$ is contained in an ultrafilter in $\mathcal{CO}(\mathbf{X})$. In consequence,  $\mathcal{W}(X, \mathcal{CO}(\mathbf{X}))$ is compact. Hence (i) implies (iv).
\end{proof}

\begin{corollary}
	\label{s2:c2}
	$[\mathbf{ZF}]$ Let $\mathbf{X}$ be a zero-dimensional $T_1$-space which satisfies conditions (i)-(iv) of Theorem \ref{s2:t1}. Then
	$$\beta_0\mathbf{X}\approx e^{\ast}_{2}\mathbf{X}\approx\mathcal{W}(X, \mathcal{CO}(\mathbf{X})).$$
\end{corollary}

\begin{corollary}
	\label{s2:c3}
	$[\mathbf{ZF}]$ Let $\mathbf{X}$ be a zero-dimensional $T_1$-space for which $\beta_T\mathbf{X}$ exists. Then $\beta_0\mathbf{X}$ also exists and $\beta_0\mathbf{X}\leq\beta_{T}\mathbf{X}$.
\end{corollary}

\begin{theorem}
	\label{s2:t4}
	$[\mathbf{ZF}]$ The following conditions are all equivalent:
	\begin{enumerate}
		\item[(i)] $\mathbf{BPI}$;
		\item[(ii)] for every zero-dimensional $T_1$-space $\mathbf{X}$, $\beta_0\mathbf{X}$ exists;
		\item[(iii)] for every infinite set $J$, the Banaschewski compactification of the Cantor cube $\mathbf{2}^J$ exists;
		\item[(iv)] for every infinite set $X$, the Banaschewski compactification of the discrete space $X_{disc}$ exists.
	\end{enumerate}
\end{theorem}
\begin{proof}
	That (i) implies (ii) follows from the well-known fact of $\mathbf{ZF}$ that $\mathbf{BPI}$ is equivalent to the statement that all Cantor cubes are compact (see Remark \ref{s1:r14}). It is obvious that (ii) implies both (iii) and (iv). 
	
	Suppose that (iii) holds. Let $J$ be an infinite set. By (iii), there exists a compactification $\alpha \mathbf{2}^J$ such that $C(\mathbf{2}^J, \mathbf{2})\subseteq C_{\alpha}(\mathbf{2}^J)$. By the equivalence of (i) and (iii) of \cite[Theorem 3.16]{kw1}, $\mathbf{UFT}$ holds. Hence, in view of Remark \ref{s1:r14}, (iii) implies (i).
	
	Let $X$ be an infinite set. Assuming (iv), we infer from Theorem \ref{s2:t1} that the space $\mathcal{W}(X, \mathcal{P}(X))$ is compact, so every filter in $\mathcal{P}(X)$ is contained in an ultrafilter in $\mathcal{P}(X)$. In consequence, (iv) implies $\mathbf{UFT}$. This, together with Remark \ref{s1:r14}, completes the proof.
\end{proof}

For Cantor cubes, we can state the following theorem:

\begin{theorem}
	\label{s2:t5}
	$[\mathbf{ZF}]$ Let $J$ be an infinite set. Then the following conditions are all equivalent:
	\begin{enumerate}
		\item[(i)] the Cantor cube $\mathbf{2}^J$ is compact;
		\item[(ii)] the maximal Tychonoff compactification of $\mathbf{2}^J$ exists;
		\item[(iii)] the Banaschewski compactification of $\mathbf{2}^J$ exists. 
	\end{enumerate}
\end{theorem}

It is well known that it holds in $\mathbf{ZFC}$ that if $\mathbf{X}$ is a zero-dimensional $T_1$-space,  then $\beta\mathbf{X}\approx\beta_0\mathbf{X}$ if and only if $\mathbf{X}$ is strongly zero-dimensional. Let us try to investigate if this result of $\mathbf{ZFC}$ also holds in $\mathbf{ZF}$.  

\begin{theorem} 
	\label{s2:t6}
	$[\mathbf{ZF}]$ Suppose that $\mathbf{X}$ is a Tychonoff space.  Then the following conditions are satisfied:
	\begin{enumerate}
		\item[(i)] $\mathbf{X}$ is strongly zero-dimensional if and only if the extension $\beta^{f}\mathbf{X}$ is strongly zero-dimensional.
		\item[(ii)] If $\mathbf{X}$ is strongly zero-dimensional and $\beta_{T}\mathbf{X}$ exists, then $\beta_0\mathbf{X}$ exists and $\beta_{T}\mathbf{X}\approx\beta_0\mathbf{X}$. In consequence, if $\mathbf{X}$ is strongly zero-dimensional and $\beta^{f}\mathbf{X}$ is compact, then $\beta_0\mathbf{X}$ exists and  $\beta^{f}\mathbf{X}\approx\beta_0\mathbf{X}$.
		\item[(iii)] If $\mathbf{X}$ is strongly zero-dimensional, then $\beta_0\mathbf{X}$ exists if and only if $\beta_{T}\mathbf{X}$ exists.
	\end{enumerate}
\end{theorem}

\begin{proof}
	That (i) holds can be proved in much the same way, as Theorem 6.2.12 in \cite{en} asserting that, in $\mathbf{ZFC}$, a Tychonoff space is strongly zero-dimensional if and only if its \v{C}ech-Stone compactification is strongly zero-dimensional. We omit details here for they are similar to the ones in the proofs of the forthcoming Theorems \ref{s5:t1} and \ref{s5:t2}.
	
	(ii) Now, assume that $\mathbf{X}$ is a strongly zero-dimensional space for which $\beta_{T}\mathbf{X}$ exists. Then, in the light of (i) and Theorem \ref{s1:t25}, $\beta_{T}\mathbf{X}$ is strongly zero-dimensional. Clearly, if $\mathbf{K}$ is a zero-dimensional $T_1$-space, then $\mathbf{K}$ is Tychonoff, so, for every $f\in C(\mathbf{X}, \mathbf{K})$, there exists $\tilde{f}\in C(\beta_T\mathbf{X}, \mathbf{K})$ with $\tilde{f}\circ\beta_{T}=f$. This shows that $\beta_{T}\mathbf{X}$ is the Banaschewski compactification of $\mathbf{X}$. In view of Theorem \ref{s1:t25}, the last statement of (ii) is also true.
	
	(iii) For the proof of (iii), we assume that $\mathbf{X}$ is strongly zero-dimensional and $\beta_0\mathbf{X}$ exists. Let $\mathbf{K}$ be a Tychonoff space and let $f\in C(\mathbf{X},\mathbf{K})$. Let $Z_1,Z_2\in\mathcal{Z}(\mathbf{K})$ and $Z_1\cap Z_2=\emptyset$. Then $f^{-1}[Z_1], f^{-1}[Z_2]\in \mathcal{Z}(\mathbf{X})$. Since $\mathbf{X}$ is strongly zero-dimensional, there exists $V\in\mathcal{CO}(\mathbf{X})$ such that $f^{-1}[Z_1]\subseteq V\subseteq X\setminus f^{-1}[Z_2]$. Clearly, there exists $U\in\mathcal{CO}(\beta_0\mathbf{X})$ with $U\cap X=V$. This implies that $\cl_{\beta_0\mathbf{X}}(f^{-1}[Z_1])\cap\cl_{\beta_0\mathbf{X}}(f^{-1}[Z_2])=\emptyset$. Hence, by Theorem \ref{s1:t10}, there exists $\tilde{f}\in C(\beta_0\mathbf{X}, \mathbf{K})$ with $\tilde{f}\circ\beta_0=f$. In consequence, $\beta_0\mathbf{X}$ is the maximal Tychonoff compactification of $\mathbf{X}$. This, together with (ii), completes the proof of (iii).
\end{proof}

Unfortunately, at this moment, we are unable to solve the following problem:

\begin{problem}
	\label{s2:q7}
	Is there a model $\mathcal{M}$ of $\mathbf{ZF}$ in which there is a strongly zero-dimensional $T_1$-space $\mathbf{X}$ for which $\beta_0\mathbf{X}$ exists in $\mathcal{M}$ but $\beta\mathbf{X}$ does not exist in $\mathcal{M}$?
\end{problem}

A positive answer to Problem \ref{s2:q7} will be also a positive answer to the following still open problem:

\begin{problem}
	\label{s2:q8}
	Is there a model $\mathcal{M}$ of $\mathbf{ZF}$ in which there is a Tychonoff space $\mathbf{X}$ for which $\beta_T\mathbf{X}$ exists in $\mathcal{M}$ but $\beta\mathbf{X}$ does not exist in $\mathcal{M}$?
\end{problem}

The following open problem is also interesting:

\begin{problem}
	\label{s2:q9}
	Can a Cantor cube fail to be strongly zero-dimensional in a model of $\mathbf{ZF}$?
\end{problem}

It is worth noticing that the following theorem holds in $\mathbf{ZF}$:

\begin{theorem}
	\label{s2:t10}
	$[\mathbf{ZF}]$
	Let $\mathbf{X}$ be a zero-dimensional $T_1$-space and let $Y$ be a set that consists of at least two points. Then $e^{\ast}_2\mathbf{X}\approx e^{\ast}_Y\mathbf{X}$. In consequence, if the Banaschewski compactification of $\mathbf{X}$ exists, then $\beta_0\mathbf{X}\approx e^{\ast}_Y\mathbf{X}$.
\end{theorem}
\begin{proof}
	Let $y_0, y_1$ be a fixed pair of distinct points of $Y$ and let $Z=\{ y_0, y_1\}$. Since $e^{\ast}_2\mathbf{X}\approx e^{\ast}_Z\mathbf{X}$, it suffices to prove that $e^{\ast}_Z\mathbf{X}\approx e^{\ast}_Y\mathbf{X}$.
	
	We notice that $e^{\ast}_Z[X]\subseteq\{t\in Y^{C^{\ast}(\mathbf{X}, Y_{disc})}: (\forall f\in C(\mathbf{X}, Z_{disc})) t(f)\in Z\}$. We can define a continuous surjection $h: e^{\ast}_Y\mathbf{X}\to e^{\ast}_Z\mathbf{X}$ as follows. For every $t\in e_{Y}^{\ast}X$ and every $f\in C(\mathbf{X}, Z_{disc})$, $h(t)(f)= t(f)$. This shows that $e^{\ast}_Z\mathbf{X}\leq e^{\ast}_Y\mathbf{X}$.  
	
	Now, to show that $e^{\ast}_Y\mathbf{X}\leq e^{\ast}_Z\mathbf{X}$, we consider any $f\in C^{\ast}(\mathbf{X}, Y_{disc})$ and any pair $A,B$ of disjoint non-empty subsets of $Y$. We take $g: Y\to Z$ defined as follows: 
	$$
	g(y)=\begin{cases} y_0 &\text{ if } y\in Y\setminus B;\\
		y_1 &\text{ if } y\in B.\end{cases}
	$$
	\noindent Then $g\circ f\in C(\mathbf{X}, Z_{disc})$, which implies that 
	$$\cl_{e^{\ast}_Z\mathbf{X}}e^{\ast}_Z[f^{-1}[A]]\cap \cl_{e^{\ast}_Z\mathbf{X}}e^{\ast}_Z[f^{-1}[B]]=\emptyset.$$
	Hence, by Theorem \ref{s1:t10} and Proposition \ref{s1:p12}, there exists a unique continuous function $\tilde{f}:e^{\ast}_Z\mathbf{X}\rightarrow (f[X])_{disc}$ such that $\tilde{f}\circ e^{\ast}_Z=f$. Let $\mathcal{F}=\{\tilde{f}: f\in C^{\ast}(\mathbf{X}, Y_{disc})\}$. Then $e_{\mathcal{F}}: e^{\ast}_Z\mathbf{X}\to Y_{disc}^{C^{\ast}(\mathbf{X}, Y_{disc})}$ is continuous, $e_{\mathcal{F}}[e^{\ast}_{Z}X]=e^{\ast}_{Y}X$ and $e_{\mathcal{F}}\circ e^{\ast}_Z=e^{\ast}_Y$. Thus, $e^{\ast}_{Y}\mathbf{X}\leq e^{\ast}_{Z}\mathbf{X}$. This completes the proof that $e^{\ast}_Z\mathbf{X}\approx e^{\ast}_Y\mathbf{X}$. 
\end{proof}

\subsection{Strong zero-dimensionality and Banaschewski compactifications in terms of $U_{\aleph_0}^{\ast}(\mathbf{X})$}
\label{s2.2}

Given a zero-dimensional $T_1$-space $\mathbf{X}$, it is obvious that $U_{\aleph_0}^{\ast}(\mathbf{X})\in\mathcal{E}^{\ast}(\mathbf{X})$; thus, is worthwhile to have a deeper look at the extension of $\mathbf{X}$ generated by $U_{\aleph_0}^{\ast}(\mathbf{X})$ and its relationship with the Banaschewski compactification of $\mathbf{X}$. To this aim, let us begin with the following proposition which gives a simpler description of $U_{\aleph_0}^{\ast}(\mathbf{X})$.

\begin{proposition}
	\label{s2:p11}
	For every topological space $\mathbf{X}$, the following equality holds:
	$U_{\aleph_0}^{\ast}(\mathbf{X})=\{ f\in C(\mathbf{X}): (\forall \varepsilon\in (0, +\infty)(\exists \mathcal{A}\in [\mathcal{CO}]^{<\omega}): \bigcup\mathcal{A}=X\wedge(\forall A\in\mathcal{A}) \osc_A(f)\leq\varepsilon\}.$
\end{proposition}
\begin{proof} Given a topological space $\mathbf{X}$,  $f\in U_{\aleph_0}^{\ast}(\mathbf{X})$ and $\varepsilon\in (0, +\infty)$, 
	we fix a disjoint countable clopen cover $\mathcal{A}$ of $\mathbf{X}$ such that, for every $A\in\mathcal{A}$, $\osc_A(f)\leq\frac{\varepsilon}{3}$. Let $a=\inf f[X]$ and $b=\sup  f[X]$. We may assume that $a<b$. Choose $n\in\omega$ such that $\frac{b-a}{n}\leq\frac{\varepsilon}{3}$. For $i\in n+1$,  let $a_i= a+\frac{i(b-a)}{n}$ and, for $i\in n\setminus\{0\}$, let  $C_i=[a_{i-1}, a_{i}]$. Let $A_i=\bigcup\{A\in\mathcal{A}: f[A]\cap C_i\neq\emptyset\}$.  Then $\{A_i: i\in n+1\}$ is a finite clopen cover of $\mathbf{X}$. Fix $i\in n+1$ and $x,y\in A_i$. There exist $A(x), A(y)\in\mathcal{A}$ such that $x\in A(x)$, $y\in A(y)$, $f[A(x)]\cap C_i\neq\emptyset\neq f[A(y)]\cap C_i$. Let $x_1\in A(x)$ and $y_1\in A(y)$ be such that $f(x_1)\in C_i$ and $f(y_1)\in C_i$. To complete the proof, it suffices to notice that  $|f(x)-f(y)|\leq |f(x)-f(x_1)|+|f(x_1)-f(y_1)|+|f(y_1)-f(y)|\leq\varepsilon$.  
\end{proof}

It is interesting that, although we do not know if it is possible to prove in $\mathbf{ZF}$ that $U_{\aleph_0}(\mathbf{X})=\cl_{C_{u}(\mathbf{X})}(C(\mathbf{X}, \mathbb{R}_{disc}))$ (see Theorem \ref{s1:t21}), the following proposition has a relatively simple proof in $\mathbf{ZF}$:

\begin{proposition}
	\label{s2:p12}
	$[\mathbf{ZF}]$  For every topological space $\mathbf{X}$, the following equality holds:
	$$U_{\aleph_0}^{\ast}(\mathbf{X})=\cl_{C_u(\mathbf{X})}C^{\ast}(\mathbf{X}, \mathbb{R}_{disc}).$$
\end{proposition}
\begin{proof}
	It has been noticed in Remark \ref{s1:r20} that, given a topological space $\mathbf{X}$, the set $U_{\aleph_0}^{\ast}(\mathbf{X})$ is closed in $C_u(\mathbf{X})$. It is obvious that $C^{\ast}(\mathbf{X}, \mathbb{R}_{disc})\subseteq U_{\aleph_0}^{\ast}(\mathbf{X})$. In consequence, $\cl_{C_u(\mathbf{X})}C^{\ast}(\mathbf{X}, \mathbb{R}_{disc})\subseteq U_{\aleph_0}^{\ast}(\mathbf{X})$. To prove that the reverse inclusion also holds, let us fix $f\in U_{\aleph_0}^{\ast}(\mathbf{X})$ and a positive real number $\varepsilon$. It follows from Proposition \ref{s2:p11} that there is  $n\in\omega$ for which we can fix a clopen partition $\{A_i; i\in n+1\}$ of $\mathbf{X}$ such that, for every $i\in n+1$, $\osc_{A_i}(f)\leq\frac{\varepsilon}{2}$. Without loss of generality, we may assume that, for every $i\in n+1$, $A_i\neq\emptyset$. We fix $z\in\prod\limits_{i\in n+1}A_i$ and define a function $g:X\to\mathbb{R}$ as follows: for every $i\in n+1$ and $t\in A_i$, $g(t)=f(z(i))$. Then $g\in C^{\ast}(\mathbf{X}, \mathbb{R}_{disc})$. Let $t_0\in X$. There exists a unique $i_0\in n+1$ such that $t_0\in A_{i_0}$. Then $|f(t_0)-g(t_0)|=|f(t_0)-f(z(i_0))|\leq\frac{\varepsilon}{2}<\varepsilon$. This shows that $f\in\cl_{C_u(\mathbf{X})}C^{\ast}(\mathbf{X}, \mathbb{R}_{disc})$.
\end{proof}

We recall that, by Theorem \ref{s1:t21}, the statement ``A completely regular space $\mathbf{X}$ is strongly zero-dimensional if and only if $U_{\aleph_0}(\mathbf{X})=C(\mathbf{X})$ is provable in $\mathbf{ZF+CMC}$; unfortunately, we do not know if this statement is also provable in $\mathbf{ZF}$. However, we can prove the following theorem in $\mathbf{ZF}$:

\begin{theorem}
	\label{s2:t13}
	$[\mathbf{ZF}]$ Let $\mathbf{X}$ be a completely regular space. Then $\mathbf{X}$ is strongly zero-dimensional if and only if $U_{\aleph_0}^{\ast}(\mathbf{X})=C^{\ast}(\mathbf{X})$. 
\end{theorem}
\begin{proof}
	Suppose that $U_{\aleph_0}^{\ast}(\mathbf{X})=C^{\ast}(\mathbf{X})$. Let $Z_1, Z_2$ be a pair of disjoint zero-sets of $\mathbf{X}$. Take a function $f\in C^{\ast}(\mathbf{X})$ such that $Z_1=\{x\in X: f(x)=0\}$ and $Z_2=\{x\in X: f(x)=1\}$. Since $f\in U_{\aleph_0}^{\ast}(\mathbf{X})$, there exists a finite clopen partition $\mathcal{A}$ of $\mathbf{X}$ such that, for every $A\in\mathcal{A}$, $\osc_{A}(f)\leq\frac{1}{2}$. Let $C=\bigcup\{A\in\mathcal{A}: A\cap Z_1\neq\emptyset\}$. Then $C$ is clopen in $\mathbf{X}$ and $Z_1\subseteq C$. Suppose that there exists $x\in C\cap Z_2$. For such an $x$, we choose $A\in\mathcal{A}$ such that $x\in A$ and $A\cap Z_1\neq\emptyset$. We choose $y\in A\cap Z_1$. Then $|f(x)-f(y)|\leq\frac{1}{2}$ because $\osc_{A}(f)\leq\frac{1}{2}$. On the other hand $|f(x)-f(y)|=1$, which is impossible. Hence $C\cap Z_2=\emptyset$. This proves that $\mathbf{X}$ is strongly zero-dimensional if $U_{\aleph_0}^{\ast}(\mathbf{X})=C^{\ast}(\mathbf{X})$.
	
	Now, suppose that $\mathbf{X}$ is strongly zero-dimensional and $f\in C^{\ast}(\mathbf{X})$. Let $a=\inf f[X]$ and $b=\sup f[X]$. Fix $\varepsilon >0$ and $n\in\omega$, such that $\frac{b-a}{n}<\frac{\varepsilon}{3}$. For $i\in n+1$, let $a_i= a+\frac{i(b-a)}{n}$. For every $i\in n$, let $E_i=[a_i, a_{i+1}]$. Let  $D_0=[a_2, b]$, $D_{n-1}=[a_0, a_{n-2}]$ and, for every $i\in n$ with $0<i<n-1$, let $D_i=[a_0, a_{i-1}]\cup [a_{i+2}, b]$. Since $\mathbf{X}$ is strongly zero-dimensional, we can choose a family $\{C_i: i\in n\}$ of clopen sets of $\mathbf{X}$ such that, for every $i\in n$, $f^{-1}[ E_i]\subseteq C_i$ and $C_i\cap f^{-1}[ D_i]=\emptyset$. Then $X=\bigcup\limits_{i\in n}C_i$. Let us fix $i\in n$ and $x,y\in C_i$. Suppose that $0<i<n-1$. Then $f(x), f(y)\in [a_{i-1}, a_{i+2}]$, so $|f(x)-f(y)|\leq\varepsilon$. A similar argument shows that if $i=0$ or $i=n-1$, then $|f(x)-f(y)|\leq\varepsilon$. Hence, for every $i\in n$, $\osc_{C_i}(f)\leq\varepsilon$. This proves that $f\in U_{\aleph_0}^{\ast}(\mathbf{X})$.  
\end{proof}

\begin{remark}
	\label{s2:r14}
	Let $\mathbf{X}$ be a completely regular space. It follows from Theorem \ref{s2:t13} that if $U_{\aleph_0}(\mathbf{X})=C(\mathbf{X})$, then it is provable in $\mathbf{ZF}$ that $\mathbf{X}$ is strongly zero-dimensional. It might be interesting to know if it can be proved in $\mathbf{ZF}$ that if $\mathbf{X}$ is strongly zero-dimensional, then  $U_{\aleph_0}(\mathbf{X})=C(\mathbf{X})$.
\end{remark}

Let us give a $\mathbf{ZF}$-proof of the following lemma for completeness.

\begin{lemma}
	\label{s2:l15}
	$[\mathbf{ZF}]$
	Let $\mathbf{X}$ be a dense subspace of a topological space $\mathbf{Y}$. Let $F\subseteq C^{\ast}(\mathbf{X})$. Suppose that every function $f\in F$ is continuously extendable over $\mathbf{Y}$. Then every function from $\cl_{C_{u}(\mathbf{X})}(F)$ is continuosly extendable over $\mathbf{Y}$. Furthermore, if $\tilde{F}=\{f\in C^{\ast}(\mathbf{Y}): f\upharpoonright X\in F\}$, then
	$$\cl_{C_{u}(\mathbf{X})}(F)=\{f\upharpoonright X: f\in\cl_{C_{u}(\mathbf{Y})}(\tilde{F})\}.$$
\end{lemma}

\begin{proof} Let $f\in \cl_{C_{u}(\mathbf{X})}(F)$. We consider any pair $a,b$ of real numbers such that $a<b$. We fix a positive real number $\varepsilon$ such that $a+\varepsilon<b-\varepsilon$. We can fix $g\in F$ such that, for every $x\in X$, $|f(x)-g(x)|\leq\varepsilon$. Let $A_f=f^{-1}[(-\infty, a]]$, $B_f=f^{-1}[[b, +\infty)]$, $C_g=g^{-1}[(-\infty, a+\varepsilon]]$ and $D_g=g^{-1}[[b-\varepsilon, +\infty)]$. We notice that $A_f\subseteq C_g$ and $B_f\subseteq D_g$. Since $g$ is continuously extendable over $\mathbf{Y}$, we have $\cl_{\mathbf{Y}}(C_g)\cap\cl_{\mathbf{Y}}(D_g)=\emptyset$. This implies that $\cl_{\mathbf{Y}}(A_f)\cap\cl_{\mathbf{Y}}(B_f)=\emptyset$. Hence, it follows from Blasco's Extension Theorem (see Theorem \ref{s1:t11}) that $f$ is continuously extendable over $\mathbf{Y}$. 
	
	We omit an elementary proof that the equality $\cl_{C_{u}(\mathbf{X})}(F)=\{f\upharpoonright X: f\in\cl_{C_{u}(\mathbf{Y})}(\tilde{F})\}$ holds.
\end{proof}

The standard well-known proof in $\mathbf{ZFC}$ of the first part of Lemma \ref{s2:l15} involves the Fr\'echet-Urysohn property of the function space $C_{u}(\mathbf{X})$ for it involves the $\mathbf{ZFC}$-fact that if $f\in \cl_{C_{u}(\mathbf{X})}(F)$, then there is a sequence $(f_n)_{n\in\omega}$ of functions from $F$ which is uniformly convergent to $f$. However, in $\mathbf{ZF}$, the space $C_{u}(\mathbf{X})$ may  fail to be Fr\'echet-Urysohn (see, e.g., \cite[Theorem 4.54]{her}). This is why we have included a $\mathbf{ZF}$-proof of Lemma \ref{s2:l15}.  

\begin{theorem}
	\label{s2:t16}
	$[\mathbf{ZF}]$ Let $\mathbf{X}$ be a non-empty zero-dimensional $T_1$-space. Then the extensions $e_2^{\ast}\mathbf{X}$ and $e_{U_{\aleph_0}^{\ast}(\mathbf{X})}\mathbf{X}$ of $\mathbf{X}$ are equivalent.  Furthermore:
	$$U_{\aleph_0}^{\ast}(\mathbf{X})=\{ f\circ e_2^{\ast}: f\in U_{\aleph_0}^{\ast}(e_2^{\ast}\mathbf{X})\},$$
	and $\beta_0\mathbf{X}$ exists if and only if $e_{U_{\aleph_0}^{\ast}(\mathbf{X})}\mathbf{X}$ is compact.
\end{theorem}
\begin{proof}
	Clearly, $C(\mathbf{X}, \mathbf{2})\subseteq U_{\aleph_0}^{\ast}(\mathbf{X})$. On the other hand, it is clear that, for every $f\in C^{\ast}(\mathbf{X}, \mathbb{R}_{disc})$, the function $f\circ (e_2^{\ast})^{-1}$ is continuously extendable over $e_2^{\ast}\mathbf{X}$. All this, taken together with Theorem \ref{s1:t23} and Lemma \ref{s2:l15}, implies that the extensions $e_2^{\ast}\mathbf{X}$ and $e_{U_{\aleph_0}^{\ast}(\mathbf{X})}\mathbf{X}$ of $\mathbf{X}$ are equivalent. The equality $U_{\aleph_0}^{\ast}(\mathbf{X})=\{ f\circ e_2^{\ast}: f\in U_{\aleph_0}^{\ast}(e_2^{\ast}\mathbf{X})\}$ follows from Lemma \ref{s2:l15}. To complete the proof, it suffices to apply Theorem \ref{s2:t1}.
\end{proof}

\begin{corollary}
	\label{s2:c17}
	$[\mathbf{ZF}]$ Let $\mathbf{X}$ be a zero-dimensional $T_1$-space whose Banaschewki compactification $\beta_0\mathbf{X}$ exists. Then $\beta_0\mathbf{X}\approx e_{U_{\aleph_0}^{\ast}(\mathbf{X})}\mathbf{X}$ and 
	$$U_{\aleph_0}^{\ast}(\mathbf{X})=\{ f\upharpoonright X: f\in C(\beta_0\mathbf{X})\}.$$
	In consequence, $C(\beta_0\mathbf{X})$ consists of the continuous extensions over $\beta_0\mathbf{X}$ of functions from $U_{\aleph_0}^{\ast}(\mathbf{X})$.
\end{corollary}
\begin{proof}
	Since the space $\beta_0\mathbf{X}$ is compact, Tychonoff and zero-dimensional, it is strongly zero-dimensional. It follows from Theorem \ref{s2:t13} that $C^{\ast}(\beta_0\mathbf{X})=U_{\aleph_0}^{\ast}(\beta_0\mathbf{X})$. Thus, by the compactness of $\beta_0\mathbf{X}$, we have $C(\beta_{0}\mathbf{X})=U_{\aleph_0}^{\ast}(\beta_0\mathbf{X})$. To conclude the proof, it suffices to apply Theorem \ref{s2:t16} and Corollary \ref{s2:c2}
\end{proof}

\subsection{Banaschewski compactifications via spaces of maximal ideals with hull-kernel topologies}
\label{s2.3}

For a zero-dimensional $T_1$-space $\mathbf{X}$ for which $\beta_0\mathbf{X}$ exists, let us have a look at the space $\BMax(U_{\aleph_0}(\mathbf{X}))$ and its connection with $\beta_0\mathbf{X}$. To this aim, we need some preliminary facts. We begin with the following definition:

\begin{definition}
	\label{s2:d18}
	Let $\mathbf{X}$ be a zero-dimensional $T_1$-space.
	\begin{enumerate}
		\item[(i)] For every $p\in e_{2}^{\ast} X$, let 
		$$\mathbb{A}^p=\{Z\in\mathcal{CO}_{\delta}(\mathbf{X}): p\in\cl_{e_{2}^{\ast}\mathbf{X}}(Z)\}.$$
		\item[(ii)] A $c_{\delta}$-\emph{ultrafilter} in $\mathbf{X}$ is an ultrafilter in $\mathcal{CO}_{\delta}(\mathbf{X})$. A $c_{\delta}$-\emph{filter} in $\mathbf{X}$ is a filter in $\mathcal{CO}_{\delta}(\mathbf{X})$.
		\item[(iii)] If $\mathcal{F}$ is a $c_{\delta}$-filter in $\mathbf{X}$, then
		$$I[\mathcal{F}]=\{f\in U_{\aleph_0}(\mathbf{X}): Z(f)\in\mathcal{F}\}.$$
		\item[(iv)] If $I$ is an ideal in $U_{\aleph_0}(\mathbf{X})$, then 
		$$Z[I]=\{ Z(f): f\in I\}.$$
	\end{enumerate}
\end{definition}

We omit a simple proof of the following lemma:
\begin{lemma}
	\label{s2:l19}
	$[\mathbf{ZF}]$ For every zero-dimensional $T_1$-space  $\mathbf{X}$ and arbitrary $c_{\delta}$-sets $Z, Z_1,Z_2$ in $\mathbf{X}$, the following conditions are satisfied:
	\begin{enumerate}
		\item[(i)] there exists $f\in C(\mathbf{X})$ such that $Z=Z(f)$ and $f$ is continuously extendable over $e_{2}^{\ast}\mathbf{X}$;
		\item[(ii)] if $Z_1\cap Z_2=\emptyset$, then $\cl_{e_{2}^{\ast}\mathbf{X}}(Z_1)\cap \cl_{e_{2}^{\ast}\mathbf{X}}(Z_2)=\emptyset$;
		\item[(iii)] $\cl_{e_{2}^{\ast}\mathbf{X}}(Z_1)\cap \cl_{e_{2}^{\ast}\mathbf{X}}(Z_2)= \cl_{e_{2}^{\ast}\mathbf{X}}(Z_1\cap Z_2)$.
	\end{enumerate}
\end{lemma}

\begin{proposition}
	\label{s2:p20}
	$[\mathbf{ZF}]$ Let $\mathbf{X}$ be a zero-dimensional $T_1$-space. Then the following conditions are satisfied:
	\begin{enumerate}
		\item[(i)] for every $p\in e_{2}^{\ast} X$, $\mathbb{A}^p$ is a $c_{\delta}$-ultrafilter in $\mathbf{X}$;
		\item[(ii)] if $e_{2}^{\ast}\mathbf{X}$ is compact, then, for every $c_{\delta}$-ultrafilter $\mathcal{F}$ in $\mathbf{X}$, there exists a unique $p\in e_{2}^{\ast}X$ such that $\mathcal{F}=\mathbb{A}^p$.
	\end{enumerate}
\end{proposition}
\begin{proof}
	(i) Let $p\in e_{2}^{\ast} X$. Clearly, $\emptyset\notin \mathbb{A}^p$. Let $Z_1,Z_2\in\mathbb{A}^p$. It follows from Lemma \ref{s2:l19}(iii) that  $Z_1\cap Z_2\in \mathbb{A}^p$. It is obvious that if $Z\in\mathbb{A}^p$, $C\in\mathcal{CO}_{\delta}(\mathbf{X})$ and $Z\subseteq C$, then $C\in\mathbb{A}^p$. Hence $\mathbb{A}^{p}$ is a $c_{\delta}$-filter in $\mathbf{X}$.
	
	Suppose that $D\in\mathcal{CO}_{\delta}(\mathbf{X})$ and $D\notin\mathbb{A}^{p}$. Then $p\notin \cl_{e_{2}^{\ast}\mathbf{X}}(D)$. Thus, there exists $E\in\mathcal{CO}(e_{2^{\ast}}\mathbf{X})$ such that $p\in E$ and $E\cap D=\emptyset$, Then $E\cap X\in\mathbb{A}^{p}$ and $E\cap X\cap D=\emptyset$. Hence $\mathbb{A}^p$ is a $c_{\delta}$-ultrafilter in $\mathbf{X}$.
	
	(ii) Now, suppose that $e_{2}^{\ast}\mathbf{X}$ is compact and $\mathcal{F}$ is a $c_{\delta}$-ultrafilter in $\mathbf{X}$. Then the family $\{\cl_{e_{2}^{\ast}\mathbf{X}}(F): F\in\mathcal{F}\}$ is a centered family of closed sets of $e_{2}^{\ast}\mathbf{X}$. By the compactness of $e_{2}^{\ast}\mathbf{X}$, there exists $p\in\bigcap \{\cl_{e_{2}^{\ast}\mathbf{X}}(F): F\in\mathcal{F}\}$. Then $\mathcal{F}\subseteq\mathbb{A}^p$. Since $\mathcal{F}$ is a $c_{\delta}$-ultrafilter, we deduce from (i) that $\mathcal{F}=\mathbb{A}^p$.
\end{proof}

In what follows, given a non-empty topological space $\mathbf{X}$, for a real number $c$, we denote by $\mathbf{c}$ the constant function from $C(\mathbf{X})$ such that $\mathbf{c}[X]=\{c\}$.

\begin{proposition}
	\label{s2:p21}
	$[{\mathbf{ZF}}]$ Let $\mathbf{X}$ be a non-empty zero-dimensional $T_1$-space. Then the following conditions (a) and (b) are satisfied:
	\begin{enumerate}
		\item[(a)] For every $c_{\delta}$-filter $\mathcal{F}$ in $\mathbf{X}$, the set $I[\mathcal{F}]$ is a proper ideal of $U_{\aleph_0}(\mathbf{X})$.
		\item[(b)] $\mathbf{CMC}$ implies the following:
		\begin{enumerate}
			\item[(i)] for every $f\in U_{\aleph_0}(\mathbf{X})$, if $Z(f)=\emptyset$, then $\frac{1}{f}\in U_{\aleph_0}(\mathbf{X})$;
			\item[(ii)] for every proper ideal $I$ of $U_{\aleph_0}(\mathbf{X})$, the set $Z[I]$ is a $c_{\delta}$-filter in $\mathbf{X}$; 
			\item[(iii)] if $\mathcal{F}$ is a $c_{\delta}$-ultrafilter in $\mathbf{X}$, then $I[\mathcal{F}]$ is a maximal ideal of $U_{\aleph_0}(\mathbf{X})$;
			\item[(iv)] if  $M$ is a maximal ideal of $U_{\aleph_0}(\mathbf{X})$, then $Z[M]$ is a $c_{\delta}$-ultrafilter in $\mathbf{X}$;
			\item[(v)] if $e_{2}^{\ast}\mathbf{X}$ is compact, then, for every maximal ideal $M$ of $U_{\aleph_0}(\mathbf{X})$, there exists a unique $p\in e_{2}^{\ast}X$ such that $M=I[\mathbb{A}^p]$.
		\end{enumerate}
	\end{enumerate}
\end{proposition}
\begin{proof}
	Although, in the proof of this proposition, we use very similar arguments to relevant ones in \cite[Chapter 2]{gj}, there are essential differences as well. Therefore, let us give details for completeness.
	
	(a) Let $\mathcal{F}$ be a $c_{\delta}$-filter in $\mathbf{X}$. Let $f,g, h \in U_{\aleph_0}(\mathbf{X})$ be such that $Z(f)$ and $Z(g)$ are both in $\mathcal{F}$. Since $Z(f)\cap Z(g)\subseteq Z(f-g)$ and $Z(f)\subseteq Z(fh)$, we infer that $f-g\in I[\mathcal{F}]$ and $fh\in I[\mathcal{F}]$. Since $\emptyset\notin\mathcal{F}$, we deduce that $\mathbf{1}\notin I[\mathcal{F}]$. All this taken together implies that $I[\mathcal{F}]$ is a proper ideal of $U_{\aleph_0}(\mathbf{X})$. \medskip
	
	(b) Let us assume $\mathbf{CMC}$. Then, by Theorem \ref{s1:t21}(c)(ii), $A(\mathbf{X})=U_{\aleph_0}(\mathbf{X})$. Take any $f\in A(\mathbf{X})$ such that $Z(f)=\emptyset$. In the light of the definition of $A(\mathbf{X})$ and Theorem \ref{s1:t18}(b), one can easily check that $\frac{1}{f}\in A(\mathbf{X})$. Hence (i) holds.
	
	(ii) Let $I$ be a proper ideal of $U_{\aleph_0}(\mathbf{X})$. It follows from $\mathbf{CMC}$ and Theorem \ref{s1:t21}(c) that $Z[I]\subseteq\mathcal{CO}_{\delta}(\mathbf{X})$. We deduce from (i) that if $f\in I$, then $Z(f)\neq\emptyset$ because $I$ is a proper ideal. Hence $\emptyset\notin Z[I]$.  Let $Z_1,Z_2\in Z[I]$. There exist $f_1,f_2\in I$ such that $Z(f_1)=Z_1$ and $Z(f_2)=Z_2$. Let $g= f_1^2+f_2^2$. Then $g\in I$ and $Z(g)=Z_1\cap Z_2$. Hence $Z_1\cap Z_2\in Z[I]$. Finally, suppose that $Z\in CO_{\delta}(\mathbf{X})$ and $Z_1\subseteq Z$. Then there exists $h\in \cl_{C_{u}(\mathbf{X})}(C(\mathbf{X}, \mathbb{R}_{disc}))$ such that $Z=Z(h)$. By Theorem \ref{s1:t21}(c)(ii), $h\in U_{\aleph_0}(\mathbf{X})$. 
	Then $h\cdot f_1\in I$ and $Z(h\cdot f_1)= Z$. Hence $Z\in Z[I]$. This completes the proof of (ii).
	
	(iii) Suppose that $\mathcal{F}$ is a $c_{\delta}$-ultrafilter in $\mathbf{X}$. Then, by (a), $I[\mathcal{F}]$ is a proper ideal in $U_{\aleph_0}(\mathbf{X})$. Suppose that $M$ is a proper ideal in $U_{\aleph_0}(\mathbf{X})$ such that $I[\mathcal{F}]\subseteq M$. By (ii), $Z[M]$ is a $c_{\delta}$-filter in $\mathbf{X}$. Let $Z\in\mathcal{F}$. It follows from Theorem \ref{s1:t21}(c)(ii) that there exists $f\in U_{\aleph_0}(\mathbf{X})$ such that $Z=Z(f)$. Then $f\in I[\mathcal{F}]$, so $f\in M$. This implies that $\mathcal{F}\subseteq Z[M]$. Since $\mathcal{F}$ is a $c_{\delta}$-ultrafilter in $\mathbf{X}$, we have $\mathcal{F}= Z[M]$. We notice that if $g\in M\setminus I[\mathcal{F}]$, then $Z(g)\in Z[M]\setminus \mathcal{F}$. Hence $M=I[\mathcal{F}]$. This proves (iii).
	
	(iv) Now, let us assume that $M$ is a maximal ideal of $U_{\aleph_0}(\mathbf{X})$. Then the ideal $M$ is proper, so it follows from (ii) that $Z[M]$ is a $c_{\delta}$-filter in $\mathbf{X}$. Let $\mathcal{U}$ be any $c_{\delta}$-filter in $\mathbf{X}$ such that $Z[M]\subseteq \mathcal{U}$. Then $M\subseteq I[\mathcal{U}]$ and, by (iii), $I[\mathcal{U}]$ is a proper ideal of $U_{\aleph_0}(\mathbf{X})$. Hence, by the maximality of $M$, $M=I[\mathcal{U}]$. Suppose that $Z\in\mathcal{U}$. There exists $f\in U_{\aleph_0}(\mathbf{X})$ such that $Z=Z(f)$. Then $f\in I[\mathcal{U}]$, so $Z\in Z[M]$. Hence $Z[M]=\mathcal{U}$.
	
	(v) Suppose that $e_{2}^{\ast}\mathbf{X}$ is compact, and $M$ is a maximal ideal of $U_{\aleph_0}(\mathbf{X})$. By (iv), $Z[M]$ is a $c_{\delta}$-ultrafilter in $\mathbf{X}$. By Proposition \ref{s2:p20}(iv), there exists a unique $p\in e_{2}^{\ast}X$ such that $Z[M]=\mathbb{A}^p$. One can check that $M=I[\mathbb{A}^p]$ and if  $p\in e_{2}^{\ast}X$ is such that $M=I[\mathbb{A}^q]$, then $p=q$ because $\mathbb{A}^p=\mathbb{A}^q$.
\end{proof}

\begin{theorem}
	\label{s2:t22}
	$[\mathbf{ZF}]$ Let $\mathbf{X}$ be a non-empty zero-dimensional $T_1$-space for which the Banaschewski compactification $\beta_0\mathbf{X}$ exists. Then $\mathbf{CMC}$ implies that the topological spaces $\BMax(U_{\aleph_0}(\mathbf{X}))$ and $\beta_0\mathbf{X}$ are homeomorphic. 
\end{theorem}
\begin{proof}
	Since $\beta_0\mathbf{X}$ exists. it follows from Corollary 2.2 that $e_{2}^{\ast}\mathbf{X}$ is a Hausdorff compactification of $\mathbf{X}$ equivalent to $\beta_0\mathbf{X}$. Therefore, we may assume that $\beta_0\mathbf{X}=e_{2}^{\ast}\mathbf{X}$. Assuming $\mathbf{CMC}$, we infer from Proposition \ref{s2:p21} that we can define a mapping $\psi: \beta_0\mathbf{X}\to\text{Max}(U_{\aleph_0}(\mathbf{X}))$ as follows:
	$$(\forall p\in\beta_0\mathbf{X}) \psi(p)=I[\mathbb{A}^p].$$
	
	It follows from Proposition \ref{s2:p21} that $\psi$ is a bijection. To show that $\psi$ is a homeomorphism of $\beta_0\mathbf{X}$ onto $\BMax(U_{\aleph_0}(\mathbf{X}))$, it suffices to prove that $\psi$ is continuous. To this aim, we fix $f\in U_{\aleph_0}(\mathbf{X})$ and consider the basic open set $D(f)=\{M\in\Max(U_{\aleph_0}(\mathbf{X})): f\notin M\}$. We notice that $\psi^{-1}[D(f)]=\beta_0 X\setminus\cl_{\beta_0\mathbf{X}}Z(f)$. This shows that $\psi$ is continuous. 
\end{proof}

\begin{remark}
	\label{s2:r23}
	Regarding Theorem \ref{s2:t22}, we can say more about it. Namely, let us assume $\mathbf{CMC}$ and suppose that $\mathbf{X}$ is a non-empty zero-dimensional $T_1$-space whose Banaschewski compactification $\beta_0\mathbf{X}$ exists. Let $\psi$ be the mapping defined in the proof of Theorem \ref{s2:t22}. We notice that, for every $p\in X$, $\psi(p)=\{f\in U_{\aleph_0}(\mathbf{X}): p\in Z(f)\}$, so for the mapping $\phi=\psi\upharpoonright X$, the pair $\langle \BMax(U_{\aleph_0}(\mathbf{X})), \phi\rangle$ is a Hausdorff compactification of $\mathbf{X}$ equivalent to $\beta_0\mathbf{X}$.
\end{remark}

We do not know if $\mathbf{CMC}$ can be omitted in Theorem \ref{s2:t22}, however,  if we replace $U_{\aleph_0}(\mathbf{X})$ with $U_{\aleph_0}^{\ast}(\mathbf{X})$, we get the following theorem of $\mathbf{ZF}$:

\begin{theorem}
	\label{s2:t24}
	$[\mathbf{ZF}]$
	Let $\mathbf{X}$ be a non-empty zero-dimensional $T_1$-space. Then the following conditions are equivalent:
	\begin{enumerate}
		\item[(i)] the Banaschewski compactification $\beta_0\mathbf{X}$ of $\mathbf{X}$ exists;
		\item[(ii)] the space $\BMax(U_{\aleph_0}^{\ast}(\mathbf{X}))$ is compact.
	\end{enumerate}
	Furthermore, if conditions (i)--(ii) are satisfied, then $\BMax(U_{\aleph_0}^{\ast}(\mathbf{X}))$ is homeomorphic with $\beta_0\mathbf{X}$.
\end{theorem}

\begin{proof}
	Suppose that (i) holds. Then the space $\beta_0\mathbf{X}$ is compact, Tychonoff and zero-dimensional, so it is strongly zero-dimensional. By Theorem \ref{s2:t13}, $C^{\ast}(\beta_0\mathbf{X})=U_{\aleph_0}^{\ast}(\beta_0\mathbf{X})$. Since $\beta_0\mathbf{X}$ is a compact Tychonoff space, we know from Theorem \ref{s1:t34} that the spaces $\beta_0\mathbf{X}$ and $\BMax(C(\beta_0\mathbf{X}))$ are homeomorphic. By Corollary \ref{s2:c17}, the rings $C(\beta_0\mathbf{X})$ and $U_{\aleph_0}^{\ast}(\mathbf{X})$ are isomorphic. This implies that the spaces $\BMax( C(\beta_0\mathbf{X}))$ and $\BMax(U_{\aleph_0}^{\ast}(\mathbf{X}))$ are homeomorphic. In consequence, the spaces $\beta_0\mathbf{X}$ and $\BMax(U_{\aleph_0}^{\ast}(\mathbf{X}))$ are homeomorphic, so (i) implies (ii). 
	
	Now, suppose that (ii) is satisfied. Then, by assigning to $p\in X$ the ideal $\psi^{\ast}(p)=\{f\in U_{\aleph_0}^{\ast}(\mathbf{X}): p\in Z(f)\}$, we obtain a compactification $\langle \BMax(U_{\aleph_0}^{\ast}(\mathbf{X})), \psi^{\ast}\rangle$ of $\mathbf{X}$. Let us consider any function $f\in U_{\aleph_0}^{\ast}(\mathbf{X})$ and show that $f\circ \psi^{\ast}$ is continuously extendable over $\BMax(U_{\aleph_0}^{\ast}(\mathbf{X}))$. To this aim, we fix any real numbers $a,b$ with $a<b$. Let $Z_a=\{x\in X: f(x)\leq a\}$ and $Z_b=\{x\in X: f(x)\geq b\}$. We are going to show that the sets $\psi^{\ast}[Z_a]$ and $\psi^{\ast}[Z_b]$ have disjoint closures in $\BMax(U_{\aleph_0}^{\ast}(\mathbf{X}))$. 
	
	Suppose that $M_0\in\cl(\psi^{\ast}[Z_a])\cap \cl(\psi^{\ast}[Z_b])$ where the closures are taken in the space $\BMax(U_{\aleph_0}^{\ast}(\mathbf{X}))$. Let $g:\mathbb{R}\to [0,  1]$ be the function defined as follows: 
	$$
	g(t)=\begin{cases} 0 &\text{ if } t\in (-\infty, a);\\
		\frac{t-a}{b-a} &\text{ if } t\in [a, b];\\
		1 &\text{ if } t\in (b, +\infty).
	\end{cases}
	$$
	Put $f_0=g\circ f$ and $f_1= \mathbf{1}-f_0$. Then $f_0, f_1\in U_{\aleph_0}^{\ast}(\mathbf{X})$. For every $i\in\{0, 1\}$, the set $V_i=\{M\in\BMax(U_{\aleph_0}^{\ast}(\mathbf{X})): f_0\notin M\}$ is open in $\BMax(U_{\aleph_0}^{\ast}(\mathbf{X}))$. Since $Z_a\subseteq Z(f_0)$ and $Z_b\subseteq Z(f_1)$, we have  $V_0\cap \psi^{\ast}[Z_a]=\emptyset$ and $V_1\cap\psi^{\ast}[Z_b]=\emptyset$. This implies that $M_0\notin V_0\cup V_1$, so $f_0, f_1\in M_0$. In consequence, $\mathbf{1}=f_0+f_1\in M_0$, which is impossible. The contradiction obtained shows that the sets $\psi^{\ast}[Z_a]$ and $\psi^{\ast}[Z_b]$ have disjoint closures in $\BMax(U_{\aleph_0}^{\ast}(\mathbf{X}))$. It follows from Blasco's Extension Theorem (see Theorem \ref{s1:t11}) that $U_{\aleph_0}^{\ast}(\mathbf{X})\subseteq C_{\psi^{\ast}}(\mathbf{X})$. By Theorems \ref{s1:t31} and \ref{s2:t16}, $\beta_0\mathbf{X}$ exists. Hence (ii) implies (i).
\end{proof}

\begin{remark}
	\label{s2:r25}
	Let $\mathbf{X}$ be a non-empty Tychonoff space. For $\mathcal{F}\in\mathcal{E}^{\ast}(\mathbf{X})$, we define:
	$$C_{\mathcal{F}}(\mathbf{X})=\{f\circ e_{\mathcal{F}}: f\in C^{\ast}(e_{\mathcal{F}}\mathbf{X})\};$$
	$$(\forall x\in X) h_{\mathcal{F}}(x)=\{f\in C_{\mathcal{F}}(\mathbf{X}): f(x)=0\}.$$
	Arguing similarly to the proof of Theorem \ref{s2:t24}, one can show that it holds in $\mathbf{ZF}$ that $e_{\mathcal{F}}\mathbf{X}$ is compact if and only if $\BMax(C_{\mathcal{F}}(\mathbf{X}))$ is compact; moreover, if $e_{\mathcal{F}}\mathbf{X}$ is compact, then the compactifications $e_{\mathcal{F}}\mathbf{X}$ and $\langle \BMax(C_{\mathcal{F}}(\mathbf{X})), h_{\mathcal{F}}\rangle$ of $\mathbf{X}$ are equivalent. In particular, it holds in $\mathbf{ZF}$ that 
	$\beta_T\mathbf{X}$ exists if and only if the space $\BMax(C^{\ast}(\mathbf{X}))$ is compact; moreover, if $\beta_T\mathbf{X}$ exists, then $\langle \BMax(C^{\ast}(\mathbf{X})), h_{C^{\ast}(\mathbf{X})}\rangle$ is a Hausdorff compactification of $\mathbf{X}$ equivalent to $\beta_T\mathbf{X}$. 
\end{remark}

\begin{remark}
	\label{s2:r26}
	In much the same way, as in the proof that (ii) implies (i) in Theorem \ref{s2:t24}, we can show in $\mathbf{ZF}$ that if $\mathbf{X}$ is a non-empty zero-dimensional $T_1$-space such that $\BMax(U_{\aleph_0}(\mathbf{X}))$ is compact, then $\beta_0\mathbf{X}$ exists.
\end{remark}

Theorem \ref{s2:t24} and Remark \ref{s2:r25} allow us to formulate the following theorem which extends the list of equivalent conditions given in Theorems \ref{s1:t30} and \ref{s2:t4}, and which shows that, in Theorem \ref{s1:t35}, $\BMax(C(\mathbf{X}))$ can be replaced with $\BMax(C^{\ast}(\mathbf{X}))$.

\begin{theorem}
	\label{s2:t27}
	$[\mathbf{ZF}]$
	The following conditions are all equivalent:
	\begin{enumerate}
		\item[(i)] $\mathbf{BPI}$;
		\item[(ii)] for every non-empty Tychonoff space $\mathbf{X}$, the space $\BMax(C^{\ast}(\mathbf{X}))$ is compact;
		\item[(iii)]for every infinite set $J$, the space $\BMax(C^{\ast}(\mathbf{2}^J))$ is compact;
		\item[(iv)] for every infinite set $J$, the space $\BMax(U_{\aleph_0}^{\ast}(\mathbf{2}^J))$ is compact;
		\item[(v)] for every non-empty zero-dimensional $T_1$-space $\mathbf{X}$,  $\BMax(U_{\aleph_0}^{\ast}(\mathbf{X}))$ is a compact space.
	\end{enumerate}
\end{theorem}

In the light of Theorem \ref{s2:t22} and Remark \ref{s2:r26}, we can write down the following corollary:

\begin{corollary}
	\label{s2:c28}
	$[\mathbf{ZF+CMC}]$  For every zero-dimensional $T_1$-space $\mathbf{X}$, it holds that $\beta_0\mathbf{X}$ exists if and only if $\BMax(U_{\aleph_0}(\mathbf{X}))$ is compact.
\end{corollary}

\begin{theorem}
	\label{s2:t29}
	$[\mathbf{ZF}]$
	Let $\mathbf{X}$ be a non-empty zero-dimensional $T_1$-space for which the Banaschewski compactification $\beta_0\mathbf{X}$ exists. Let $\mathbf{Y}$ be a non-empty Tychonoff space. Then the following conditions are satisfied:
	\begin{enumerate}
		\item[(i)] if the rings $U_{\aleph_0}^{\ast}(\mathbf{X})$ and $C(\mathbf{Y})$ are isomorphic, then $\beta_T\mathbf{Y}$  exists and, moreover, the spaces $\beta_0\mathbf{X}$ and $\beta_T\mathbf{Y}$ are homeomorphic;
		\item[(ii)] $\mathbf{CMC}$ implies that if the rings $U_{\aleph_0}(\mathbf{X})$ and $C(\mathbf{Y})$ are isomorphic, then $\beta_T\mathbf{Y}$ exists and, moreover, the spaces  $\beta_0\mathbf{X}$ and $\beta_T\mathbf{Y}$ are homeomorphic.
	\end{enumerate}
\end{theorem}

\begin{proof}
	(i) Suppose that the rings $U_{\aleph_0}^{\ast}(\mathbf{X})$ and $C(\mathbf{Y})$ are isomorphic. Then the spaces $\BMax(U_{\aleph_0}^{\ast}(\mathbf{X}))$ and $\BMax(C(\mathbf{Y}))$ are homeomorphic. By Theorem \ref{s2:t24}, $\BMax(C(\mathbf{Y}))$ is compact. Therefore, by Theorem \ref{s1:t34}, $\beta_T\mathbf{Y}$ exists and is homeomorphic with the space $\BMax(C(\mathbf{Y}))$. In consequence, $\beta_0\mathbf{X}$ and $\beta_T\mathbf{Y}$ are homeomorphic.\medskip
	
	(ii) Now, assume $\mathbf{CMC}$ and suppose that the rings $U_{\aleph_0}(\mathbf{X})$ and $C(\mathbf{Y})$ are isomorphic. Then the spaces $\BMax(U_{\aleph_0}(\mathbf{X}))$ and $\BMax(C(\mathbf{Y}))$ are homeomorphic. Hence, by Theorem \ref{s2:t22}, the spaces $\beta_0\mathbf{X}$ and $\BMax(C(\mathbf{Y}))$ are homeomorphic. In much the same way, as in the proof of (i), we conclude that $\beta_T\mathbf{Y}$ exists and is homeomorphic with $\beta_0\mathbf{X}$.
\end{proof}

In Section \ref{s10}, given a non-empty $\mathbb{N}$-compact space $\mathbf{X}$ for which $\beta_0\mathbf{X}$ exists, other results are obtained about the existence of a Tychonoff space $\mathbf{Y}$ for which the rings $U_{\aleph_0}(\mathbf{X})$ and $C(\mathbf{Y})$ are isomorphic.

\subsection{On Baire and $G_{\delta}$-topologies}
\label{s3}

In this section, we establish several facts concerning Baire topologies and $G_{\delta}$-topologies (see Definition \ref{s1:d17}). We begin with the following proposition:

\begin{proposition}
	\label{s3:p1} 
	$[\mathbf{ZF}]$ For every topological space $\mathbf{X}=\langle X, \tau\rangle$, the following conditions are satisfied:
	\begin{enumerate}
		\item[(i)] $\tau_z\subseteq\tau_{\delta}$ and $\tau\subseteq\tau_{\delta}$;
		\item[(ii)] if $\tau\subseteq\tau_z$ and $\mathbf{X}$ is a $T_0$-space, then $\mathbf{X}$ is functionally Hausdorff;
		\item[(iii)] if $\mathbf{X}$ is completely regular, then $\tau\subseteq\tau_{z}$;
		\item[(iv)] if $\tau\subseteq\tau_z$ and $\mathbf{X}$ is either a rimcompact $T_0$-space or a countably compact $($not necessarily $T_0$$)$ space, then $\mathbf{X}$ is completely regular;
		\item[(v)] $\mathbf{CMC}$ implies that if $\mathbf{X}$ is completely regular, then $\tau_{\delta}=\tau_z$. 
	\end{enumerate}
\end{proposition}

\begin{proof}
	It is obvious that (i) and (iii) hold. For the proofs of (ii) and (iv), we assume that $\mathbf{X}=\langle X, \tau\rangle$ is a given topological space for which $\tau\subseteq\tau_z$.\medspace
	
	(ii) Suppose that $\mathbf{X}$ is a $T_0$-space. For a pair $x,y$ of distinct points of $X$, we choose $U\in\tau$ such that $U\cap\{x,y\}$ is a singleton.  Without loss of generality,  we may assume that $x\in U$ and $y\notin U$. Since $\tau\subseteq\tau_z$, we can choose $f\in C(\mathbf{X})$ such that $x\in Z(f)\subseteq U$.  Then $g=|\frac{f}{f(y)}|\in C(\mathbf{X})$, $g(x)=0$ and $g(y)=1$. Hence $\mathbf{X}$ is functionally Hausdorff.\medspace
	
	(iv)  Now, we assume that $A$ is a closed set in $\mathbf{X}$, and  $x\in X\setminus A$. 
	
	Suppose that $\mathbf{X}$ is a rimcompact  $T_0$-space. Since $\tau\subseteq\tau_z$, it follows from (ii) that $\mathbf{X}$ is Hausdorff. Since every Hausdorff rimcompact space is regular, the space $\mathbf{X}$ is regular. This, together with the rimcompactness of $\mathbf{X}$, implies that there exists $U\in\tau$ such that $x\in U\subseteq\cl_{\mathbf{X}}(U)\subseteq X\setminus A$ and $\text{bd}_{\mathbf{X}}(U)$ is compact.  Since $U\in\tau_z$, there exists $f\in C(\mathbf{X})$ such that $x\in Z(f)\subseteq U$. Then $\text{bd}_{\mathbf{X}}(U)\subseteq f^{-1}[\mathbb{R}\setminus \{0\}]$. Since the set $\text{bd}_{\mathbf{X}}(U)$ is compact in $\mathbf{X}$, there exists a positive real number $r$ such that, for every $t\in \text{bd}_{\mathbf{X}}(U)$, $|f(t)|>r$. Then $g=\min\{\frac{|f|}{r}, 1\}\in C(\mathbf{X})$, $g(x)=0$ and $\text{bd}_{\mathbf{X}}(U)\subseteq g^{-1}(1)$. We define a function $h: X\to [0,1]$ as follows:
	$$
	h(t)=\begin{cases} g(t) &\text{ if } t\in \cl_{\mathbf{X}}(U);\\
		1 &\text{ if } t\in X\setminus\cl_{\mathbf{X}}(U).\end{cases}
	$$
	One can easily check that $h\in C(\mathbf{X}, [0,1])$, $h(x)=0$ and $A\subseteq h^{-1}[\{1\}]$. Hence $\mathbf{X}$ is completely regular.
	
	Now, suppose that $\mathbf{X}$ is a countably compact space. Since $\tau\subseteq\tau_z$, there exists $\psi\in C(\mathbf{X}, [0,1])$ with $x\in Z(\psi)\subseteq X\setminus A$. Since $A$ is countably compact, there exists $s\in (0, 1]$ such that, for every $y\in A$, $\psi(y)>s$. Then $\phi=\min\{\frac{\psi}{s}, 1\}\in C(\mathbf{X}, [0,1])$, $\phi(x)=0$ and $A\subseteq \phi^{-1}[\{1\}]$.\medspace
	
	(v) Assume $\mathbf{CMC}$ and suppose that $\mathbf{X}$ is completely regular. By (i), $\tau_z\subseteq\tau_{\delta}$, so it suffices to show that $\tau_{\delta}\subseteq\tau_z$. To this aim, we consider any $V\in\tau_{\delta}$ and $x\in V$. There exists a family $\{U_n: n\in\omega\}\subseteq\tau$ such that $x\in\bigcap\limits_{n\in\omega}U_n\subseteq V$. For every $n\in\omega$, let $C_n=\{f\in C(\mathbf{X}, [0,1]): f(x)=0\wedge X\setminus U_n\subseteq f^{-1}[\{1\}]\}$. Since $\mathbf{X}$ is completely regular, for every $n\in\omega$, $C_n\neq\emptyset$. It follows from $\mathbf{CMC}$ that there exists a family $\{F_n: n\in\omega\}$ of non-empty finite sets such that, for every $n\in\omega$, $F_n\subseteq C_n$. For every $n\in\omega$, let $f_n=\min\{f: f\in F_n\}$. Then, for every $n\in\omega$, $f_n\in C(\mathbf{X}, [0,1])$, $f_n(x)=0$ and $X\setminus U_n\subseteq f_n^{-1}[\{1\}]$. Put $f=\sum\limits_{n\in\omega}\frac{f_n}{2^{n+1}}$. It is easily seen that $x\in Z(f)\subseteq V$. Hence $V\in\tau_z$.
\end{proof}

\begin{corollary}
	\label{s3:c2}
	$[\mathbf{ZF+CMC}]$ For every completely regular space $\mathbf{X}$, $(\mathbf{X})_z=(\mathbf{X})_{\delta}$ and $(\mathbf{X})_z$ is a $P$-space. 
\end{corollary}
\begin{proof}
	This follows directly from Proposition \ref{s3:p1}(v) and Theorem \ref{s1:t18}(b).
\end{proof}

\begin{proposition}
	\label{s3:p3}
	$[\mathbf{ZF}]$ The statement ``For every Tychonoff space $\mathbf{X}$, $(\mathbf{X})_z$ is a $P$-space'' implies $\mathbf{CMC}_{\omega}$. 
\end{proposition}
\begin{proof}
	The proposition can be proved in much the same way, as \cite[Theorem 4.10]{kow}.
\end{proof}

Regarding item (iv) of Proposition \ref{s3:p1}, the following example is worth noticing for it shows that, in general, for a Hausdorff space $\langle X, \tau\rangle$,  the inclusion $\tau\subseteq \tau_z$ is not enough to get the complete regularity of $\langle X, \tau\rangle$:

\begin{example}
	\label{s3:e4}
	$[\mathbf{ZF}]$
	
	(a) Let $X=\mathbb{R}$ and $A=\{\frac{1}{n}: n\in\mathbb{N}\}$. Let $\tau$ be the topology on $X$ generated by $\tau_{nat}\cup\{\mathbb{R}\setminus A\}$. It is well known that the space $\langle X, \tau\rangle$ is Hausdorff but not regular (see, e.g., \cite[Example 1.5.6]{en}). Since, for every $x\in\mathbb{R}$, $\{x\}\in\mathcal{Z}(\mathbf{X})$, the inclusion $\tau\subseteq\tau_z$ holds. In fact, $\tau_z=\mathcal{P}(\mathbb{R})$, so $(\mathbf{X})_z$ is discrete.\medskip
	
	(b) (\emph{Mysior's space}) The simplest known example of a regular $T_1$-space which is not completely regular in $\mathbf{ZFC}$ is Mysior's space constructed in \cite{mys}. Let $\mathbf{M}=\langle M, \tau\rangle$ be Mysior's space shown in \cite{mys} (see also \cite[Example 1.5.9]{en}). Although Mysior's arguments that $\mathbf{M}$ is not completely regular are not all in $\mathbf{ZF}$, one can modify them to see that $\mathbf{M}$ is also not completely regular in $\mathbf{ZF}$. To show this, let us use the notation from \cite{mys} and, in the inductive process in \cite{mys},  assuming that $n\in\mathbb{N}$ is such that $K_n$ is Dedekind-infinite, choose a denumerable subset $C$ of $K_n$.  Next, notice that even in $\mathbf{ZF}$, for every $\langle  c, 0\rangle\in C$, the set $I_c^{'}\setminus f^{-1}[\{0\}]$ is countable. Furthermore, there is a family $\{\leq_c:\langle c, 0\rangle\in C\}$ such that, for every $\langle c, 0\rangle\in C$, $\leq_c$ is a well-ordering on $I_c^{'}\setminus f^{-1}[\{0\}]$. Therefore, the set $\bigcup\{I_c^{'}\setminus f^{-1}[\{0\}]: \langle c, o\rangle\in C\}$ is also countable in $\mathbf{ZF}$ as a countable union of well-ordered countable sets. Its projection $P$ into the line $\{\langle x, 0\rangle: x\in\mathbb{R}\}$ is countable, so the set $F=\{\langle x, 0\rangle: x\in [n, n+1]\}\setminus P$ is Dedekind-infinite because no infinite Dedekind-finite subset of $\mathbb{R}$ is of type $G_{\delta}$ in $\mathbb{R}$. All this taken together with the arguments from \cite{mys} shows that it holds in $\mathbf{ZF}$ that $\mathbf{M}$ is a regular $T_1$-space that is not completely regular.  It is easy to notice that $\tau\subseteq\tau_z$. 
\end{example}

\begin{definition}
	\label{s3:d5}
	A topological space $\mathbf{X}$ is of \emph{countable pseudocharacter} if every singleton of $X$ is of type $G_{\delta}$ in $\mathbf{X}$.
\end{definition}

\begin{proposition}
	\label{s3:p6}
	$[\mathbf{ZF}]$
	\begin{enumerate}
		\item[(i)] For every topological space $\mathbf{X}$, it holds that $(\mathbf{X})_{\delta}$ is discrete if and only if $\mathbf{X}$ is of countable pseudocharacter.
		\item[(ii)] For every topological space $\mathbf{X}$ it holds that $(\mathbf{X})_z$ is discrete if and only if every singleton of $X$ is a zero-set in $\mathbf{X}$.
		
		\item[(iii)] $\mathbf{CMC}$ implies that if $\mathbf{X}$ is a completely regular space and $x\in X$ is such that $\{x\}\in\mathcal{G}_{\delta}(\mathbf{X})$, then $\{x\}\in\mathcal{Z}(\mathbf{X})$.
		
		\item[(iv)] $\mathbf{CMC}$ implies that if a completely regular space $\mathbf{X}$ is of countable pseudocharacter, then $(\mathbf{X})_z$ is discrete. 
		
	\end{enumerate}
\end{proposition}

\begin{proof}
	It is obvious that (i) and (ii) hold, and (iv) follows from (iii). Mimicking the proof of (v) of Proposition \ref{s3:p1}, one can easily show that (iii) also holds.
\end{proof}

The following new open problems seem interesting:

\begin{problem}
	\label{s3:q7}
	Is there a model of $\mathbf{ZF}$ in which there exists a Tychonoff space $\mathbf{X}$ of countable pseudocharacter for which $(\mathbf{X})_z$ is not discrete?
\end{problem}

\begin{problem}
	\label{s3:q8}
	Is there a model of $\mathbf{ZF}$ in which there exists a Tychonoff space $\mathbf{X}$ for which $(\mathbf{X})_z\neq (\mathbf{X})_{\delta}$?
\end{problem}

Clearly, a positive answer to Problem \ref{s3:q7} gives also a positive answer to Problem \ref{s3:q8}.

\begin{proposition} 
	\label{s3:p9}
	$[\mathbf{ZF}]$ For every topological space  $\mathbf{X}$, the space $(\mathbf{X})_z$ is zero-dimension\-al; furthermore, $(\mathbf{X})_z$ is $T_0$ if and only if $\mathbf{X}$ is functionally Hausdorff.
\end{proposition}
\begin{proof}
	Let us fix any topological space $\mathbf{X}=\langle X, \tau\rangle$. To show that the space $(\mathbf{X})_z$ is zero-dimensional, it suffices to notice that if $U\in\mathcal{Z}(\mathbf{X})$, then $X\setminus U\in\tau_z$ because $X\setminus U$ is a countable union of zero-sets of $\mathbf{X}$. It is obvious that if $\mathbf{X}$ is functionally Hausdorff, then $(\mathbf{X})_z$ is Hausdorff, so also $T_0$. Now, suppose that $(\mathbf{X})_z$ is $T_0$. Then $(\mathbf{X})_z$ is Hausdorff for it is a zero-dimensional $T_0$-space. Let $x_1, x_2$ be a pair of distinct points of $X$. Since $(\mathbf{X})_z$ is Hausdorff, there exists a pair $Z_1,Z_2$ of members of $\mathcal{Z}(\mathbf{X})$ such that $Z_1\cap Z_2=\emptyset$, $x_1\in Z_1$ and $x_2\in Z_2$. This implies that there exists $f\in C(\mathbf{X})$ such that $f(x_1)=0$ and $f(x_2)=1$, so $\mathbf{X}$ is functionally Hausdorff.
\end{proof}

\section{Basic facts on $\mathbf{E}$-compact spaces in $\mathbf{ZF}$}
\label{s4}

Let us have a look at the concepts from Definition \ref{s1:d2} in $\mathbf{ZF}$. Clearly, every Cantor-compact space is Tychonoff-compact, realcompact and $\mathbb{N}$-compact. If a topological space $\mathbf{X}$ is $\mathbb{N}$-compact, then $\mathbf{X}$ is a zero-dimensional $T_1$-space. If a topological space $\mathbf{X}$ is $\mathbf{E}$-compact for some Tychonoff space $\mathbf{E}$, then $\mathbf{X}$ is Tychonoff.  Every compact Tychonoff space is Tychonoff-compact, and every compact zero-dimensional $T_1$-space is Cantor-compact. It follows from Theorem \ref{s1:t30} that the statements ``Every Tychonoff-compact space is compact'' and ``Every Cantor-compact space is compact'' are both equivalent to $\mathbf{BPI}$, so independent of $\mathbf{ZF}$. In $\mathbf{ZF}$, a Hausdorff compact space may fail to be Tychonoff-compact because there exists a model of $\mathbf{ZF}$ in which a Hausdorff compact space need not be completely regular (cf., e.g., \cite{gt}). 

Using Theorem \ref{s1:t8}, one can easily deduce that the following proposition holds in $\mathbf{ZF}$:

\begin{proposition}
	\label{s4:p1}
	$[\mathbf{ZF}]$
	Let $\mathbf{X}$ and $\mathbf{E}$ be non-empty topological spaces. Then the following conditions are equivalent:
	\begin{enumerate}
		\item[(i)] $\mathbf{X}$ is $\mathbf{E}$-completely regular;
		\item[(ii)] $\mathcal{E}(\mathbf{X}, \mathbf{E})\neq\emptyset$;
		\item[(ii)]  $C(\mathbf{X},\mathbf{E})\in\mathcal{E}(\mathbf{X}, \mathbf{E})$. 
	\end{enumerate}
\end{proposition} 

Proposition \ref{s4:p1} can be regarded as a slight reformulation of \cite[Theorem 3.8]{mr2}. Our next theorem is a generalization of Theorem 3.11.3 in \cite{en}, having its roots already in \cite{sh},  but it has not been pointed out so far that this theorem holds in $\mathbf{ZF}$. Therefore, let us sketch its proof for completeness and to make our paper more self-contained.

\begin{theorem}
	\label{s4:t2}
	$[\mathbf{ZF}]$  Let $\mathbf{E}$ be a Hausdorff space and let $\mathbf{X}$ be an $\mathbf{E}$-completely regular space. 
	Then the following conditions are equivalent:
	\begin{enumerate}
		\item[(i)] $\mathbf{X}$ is $\mathbf{E}$-compact;
		\item[(ii)] for every Hausdorff space $\mathbf{Y}$ containing $\mathbf{X}$ as a dense subspace such that, for every $f\in C(\mathbf{X},\mathbf{E})$, there exists $\tilde{f}\in C(\mathbf{Y},\mathbf{E})$ with $\tilde{f}|_X=f$, the equality $\mathbf{X}=\mathbf{Y}$ holds.
	\end{enumerate}
\end{theorem}
\begin{proof}
	$(i)\rightarrow (ii)$ Suppose that $\mathbf{X}$ is $\mathbf{E}$-compact. There exists a non-empty set $J$ and a homeomorphic embedding $h$ of $\mathbf{X}$ into  $\mathbf{E}^{J}$ such that $h[X]$ is a closed subset of $\mathbf{E}^{J}$. Let $\mathbf{Y}$ be a Hausdorff space such that $\mathbf{X}$ is a dense subspace of $\mathbf{Y}$ and, for every $f\in C(\mathbf{X},\mathbf{E})$, there exists  $\tilde{f}\in C(\mathbf{Y},\mathbf{E})$ with $\tilde{f}|_X=f$. In view of Proposition \ref{s1:p12}, for every $j\in J$, there exists a unique $h_j\in C(\mathbf{Y}, \mathbf{E})$ such that, for every $x\in X$, $h_j(x)=\pi_j\circ h(x)$. Let $\mathcal{H}=\{h_j: j\in J\}$ and $e=e_{\mathcal{H}}$. Then $e\in C(\mathbf{Y}, \mathbf{E}^J)$. Hence $e[Y]=e[\cl_{\mathbf{Y}} X]\subseteq \cl_{\mathbf{E}^J} e[X]=\cl_{\mathbf{E}^J} h[X]=h[X]$. Therefore, we can consider the mapping $h^{-1}\circ e: Y\to X$. We notice that, for every $x\in X$, $h^{-1}(e(x))=x$. It follows from Proposition \ref{s1:p12} that, for every $y\in Y$,  $h^{-1}(e(y))=y$. This is possible only when $Y=X$. Hence (i) implies (ii).
	
	$(ii)\rightarrow (i)$ Now, let $J=C(\mathbf{X}, \mathbf{E})$ and let $\mathbf{Y}$ be the subspace of $\mathbf{E}^{J}$ with the underlying set $Y=\cl_{\mathbf{E}^J}e_J[X]$. By Proposition \ref{s4:p1}, the subspace $e_J[X]$ of $\mathbf{E}^J$ is homeomorphic with $\mathbf{X}$.  Since, for every mapping $f\in C(e_J[X], \mathbf{E})$, there exists a mapping $\tilde{f}\in C(\mathbf{Y}, \mathbf{E})$ such that, for every $x\in X$, $\tilde e_J(x)= f(e_J(x))$, it follows from (ii) that $Y=e_J[X]$. This shows that (ii) implies (i).
\end{proof}

\begin{corollary}
	\label{s4:c3}
	$[\mathbf{ZF}]$  Let $\mathbf{E}$ be a Hausdorff space and let $\mathbf{X}$ be an $\mathbf{E}$-completely regular space. Then $\mathbf{X}$ is $\mathbf{E}$-compact if and only if the set $e_{C(\mathbf{X}, \mathbf{E})}[X]$ is closed in $\mathbf{E}^{C(\mathbf{X}, \mathbf{E})}$. 
\end{corollary}

Let us write down several basic properties of $\mathbf{E}$-compactness, being relevant to the properties of realcompactness described in \cite[Chapter 3.11]{en}.

\begin{proposition}
	\label{s4:p4}
	$[\mathbf{ZF}]$ For every topological space $\mathbf{E}$, every closed subspace of an $\mathbf{E}$-compact space is $\mathbf{E}$-compact.
\end{proposition}

The following theorem is a modification of the $\mathbf{ZFC}$-theorem 3.11.5 of \cite{en} (cf. also \cite{sh}). In fact, Theorem 3.11.5 of \cite{en} is not a theorem of $\mathbf{ZF}$ for it is false in every model of $\mathbf{ZF}+\neg\mathbf{AC}$.

\begin{theorem}
	\label{s4:t5}
	$[\mathbf{ZF}]$
	Let $J$ be a non-empty set, $\{\mathbf{X}_j: j\in J\}$ a family of topological spaces,  and $\mathbf{E}$ a Hausdorff space. Let $\mathbf{X}=\prod\limits_{j\in J}\mathbf{X}_j$.
	\begin{enumerate}
		\item[(i)] If, for every $j\in J$, the space $\mathbf{X}_j$ is $\mathbf{E}$-compact, so is $\mathbf{X}$.
		\item[(ii)]  If $X=\prod\limits_{j\in J}X_j\neq\emptyset$, and $\mathbf{X}$ is $\mathbf{E}$-compact, then, for every $j\in J$, $\mathbf{X}_j$ is $\mathbf{E}$-compact. 
	\end{enumerate}
\end{theorem}
\begin{proof}
	(i) Assume that, for every $j\in J$, $\mathbf{X}_j$ is $\mathbf{E}$-compact. For $j\in J$, let $C_j=C(\mathbf{X}_j, \mathbf{E})$ and $e_j=e_{C_j}$ Then $e_j$ is a homeomorphic embedding of $\mathbf{X}_j$ into $\mathbf{E}^{C_j}$. Of course, for every $j\in J$, the space $\mathbf{X}_j$ is Hausdorff. We have shown in the proof of Theorem \ref{s4:t2} that, for every $j\in J$, the set $e_j[X_j]$ is closed in $\mathbf{E}^{C_j}$. This implies that the set $Y=\prod\limits_{j\in J}e_j[X_j]$ is closed in $\prod\limits_{j\in J}\mathbf{E}^{C_j}$. Since $\prod\limits_{j\in J}\mathbf{E}^{C_j}$ is $\mathbf{E}$-compact, it follows that $\mathbf{X}$ is $\mathbf{E}$-compact because $\mathbf{X}$ is homeomorphic with a closed subspace of $\prod\limits_{j\in J}\mathbf{E}^{C_j}$.\medskip
	
	(ii) Now, assume that $\mathbf{X}$ is an $\mathbf{E}$-compact space, and $X\neq\emptyset$.  Then $\mathbf{X}$ is Hausdorff. Since $X\neq\emptyset$ and $\mathbf{X}$ is a $T_1$-space, we deduce that, for every $j\in J$, the space $\mathbf{X}_j$ is homeomorphic with a closed subspace of $\mathbf{X}$, so $\mathbf{X}_j$ is $\mathbf{E}$-compact by Proposition \ref{s4:p4}.
\end{proof}

That products of realcompact spaces are realcompact in $\mathbf{ZFC}$ was noticed already in \cite{com} and  \cite{sal}.

In much the same way, as in \cite[proofs of Corollaries 3.11.17 and 2.11.18]{en}, one can prove the following propositions:

\begin{proposition}
	\label{s4:p6}
	$[\mathbf{ZF}]$ Let $J$ be a non-empty set and $\mathbf{E}$ a non-empty Hausdorff space. Suppose that $\{A_j: j\in J\}$ is a family of $\mathbf{E}$-compact subspaces of a topological space $\mathbf{X}$.  Then the subspace $A=\bigcap\limits_{j\in J}A_j$ of $\mathbf{X}$ is $\mathbf{E}$-compact.
\end{proposition}

\begin{proposition}
	\label{s4:p7}
	$[\mathbf{ZF}]$ Let $\mathbf{E}$ be a non-empty Hausdorff space. Let $\mathbf{X}$, $\mathbf{Y}$ be Hausdorff spaces and let $f\in C(\mathbf{X}, \mathbf{Y})$. If $\mathbf{X}$ is  $\mathbf{E}$-compact, then, for every $\mathbf{E}$-compact subspace $B$ of $\mathbf{Y}$, the subspace $f^{-1}[B]$ of $\mathbf{X}$ is $\mathbf{E}$-compact. 
\end{proposition}

We obtain the following result relevant to \cite[Theorem 3.11.16]{en}. We enclose a sketch of a proof of it in $\mathbf{ZF}$ because it differs from the proof of Theorem 3.11.16 given in \cite{en}. Our method is similar to that in \cite[the proof of Theorem 2.1]{mr2}.

\begin{theorem}
	\label{s4:t8}
	$[\mathbf{ZF}]$
	Let $\mathbf{E}$ be a non-empty Hausdorff space, and $\mathbf{X}$ an $\mathbf{E}$-completely regular space. Then $e_{C(\mathbf{X}, \mathbf{E})}\mathbf{X}$ is a Hewitt $\mathbf{E}$-compact extension of $\mathbf{X}$. Furthermore, for every Hewitt $\mathbf{E}$-compact extension $\langle \mathbf{Y}, h\rangle$ of $\mathbf{X}$, it holds that $\langle\mathbf{Y}, h\rangle\approx e_{C(\mathbf{X}, \mathbf{E})}\mathbf{X}$ and, for every $\mathbf{E}$-compact space $\mathbf{Z}$ and every $f\in C(\mathbf{X},\mathbf{Z})$, there exists $\tilde{f}\in C(\mathbf{Y}, \mathbf{Z})$ such that $\tilde{f}\circ h=f$.  
\end{theorem}
\begin{proof}
	By Definition \ref{s1:d5}, it is obvious that $e_{C(\mathbf{X}, \mathbf{E})}\mathbf{X}$ is a Hewitt $\mathbf{E}$-compact extension of $\mathbf{X}$. Let $\langle\mathbf{Y}, h\rangle$ be any Hewitt $\mathbf{E}$-compact extension of $\mathbf{X}$. By Proposition \ref{s4:p1} and Corollary \ref{s4:c3}, $e_{C(\mathbf{Y}, \mathbf{E})}$ is a homeomorphic embedding such that $e_{C(\mathbf{Y},\mathbf{E})}[Y]$ is closed in $\mathbf{E}^{C(\mathbf{Y}, \mathbf{E})}$. We define a mapping $\psi: E^{C(\mathbf{X}, \mathbf{E})}\to E^{C(\mathbf{Y},\mathbf{E})}$ as follows: for every $t\in E^{C(\mathbf{X}, \mathbf{E})}$ and every $f\in C(\mathbf{Y}, \mathbf{E})$, $\psi(t)(f)= t(f\circ h).$ Then $\psi$ is a homeomorphism of $\mathbf{E}^{C(\mathbf{X},\mathbf{E})}$ onto $\mathbf{E}^{C(\mathbf{Y},\mathbf{E})}$. Let $g: e_{C(\mathbf{X}, \mathbf{E})}X\to Y$ be defined by: for every $t\in e_{C(\mathbf{X}, \mathbf{E})}X$, $g(t)=e_{C(\mathbf{Y},\mathbf{E})}^{-1}\circ \psi(t)$. Then $g$ is the required homeomorphism of $e_{C(\mathbf{X}, \mathbf{E})}\mathbf{X}$ onto $\mathbf{Y}$ showing that $e_{C(\mathbf{X}, \mathbf{E})}\mathbf{X}\approx\langle \mathbf{Y}, h\rangle$.
	
	Now, let $\mathbf{Z}$ be an $\mathbf{E}$-compact space and let $f\in C(\mathbf{X},\mathbf{Z})$. By Proposition \ref{s4:p1} and Corollary \ref{s4:c3}, $e_{C(\mathbf{Z}, \mathbf{E})}$ is a homeomorphic embedding such that $e_{C(\mathbf{Z}, \mathbf{E})}[Z]$ is closed in $\mathbf{E}^{C(\mathbf{Z},\mathbf{E})}$. For every $j\in C(\mathbf{Z}, \mathbf{E})$, let $f_j=\pi_j\circ e_{C(\mathbf{Z}, \mathbf{E})}\circ f$  and let $\tilde{f}_j\in C(\mathbf{Y}, \mathbf{E})$ be the continuous extension of $f_j$ over $\mathbf{Y}$. We define $\tilde{f}: Y\to Z$ as follows: for every $y\in Y$, $\tilde{f}(y)=e_{C(\mathbf{Z},\mathbf{E})}^{-1}\circ e_{\{\tilde{f}_j: j\in C(\mathbf{Z}, \mathbf{E})\}}$ (see \cite[Lemma 3.11.15]{en}). Then $\tilde{f}\in C(\mathbf{Y},\mathbf{Z})$ and $\tilde{f}\circ h=f$.
\end{proof}

\begin{remark}
	\label{s4:r9}
	Let us recall that, for a Hausdorff space $\mathbf{E}$ and an $\mathbf{E}$-completely regular space $\mathbf{X}$, we denote by $v_{\mathbf{E}}\mathbf{X}$ any Hewitt $\mathbf{E}$-compact extension of $\mathbf{X}$ for, in view of Theorem \ref{s4:t8}, such an extension is unique (up to $\approx$); moreover, if $v_{\mathbf{E}}\mathbf{X}=\langle \mathbf{Y}, h\rangle$, for simplicity, we identify $v_{\mathbf{E}}\mathbf{X}$ with $\mathbf{Y}$, $\mathbf{X}$ with the subspace $h[X]$ of $\mathbf{Y}$, and $h$ with the identity map $\text{id}_{X}$. 
\end{remark}

\begin{corollary}
	\label{s4:c10}
	$[\mathbf{ZF}]$ Let $\mathbf{E}$ be a Hausdorff space. Then an $\mathbf{E}$-completely regular space $\mathbf{X}$ is $\mathbf{E}$-compact if and only if $\mathbf{X}=v_{\mathbf{E}}\mathbf{X}$.
\end{corollary}

A well-known theorem of $\mathbf{ZFC}$ asserts that a Tychonoff space $\mathbf{X}$ is realcompact if and only if, for every $x_0\in\beta X\setminus X$, there exists $f\in C(\beta\mathbf{X}, [0,1])$ such that $x_0\in Z(f)$ and $Z(f)\cap X=\emptyset$ (see \cite[Theorem 3.11.10]{en}). To get, in $\mathbf{ZF}$, a modification and a generalization of this theorem, we introduce a new concept of a compactly $\mathbf{E}$-Urysohn space in the following definition:

\begin{definition}
	\label{s4:d11}
	Let $\mathbf{X}$ and $\mathbf{E}$ be Hausdorff spaces. We say that $\mathbf{X}$ is compactly $\mathbf{E}$-Urysohn if the following condition is satisfied: for every pair $A, B$ of disjoint closed sets in $\mathbf{X}$, if at least one of the sets $A, B$ is compact in $\mathbf{X}$, then there exists $f\in C^{\ast}(\mathbf{X}, \mathbf{E})$ such that $\cl_{\mathbf{E}}(f[A])\cap\cl_{\mathbf{E}}(f[B])=\emptyset$. 
\end{definition}

In what follows, for a locally compact non-compact Hausdorff space $\mathbf{E}$, we fix an element $\infty\notin E$ and, as in Definition \ref{s1:d32}, denote by $\mathbf{E}(\infty)$ the Alexandroff compactification of $\mathbf{E}$ with $E(\infty)\setminus E=\{\infty\}$.

One can easily observe that the following proposition holds:
\begin{proposition}
	\label{s4:p12}
	$[\mathbf{ZF}]$
	\begin{enumerate}
		\item[(i)]  Every discrete or compact space $\mathbf{E}$ is compactly $\mathbf{E}$-Urysohn. In particular, the discrete space $\mathbb{N}$ is compactly $\mathbb{N}$-Urysohn.
		\item[(ii)] A Hausdorff space $\mathbf{X}$ is compactly $\mathbb{R}$-Urysohn if and only if $\mathbf{X}$ is completely regular. In particular, $\mathbb{R}$ is compactly $\mathbb{R}$-Urysohn.
		\item[(iii)] Let $\mathbf{E}$ be a non-compact locally compact Hausdorff space. Then $\mathbf{E}$ is compactly $\mathbf{E}(\infty)$-Urysohn.
	\end{enumerate}
\end{proposition}

\begin{remark}
	\label{s4:r13}
	That a metrizable space $\mathbf{X}$  may fail to be $\mathbf{X}$-Urysohn can be shown by using De Groot's example of a dense subspace $\mathbf{G}$ of the plane $\mathbb{R}^2$ such that every $f\in C(\mathbf{G},\mathbf{G})$ is either the identity $\text{id}_{G}$ or a constant mapping (see \cite{DG}). Then $\mathbf{G}$ is not compactly $\mathbf{G}$-Urysohn. 
	
	At this moment, we do not have an example of a locally compact Hausdorff space $\mathbf{E}$ which is not compactly $\mathbf{E}$-Urysohn. 
\end{remark}

\begin{theorem}
	\label{s4:t14}
	$[\mathbf{ZF}]$ Let $\mathbf{E}$ be a non-compact, locally compact Hausdorff space which is compactly $\mathbf{E}$-Urysohn. Suppose that $\mathbf{X}$ is a topological space such that $C^{\ast}(\mathbf{X}, \mathbf{E})\in\mathcal{E}(\mathbf{X}, \mathbf{E})$. Then, for every $f\in C(\mathbf{X}, \mathbf{E})$, there exists a unique $\tilde{f}\in C(e_{C^{\ast}(\mathbf{X}, \mathbf{E})}\mathbf{X}, \mathbf{E}(\infty))$ such that $\tilde{f}\circ e_{C^{\ast}(\mathbf{X}, \mathbf{E})}=f$. Furthermore, if $Y=\bigcap\{\tilde{f}^{-1}[E]: f\in C(\mathbf{X}, \mathbf{E})\}$, then the subspace $\mathbf{Y}$ of $e_{C^{\ast}(\mathbf{X},\mathbf{E})}\mathbf{X}$ is $\mathbf{E}$-compact, and $\langle \mathbf{Y}, e_{C^{\ast}(\mathbf{X}, \mathbf{E})}\rangle\approx v_{\mathbf{E}}\mathbf{X}$.
\end{theorem}

\begin{proof}
	Let $e^{\ast}=e_{C^{\ast}(\mathbf{X}, \mathbf{E})}$. We fix $f\in C(\mathbf{X}, \mathbf{E})$ and a pair $C_1, C_2$ of disjoint closed sets of $\mathbf{E}(\infty)$. For $i\in\{1,2\}$, we put $A_i=C_i\cap E$. At least one of the sets $A_1, A_2$ is compact. Since $\mathbf{E}$ is compactly $\mathbf{E}$-Urysohn, there exists $g\in C^{\ast}(\mathbf{E}, \mathbf{E})$ such that $\cl_{\mathbf{E}}(g[A_1])\cap\cl_{\mathbf{E}}(g[A_2])=\emptyset$. Since $g\circ f\in C^{\ast}(\mathbf{X},\mathbf{E})$, there exists $\psi\in C(e^{\ast}\mathbf{X}, \mathbf{E})$ such that $\psi\circ e^{\ast}=g\circ f$.  We notice that, for $i\in\{1,2\}$,  $e^{\ast}[f^{-1}[C_i]]\subseteq \psi^{-1}[\cl_{\mathbf{E}}g[A_i]]$. This implies  that $\cl_{e^{\ast}\mathbf{X}}(e^{\ast}[f^{-1}[C_1]])\cap \cl_{e^{\ast}\mathbf{X}}(e^{\ast}[f^{-1}[C_2]])=\emptyset$. By Theorem \ref{s1:t10}, there exists a function $\tilde{f}\in C(e^{\ast}\mathbf{X}, \mathbf{E}(\infty))$ such that $\tilde{f}\circ e^{\ast}=f$. It follows from Propositions \ref{s4:p6} and \ref{s4:p7} that the space $\mathbf{Y}$ is $\mathbf{E}$-compact. This, together with Definition \ref{s1:d5}, shows that $\langle \mathbf{Y}, e^{\ast}\rangle \approx v_{\mathbf{E}}\mathbf{X}$. 
\end{proof}

\begin{corollary}
	\label{s4:c15}
	$[\mathbf{ZF}]$ Let $\mathbf{E}$ be a non-compact, locally compact Hausdorff space which is compactly $\mathbf{E}$-Urysohn. Then a topological space $\mathbf{X}$ such that $C^{\ast}(\mathbf{X}, \mathbf{E})\in\mathcal{E}(\mathbf{X}, \mathbf{E})$ is $\mathbf{E}$-compact if and only if, for every $p\in e_{C^{\ast}(\mathbf{X}, \mathbf{E})}X\setminus e_{C^{\ast}(\mathbf{X}, \mathbf{E})}[X]$, there exists $f\in C(\mathbf{X}, \mathbf{E})$ such that $\tilde{f}(p)=\infty$ where $\tilde{f}\in C(e_{C^{\ast}(\mathbf{X}, \mathbf{E})}\mathbf{X}, \mathbf{E}(\infty))$ is such that  $\tilde{f}\circ e_{C^{\ast}(\mathbf{X}, \mathbf{E})}=f$.
\end{corollary}

\begin{theorem}
	\label{s4:t16}
	$[\mathbf{ZF}]$ 
	\begin{enumerate}
		\item[(i)]  A Tychonoff space $\mathbf{X}$ is realcompact if and only if, for every $p\in \beta^{f}X\setminus e_{\beta^f}[X]$, there exists $f\in C^{\ast}(\beta^f\mathbf{X})$ such that $f(p)=0$ and, for every $x\in X$, $f(e_{\beta^f}(x))>0$.
		\item[(ii)] Let $N=\{\frac{1}{n}: n\in\mathbb{N}\}$ and $N_0=\{0\}\cup N$ be considered as subspaces of $\mathbb{R}$. A zero-dimensional $T_1$-space $\mathbf{X}$ is $\mathbb{N}$-compact if and only if, for every $p\in e^{\ast}_{2}X\setminus e^{\ast}_{2}[X]$, there exists $f\in C(e^{\ast}_{2}\mathbf{X}, N_0)$ such that $f(p)=0$ and, for every $x\in X$, $f(e^{\ast}_{2}(x))>0$.
	\end{enumerate}
\end{theorem}
\begin{proof}
	That (i) holds, can be proved in much the same way, as Theorem 3.11.10 in \cite{en}.
	
	(ii) Let $\mathbf{X}$ be a zero-dimensional $T_1$-space. By Theorem \ref{s2:t10}, $e^{\ast}_{2}\mathbf{X}\approx e^{\ast}_{N}\mathbf{X}$. Clearly, $N_0$ is the one-point Hausdorff compactification of $N$. To conclude the proof of (ii), it suffices to apply Theorem \ref{s4:t14} and Proposition \ref{s4:p12}(i).
\end{proof}

\section{ $U_{\aleph_0}(\mathbf{X})$, $v_{\mathbf{E}}\mathbf{X}$ and strong zero-dimensionality}
\label{s5}

We know from Theorem \ref{s2:t6}(i) that it holds in $\mathbf{ZF}$ that a Tychonoff space $\mathbf{X}$ is strongly zero-dimensional if and only if its extension $\beta^{f}\mathbf{X}$ is strongly zero-dimensional. Let us show in the following two theorems that, to some extent, an analogous result about Hewitt $\mathbf{E}$-compact extensions can be obtained in $\mathbf{ZF}$. 

\begin{theorem}
	\label{s5:t1}
	$[\mathbf{ZF}]$ Let $\mathbf{E}$ be a Tychonoff space consisting of at least two points. Suppose that $\mathbf{X}$ is a strongly zero-dimensional $\mathbf{E}$-completely regular space. Then $v_{\mathbf{E}}\mathbf{X}$ is strongly zero-dimensional. In particular, for every strongly zero-dimensional $T_1$-space $\mathbf{Y}$, the space $v_{\mathbb{N}}\mathbf{Y}$ is strongly zero-dimensional.
\end{theorem}
\begin{proof}
	Let us slightly modify the proof of Theorem 6.2.12 in \cite{en}. Let $Z_0, Z_1\in\mathcal{Z}(v_{\mathbf{E}}\mathbf{X})$ and $Z_1\cap Z_2=\emptyset$. There exists $f\in C(v_{\mathbf{E}}\mathbf{X})$ such that, for every $i\in\{0, 1\}$,  $Z_i\subseteq f^{-1}[\{i\}]$ and $f[v_{\mathbf{E}}\mathbf{X}]\subseteq [0, 1]$. Since $\mathbf{X}$ is strongly zero-dimensional, there exists $g\in C(\mathbf{X})$ such that $g[X]\subseteq\{0,1\}$ and $X\cap f^{-1}[[0, 1/3]]\subseteq g^{-1}[\{0\}]\subseteq f^{-1}[[0, 2/3)]$. Since the discrete space $\mathbf{2}$ is $\mathbf{E}$-compact, there exists $\tilde{g}\in C(v_{\mathbf{E}}\mathbf{X})$ such that $\tilde{g}|X=g$ and $\tilde{g}[v_{\mathbf{E}}\mathbf{X}]\subseteq\{0, 1\}$. The set $U=\tilde{g}^{-1}[\{0\}]$ is clopen in $v_{\mathbf{E}}\mathbf{X}$, $Z_0\subseteq U$ and $U\cap Z_1=\emptyset$. 
\end{proof}

\begin{theorem}
	\label{s5:t2}
	$[\mathbf{ZF}]$ 
	Let $\mathbf{E}$ be a Tychonoff space such that $\mathbb{R}$ is $\mathbf{E}$-compact.  Then, for every $\mathbf{E}$-completely regular space $\mathbf{X}$,  the following conditions are equivalent:
	\begin{enumerate}
		\item[(i)] $\mathbf{X}$ is strongly zero-dimensional.
		\item[(ii)] $v_{\mathbf{E}}\mathbf{X}$ is strongly zero-dimensional. 
	\end{enumerate}
\end{theorem}
\begin{proof}
	That (i) implies (ii) follows from Theorem \ref{s5:t1}. That (ii) implies (i) can be deduced from \cite[Theorem 6.2.11]{en}.
\end{proof}

\begin{remark}
	\label{s5:r3}
	In \cite{ny}, P. Nyikos proved in $\mathbf{ZFC}$ that there exists a realcompact,  zero-dimensional space $\mathbf{X}$ such that $v_{\mathbb{N}}\mathbf{X}$ is strongly zero-dimensional but $\mathbf{X}$ is not strongly zero-dimensional. This result due to P. Nyikos was discussed also in \cite{mr4} and \cite{mr5}.  
\end{remark}

In view of Remark \ref{s1:r20}, for every non-empty topological space $\mathbf{X}$, $U_{\aleph_0}(\mathbf{X})$ is a subring of the ring $C(\mathbf{X})$. In the light of Theorem \ref{s1:t21}, for a non-empty $\mathbf{E}$-completely regular space $\mathbf{X}$, it is reasonable to have a look at a relationship between the rings $U_{\aleph_0}(v_{\mathbf{E}}\mathbf{X})$ and $U_{\aleph_0}(\mathbf{X})$ to get other natural conditions under which $\mathbf{X}$ is strongly zero-dimensional or a $P$-space. The following results give such conditions.

\begin{theorem}
	\label{s5:t4}
	$[\mathbf{ZF}]$ 
	Let $\mathbf{E}$ be a Hausdorff space such that $\mathbb{R}_{disc}$ is $\mathbf{E}$-compact. Let $\mathbf{X}$ be a non-empty $\mathbf{E}$-completely regular space. Then $\mathbf{CMC}$ implies that the following conditions are satisfied: 
	\begin{enumerate}
		\item[(i)] $U_{\aleph_0}(\mathbf{X})=\{f\upharpoonright X: f\in U_{\aleph_0}(v_{\mathbf{E}}\mathbf{X})\}$; in consequence, the rings $U_{\aleph_0}(\mathbf{X})$ and $U_{\aleph_0}(v_{\mathbf{E}}\mathbf{X})$ are isomorphic;
		\item[(ii)] $\mathbf{X}$ is strongly zero-dimensional if and only if $$C(\mathbf{X})=\{f\upharpoonright X: f\in U_{\aleph_0}(v_{\mathbf{E}}\mathbf{X})\}.$$
	\end{enumerate}
\end{theorem}
\begin{proof}
	It is obvious that the inclusion $\{f\upharpoonright X: f\in U_{\aleph_0}(v_{\mathbf{E}}\mathbf{X})\}\subseteq U_{\aleph_0}(\mathbf{X})$ holds in $\mathbf{ZF}$. For the rest of the proof, we assume $\mathbf{ZF+CMC}$.
	
	(i) To show that $U_{\aleph_0}(\mathbf{X})\subseteq \{f\upharpoonright X: f\in U_{\aleph_0}(v_{\mathbf{E}}\mathbf{X})\}$, let us fix $f\in U_{\aleph_0}(\mathbf{X})$. It follows from Theorem \ref{s1:t21}(c) that $f\in \cl_{C_{u}(\mathbf{X})}(C(\mathbf{X}, \mathbb{R}_{disc}))$. Thus, by $\mathbf{CMC}$, there exists a sequence $(f_n)_{n\in\mathbb{N}}$ of functions from $C(\mathbf{X}, \mathbb{R}_{disc})$ which converges in $C_{u}(\mathbf{X})$ to $f$ (see \cite[the proof of Theorem 2.10]{kow}).  Since $\mathbb{R}_{disc}$ is $\mathbf{E}$-compact, for every $n\in\mathbb{N}$, there exists a unique $\tilde{f}_n\in C(v_{\mathbf{E}}\mathbf{X}, \mathbb{R}_{disc})$ such that $f=\tilde{f}\upharpoonright X$. Let $p\in v_{\mathbf{E}} X$. To show that the sequence $(\tilde{f}_n(p))_{n\in\mathbb{N}}$ is a Cauchy sequence in $\mathbb{R}$, we take any real number $\varepsilon>0$ and choose $n_0\in\mathbb{N}$ such that, for every $n\in\mathbb{N}$ with $n\geq n_0$ and for every $x\in X$, $|f_n(x)-f(x)|<\frac{\varepsilon}{2}$. Let $m,n\in\mathbb{N}$ and $n,m>n_0$. By the continuity of $\tilde{f}_n, \tilde{f}_m$, there exists an open neighborhood $V$ of $p$ in $v_{\mathbf{E}}\mathbf{X}$ such that, for every $y\in V$, $\tilde{f}_m(p)=\tilde{f}_m(y)$ and $\tilde{f}_n(p)=\tilde{f}_n(y)$. There exists $x_0\in V\cap X$. We have $|\tilde{f}_m(p)-\tilde{f}_n(p)|\leq |\tilde{f}_m(p)-\tilde{f}_m(x_0)|+|\tilde{f}_m(x_0)-\tilde{f}_n(x_0)|+|\tilde{f}_n(x_0)-\tilde{f}_n(p)|=|f_n(x_0)-f_m(x_0)|\leq |f_n(x_0)-f(x_0)|+|f(x_0)-f_m(x_0)|<\varepsilon$. This proves that $(\tilde{f}_n(p))_{n\in\mathbb{N}}$ is a Cauchy sequence in $\mathbb{R}$. We define a function $\tilde{f}: v_{\mathbf{E}}X\to\mathbb{R}$ by putting $\tilde{f}(p)=\lim\limits_{n\to+\infty}\tilde{f}_n(p)$ for every $p\in v_{\mathbf{E}}X$.  Then $\tilde{f}\in \cl_{C_{u}(v_{\mathbf{E}}\mathbf{X})}C(v_\mathbf{E}\mathbf{X}, \mathbb{R}_{disc})$, so $\tilde{f}\in U_{\aleph_0}(v_{\mathbf{E}}\mathbf{X})$ by Theorem \ref{s1:t21}(c). It is easily seen that $\tilde{f}\upharpoonright X=f$. Hence $\mathbf{CMC}$ implies (i) in $\mathbf{ZF}$.
	
	(ii) Now, suppose that $\mathbf{X}$ is strongly zero-dimensional. By Theorem \ref{s1:t21}, $C(\mathbf{X})=U_{\aleph_0}(\mathbf{X})$. This, together with  (i), implies that $C(\mathbf{X})=\{f\upharpoonright X: f\in U_{\aleph_0}(v_{\mathbb{R}}\mathbf{X})\}$. On the other hand, if $\mathbf{X}$ is an $\mathbf{E}$-completely regular space for which $C(\mathbf{X})=\{f\upharpoonright X: f\in U_{\aleph_0}(v_{\mathbf{E}}\mathbf{X})\}$, it follows from (i) that $C(\mathbf{X})= U_{\aleph_0}(\mathbf{X})$, so $\mathbf{X}$ is strongly zero-dimensional by Theorem \ref{s1:t21}(c).
\end{proof}

Let us notice that, by \cite[the hint to Exercise 3.1 H(a)]{en}, the following proposition holds:

\begin{proposition}
	\label{s5:p5}
	$[\mathbf{ZF}]$ $\mathbb{R}_{disc}$ is $\mathbb{N}$-compact, so also $\mathbb{R}$-compact.
\end{proposition}

\begin{theorem}
	\label{s5:t6}
	In $\mathbf{ZF}$,  $\mathbf{CMC}$ implies that, for every Tychonoff space $\mathbf{X}$, the following conditions are equivalent:
	\begin{enumerate}
		\item[(i)] $\mathbf{X}$ is strongly zero-dimensional;
		\item[(ii)] $C(\mathbf{X})=\{f\upharpoonright X: f\in U_{\aleph_0}(v_{\mathbb{R}}\mathbf{X})\}$;
		\item[(iii)] $\mathbf{X}$ is zero-dimensional and $C(\mathbf{X})=\{f\upharpoonright X: f\in U_{\aleph_0}(v_{\mathbb{N}}\mathbf{X})\}$.
	\end{enumerate}
\end{theorem}
\begin{proof}
	This follows directly from Proposition \ref{s5:p5}, Theorems \ref{s5:t4}, \ref{s1:t21}(c) and the fact that every zero-dimensional $T_1$- space is $\mathbb{N}$-completely regular.
\end{proof}

\begin{theorem}
	\label{s5:t7}
	$[\mathbf{ZF}]$ Let $\mathbf{E}$ be a Tychonoff space such that $\mathbb{R}_{disc}$ is $\mathbf{E}$-compact. Let $\mathbf{X}$ be an $\mathbf{E}$-completely regular $P$-space. Then $\mathbf{CMC}$ implies that $v_{\mathbf{E}}\mathbf{X}$ is a $P$-space.
\end{theorem}
\begin{proof}
	Let us assume that $\mathbf{CMC}$ holds. In view of Theorem \ref{s1:t21}(c), to prove that $v_{\mathbf{E}}\mathbf{X}$ is a $P$-space, it suffices to check that $C(v_{\mathbf{E}}\mathbf{X}, \mathbb{R}_{disc})=C(v_{\mathbf{E}}\mathbf{X})$. Since $\mathbb{R}_{disc}$ is $\mathbf{E}$-compact, $C(\mathbf{X}, \mathbb{R}_{disc})=\{f\upharpoonright X: f\in C(v_{\mathbf{E}}\mathbf{X}, \mathbb{R}_{disc})\}$. Let $f\in C(v_{\mathbf{E}}\mathbf{X})$. Then, since $\mathbf{X}$ is a P-space, it follows from Theorem \ref{s1:t21}(c) that $f\upharpoonright X\in C(\mathbf{X}, \mathbb{R}_{disc})$. Hence, since $\mathbb{R}_{disc}$ is $\mathbf{E}$-compact, there exists $g\in C(v_{\mathbf{E}}\mathbf{X}, \mathbb{R}_{disc})$ with $g\upharpoonright X=f\upharpoonright X$. Then $g=f$ by Proposition \ref{s1:p12} and, in consequence,  $C(v_{\mathbf{E}}\mathbf{X}, \mathbb{R}_{disc})=C(v_{\mathbf{E}}\mathbf{X})$.
\end{proof}

A hard open problem seems to be if $\mathbf{CMC}$ can be omitted or replaced by a weaker assumption in Theorems \ref{s5:t4}, \ref{s5:t6} and \ref{s5:t7}. In particular, the following open problems are interesting:

\begin{problem}
	\label{s5:q8} 
	Let $\mathbf{E}$ be a Tychonoff space such that $\mathbb{R}_{disc}$ is $\mathbf{E}$-compact. 

\begin{enumerate}
	\item[(i)]  If $\mathbf{X}$ is an $\mathbf{E}$-completely regular $P$-space, may $v_{\mathbf{E}}\mathbf{X}$ fail to be a $P$-space in a model of $\mathbf{ZF}$?
	\item[(ii)] If $\mathbf{X}$ is an  $\mathbf{E}$-completely regular space, may the rings $U_{\aleph_0}(\mathbf{X})$ and $U_{\aleph_0}(v_{\mathbf{E}}\mathbf{X})$ fail to be isomorphic in a model of $\mathbf{ZF}$?
\end{enumerate}
\end{problem}
\section{$\mathbb{N}$-compactness and realcompactness via filters}
\label{s6}

In \cite[Theorem A]{Chew}, Kim-Peu Chew gave a very useful characterization of $\mathbb{N}$-compact spaces in $\mathbf{ZFC}$ by showing that it holds in $\mathbf{ZFC}$ that a zero-dimensional $T_1$-space $\mathbf{X}$ is $\mathbb{N}$-compact if and only if every clopen ultrafilter in $\mathbf{X}$ with the countable intersection property is fixed. To prove the necessity of this condition,  Kim-Peu Chew used in \cite{Chew} the statement that, for every non-empty set $J$, the product $\mathbb{N}(\infty)^J$ is compact. However, in $\mathbf{ZF}$, the space  $\mathbb{N}(\infty)^J$ may fail to be compact, so Chew's proof of Theorem A of \cite{Chew} is incorrect in $\mathbf{ZF}$, but it is easily seen that it remains correct in $\mathbf{ZF+BPI}$. Let us modify Chew's proof to get a correct $\mathbf{ZF}$-proof of Theorem A of \cite{Chew}.

\begin{theorem}
	\label{s6:t1}
	$[\mathbf{ZF}]$ A zero-dimensional $T_1$-space $\mathbf{X}$ is $\mathbb{N}$-compact if and only if every clopen ultrafilter in $\mathbf{X}$ with c.i.p. is fixed. (Cf. \cite[Theorem A]{Chew}.)
\end{theorem}

\begin{proof}
	Let $\mathbf{X}$ be a zero-dimensional $T_1$-space. Chew's proof given in \cite{Chew} that if every clopen ultrafilter in $\mathbf{X}$ with c.i.p. is fixed, then $\mathbf{X}$ is $\mathbb{N}$-compact is correct also in $\mathbf{ZF}$. So, let us prove in $\mathbf{ZF}$ that if $\mathbf{X}$ is $\mathbb{N}$-compact, then every clopen ultrafilter in $\mathbf{X}$ with c.i.p. is fixed. 
	
	Suppose that $\mathcal{U}$ is a free clopen ultrafilter on an $\mathbb{N}$-compact space $\mathbf{X}$. We are going to show in $\mathbf{ZF}$ that $\mathcal{U}$ does not have the countable intersection property. To this aim, we put $J= C(\mathbf{X}, \mathbb{N}(\infty))$, $\mathbf{Y}=\mathbb{N}(\infty)^{J}$ and $P=\cl_{\mathbf{Y}}e_J[X]$. Then, by Theorem \ref{s4:t8}, for the subspace $\mathbf{P}$ of $\mathbf{Y}$, we  have $\mathbf{P}=v_{\mathbb{N}(\infty)}\mathbf{X}$. Since $C(\mathbf{X}, \mathbf{2})=\{f\circ e_J: f\in C(\mathbf{P}, \mathbf{2})\}$, for every $U\in\mathcal{U}$, the set $\cl_{\mathbf{P}}e_J[U]$ is clopen in $\mathbf{P}$. Let $\mathcal{U}^{\ast}=\{\cl_{\mathbf{P}}e_J[U]: U\in\mathcal{U}\}$. Then $\mathcal{U}^{\ast}$ is a clopen filter on $\mathbf{P}$. If $A\in\mathcal{CO}(\mathbf{P})$ and $A\notin\mathcal{U}^{\ast}$, then there exists $U_A\in\mathcal{U}$ such that $U_A\cap e_J^{-1}[A]=\emptyset$. Then $A\cap\cl_{\mathbf{P}}e_J[U_A]=\emptyset$. This implies that $\mathcal{U}^{\ast}$ is an ultrafilter in $\mathcal{CO}(\mathbf{P})$. Now, let us modify the standard proof of the Tychonoff Product Theorem.  
	
	For every $j\in J$, let $\mathcal{U}_j=\{\cl_{\mathbb{N}(\infty)}\pi_j[U]: U\in\mathcal{U}^{\ast}\}$. Then, for every $j\in J$,  $\mathcal{U}_j$ has the finite intersection property.  Since $\mathbb{N}(\infty)$ is compact and well-ordered, we can fix $p\in\mathbf{Y}$ such that, for every $j\in J$, $p(j)\in\bigcap\mathcal{U}_j$. Suppose that there exists $U_0\in\mathcal{U}^{\ast}$ such that $p\notin U_0$. There exists a family $\{V_j: j\in J\}$ of clopen sets of $\mathbb{N}(\infty)$ such that the set $K=\{j\in J: V_j\neq \mathbb{N}(\infty)\}$ is finite and, for $V=\prod\limits_{j\in J}V_j$, we have $p\in V\subseteq Y\setminus U_0$. Clearly, $K\neq\emptyset$ because $U_0\neq\emptyset$. For every $j\in K$ and $U\in\mathcal{U}^{\ast}$, we have $U\cap \pi_j^{-1}[V_j]\neq\emptyset$, so $\pi_j^{-1}[V_j]\cap P\in\mathcal{U}^{\ast}$ because $\mathcal{U}^{\ast}$ is a clopen ultrafilter in $\mathbf{P}$. This implies that $V\cap P\in\mathcal{U}^{\ast}$ which is impossible. The contradiction obtained proves that $p\in\bigcap\mathcal{U}^{\ast}$. Since $\mathcal{U}$ is free, we deduce that $p\notin e_J[X]$. This, together with the $\mathbb{N}$-compactness of $\mathbf{X}$ and Corollary \ref{s4:c15}, implies that there exists $f\in C(\mathbf{P}, \mathbb{N}(\infty))$ such that $f(p)=\infty$. In much the same way, as in the proof of Theorem A in \cite{Chew}, for every $n\in\mathbb{N}$, we consider the set $V_n=\mathbb{N}(\infty)\setminus\{i\in\mathbb{N}: i\leq n\}$ and $W_n=f^{-1}[V_n]$. We observe that, for every $n\in\mathbb{N}$,  $e_J^{-1}[W_n]\in\mathcal{U}$ but $\bigcap\limits_{n\in\mathbb{N}}e_J^{-1}[W_n]=\emptyset$, so $\mathcal{U}$ does not have the countable intersection property. 
\end{proof}

Let us recall the principle $\mathbf{A}(\delta\delta)$ defined in \cite{kow}.

\begin{definition}
	\label{s6:d2}
	(Cf. \cite[Definition 1.7(4)]{kow}.) $\mathbf{A}(\delta\delta)$ is the statement: For every set $X$ and every family $\mathcal{A}$ of subsets of $X$ such that $\mathcal{A}$ is closed under finite unions and finite intersections, for every family $\{\mathcal{U}_n: n\in\omega\}$ of non-empty countable subfamilies of $\mathcal{A}$, there exists a family $\{V_{m,n}: m,n\in\omega\}$ of members of $\mathcal{A}$ such that, for every $n\in\omega$, $\bigcap\mathcal{U}_n=\bigcap\limits_{m\in\omega}V_{m,n}$. 
\end{definition}

\begin{proposition}
	\label{s6:p3}
	$($Cf. \cite[Theorem 2.9]{kow}.$)$ The following implications are true in $\mathbf{ZF}$:
	$$\mathbf{CMC}\rightarrow \mathbf{A}(\delta\delta)\rightarrow \mathbf{CMC}_{\omega}.$$
\end{proposition}

\begin{lemma}
	\label{s6:l4}
	$[\mathbf{ZF}]$
	Let $\mathcal{A}$ be a non-empty family of subsets of a set $X$ such that $\mathcal{A}$ is closed under finite unions and finite intersections. Then $\mathbf{CMC}$ implies that $\mathcal{A}_{\delta\delta}=\mathcal{A}_{\delta}$. 
\end{lemma}
\begin{proof}
	Assuming $\mathbf{CMC}$, it suffices to show that $\mathcal{A}_{\delta\delta}\subseteq\mathcal{A}_{\delta}$. Suppose that, for every $n\in\omega$, $A_n\in\mathcal{A}_{\delta}$ and $A=\bigcap\limits_{n\in\omega}A_n$. For every $n\in\omega$, let $\mathcal{B}_n=\{ \mathcal{B}\in [\mathcal{A}]^{\leq\omega}\setminus\{\emptyset\}: A_n=\bigcap\mathcal{B}\}$. By $\mathbf{CMC}$, there exists a family $\{\mathcal{F}_n: n\in\omega\}$ such that, for every $n\in\omega$, $\mathcal{F}_n$ is a non-empty finite subfamily of $\mathcal{B}_n$. For every $n\in\omega$, let $\mathcal{U}_n=\bigcup\mathcal{F}_n$. Then, for every $n\in\omega$, $\mathcal{U}_n\in [\mathcal{A}]^{\leq\omega}\setminus\{\emptyset\}$ and $A_n=\bigcap\mathcal{U}_n$. By $\mathbf{A}(\delta\delta)$, there exists a family $\{V_{m,n}: m,n\in\omega\}$ of members of $\mathcal{A}$ such that, for every $n\in\omega$, $A_n=\bigcap\limits_{m\in\omega}V_{m,n}$. Then $A=\bigcap\limits_{n\in\omega}A_n\in \mathcal{A}_{\delta}$.
\end{proof}

\begin{theorem}
	\label{s6:t5}
	$[\mathbf{ZF}]$ Let $\mathbf{X}$ be zero-dimensional $T_1$-space. Then the following conditions are satisfied:
	\begin{enumerate}
		\item[(i)] if $\mathbf{X}$ is  $\mathbb{N}$-compact, then every ultrafilter in $\mathcal{CO}_{\delta}(\mathbf{X})$ with c.i.p. is fixed;
		\item[(ii)] $\mathbf{CMC}$ implies that if every ultrafilter in $\mathcal{CO}_{\delta}(\mathbf{X})$ with c.i.p. is fixed, then $\mathbf{X}$ is $\mathbb{N}$-compact.
	\end{enumerate}
\end{theorem}
\begin{proof}
	(i) Suppose that $\mathbf{X}$ is $\mathbb{N}$-compact. Similarly to the proof of Theorem \ref{s6:t1}, we put $J= C(\mathbf{X}, \mathbb{N}(\infty))$, $\mathbf{Y}=\mathbb{N}(\infty)^{J}$ and $P=\cl_{\mathbf{Y}}e_J[X]$. For the subspace $\mathbf{P}$ of $\mathbf{Y}$, we  have $\mathbf{P}=v_{\mathbb{N}(\infty)}\mathbf{X}$. 
	
	Suppose that  $\mathcal{U}$ is a free ultrafilter with c.i.p. in $\mathcal{CO}_{\delta}(\mathbf{X})$.  Let $\mathcal{U}^{\ast}=\{V\in \mathcal{CO}_{\delta}(\mathbf{P}): (\exists U\in\mathcal{U}): e_J[U]\subseteq V\}$. Then $\mathcal{U}^{\ast}$ is a filter in $\mathcal{CO}_{\delta}(\mathbf{P})$. In much the same way, as in the proof of Theorem \ref{s6:t1}, one can show that if $A\in\mathcal{CO}(\mathbf{P})$ and $A\notin\mathcal{U}^{\ast}$, then there exists $U\in\mathcal{U}$ such that $A\cap e_J[U]=\emptyset$. Mimicking the proof to Theorem \ref{s6:t1}, we can deduce that $\bigcap\mathcal{U}^{\ast}\neq\emptyset$ and, in consequence, there exist $p\in\bigcap\mathcal{U}^{\ast}$ and $f\in C(\mathbf{P}, \mathbb{N}(\infty))$ such that $f(p)=\infty$. In much the same way, as in the proof of Theorem \ref{s6:t1}, we deduce that $\mathcal{U}$ does not have the countable intersection property.\medskip
	
	(ii) Now, suppose that $\mathbf{X}$ is not $\mathbb{N}$-compact. By Theorem \ref{s6:t1}, there exists a free clopen ultrafilter $\mathcal{U}$ in $\mathbf{X}$ such that $\mathcal{U}$ has the countable intersection property. We define
	$$\mathcal{V}=\{ V\in\mathcal{CO}_{\delta}(\mathbf{X}): (\exists U\in \mathcal{U}_{\delta}) U\subseteq V\}.$$
	It is easily seen that $\mathcal{V}$ is a filter in $\mathcal{CO}_{\delta}(\mathbf{X})$ such that $\mathcal{U}\subseteq \mathcal{V}$. To check that $\mathcal{V}$ is an ultrafilter in $\mathcal{CO}_{\delta}(\mathbf{X})$, we consider any $A\in\mathcal{CO}_{\delta}(\mathbf{X})$ such that $A\notin\mathcal{V}$. We choose $\mathcal{A}\in[\mathcal{CO}(\mathbf{X})]^{\leq\omega}\setminus\{\emptyset\}$ such that $A=\bigcap\mathcal{A}$. Since $A\notin\mathcal{V}$, there exists $A_0\in\mathcal{A}$ such that $A_0\notin\mathcal{U}$. Since $\mathcal{U}$ is a clopen ultrafilter, there exists $U_0\in\mathcal{U}$ such that $A_0\cap U_0=\emptyset$. Then $A\cap U_0=\emptyset$. This shows that $\mathcal{V}$ is an ultrafilter in $\mathcal{CO}_{\delta}(\mathbf{X})$. Since $\mathcal{U}$ is free, so is $\mathcal{V}$.
	
	Assuming $\mathbf{CMC}$, let us show that $\mathcal{V}$ has the countable intersection property. To this aim, assume that, for every $n\in\omega$, $V_n\in\mathcal{V}$ and $\mathcal{A}_n=\{U\in\mathcal{U}_{\delta}: U\subseteq V_n\}$. Since $\mathcal{U}$ is an ultrafilter in $\mathcal{CO}(\mathbf{X})$, $\mathcal{U}$ is closed under finite unions and finite intersections. Therefore, in view of Lemma \ref{s6:l4}, it follows from $\mathbf{CMC}$ that $\mathcal{U}_{\delta\delta}=\mathcal{U}_{\delta}$.  By $\mathbf{CMC}$, there exists a sequence $(\mathcal{F}_n)_{n\in\omega}$ such that, for every $n\in\omega$, $\mathcal{F}_n$ is a non-empty finite subset of $\mathcal{A}_n$. For every $n\in\omega$, let $U_n=\bigcap \mathcal{F}_n$. Then, for every $n\in\omega$, $U_n\in\mathcal{U}_{\delta}$, so $W=\bigcap\limits_{n\in\omega}U_n\in\mathcal{U}_{\delta\delta}$. This, together with $\mathbf{CMC}$, implies that $W\in\mathcal{U}_{\delta}$. Hence $W\neq\emptyset$ because $\mathcal{U}$ has the countable intersection property.  Since $W\subseteq\bigcap\limits_{n\in\omega}V_n$, we have $\bigcap\limits_{n\in\omega}V_n\neq\emptyset$ . Hence, $\mathbf{CMC}$ implies that $\mathcal{V}$ is a free ultrafilter in $\mathcal{CO}_{\delta}(\mathbf{X})$ with the countable intersection property. 
\end{proof}

\begin{definition} 
	\label{s6:d6}
	$\mathbf{CHR}$ is the statement: For every Tychonoff space $\mathbf{X}$, it holds that $\mathbf{X}$ is realcompact if and only if every $z$-ultrafilter in $\mathbf{X}$ having the countable intersection property is fixed.
\end{definition}

It is well known that $\mathbf{CHR}$ is true in $\mathbf{ZFC}$ (cf., e.g.,  \cite[Theorem 3.11.11]{en}); however, the following problem is unsolved:

\begin{problem}
	\label{s6:q7}
	Is $\mathbf{CHR}$ provable in $\mathbf{ZF}$?
\end{problem}

To learn more about $\mathbf{CHR}$ in $\mathbf{ZF}$, we need the following definition:

\begin{definition} 
	\label{s6:d8}
	Let $\mathbf{X}$ be a topological space and let $\mathcal{F}\subseteq\mathcal{Z}(\mathbf{X})$. We say that:
	\begin{enumerate}
		\item[(i)] $\mathcal{F}$ is \emph{functionally accessible} if there exists a family $\{f_Z: Z\in\mathcal{F}\}$ of functions from $C^{\ast}(\mathbf{X})$ such that, for every $Z\in\mathcal{F}$, $f_Z^{-1}(\{0\})=Z$;
		\item[(ii)]$\mathcal{F}$ has the weak countable intersection property if, for every non-empty functionally accessible countable subfamily $\mathcal{F}^{\ast}$ of $\mathcal{F}$, $\bigcap\mathcal{F}^{\ast}\neq\emptyset$. 
	\end{enumerate}
\end{definition}

\begin{proposition}
	\label{s6:p9}
	$[\mathbf{ZF}]$ $\mathbf{CMC}$ implies that, for every topological space $\mathbf{X}$, it holds that every $z$-ultrafilter in $\mathbf{X}$ with the weak countable intersection property has the countable intersection property.
\end{proposition}

\begin{proof} Assuming $\mathbf{CMC}$, we fix a topological space $\mathbf{X}$ and a $z$-ultrafilter $\mathcal{U}$ in $\mathbf{X}$ such that $\mathcal{U}$ has the weak countable intersection property.  Let $\{Z_n: n\in\omega\}$ be a subfamily of $\mathcal{U}$. For every $n\in\omega$, let $\mathcal{F}_n=\{f\in C(X): Z_n= Z(f)\}$. By $\mathbf{CMC}$, there exists a family $\{\mathcal{H}_n: n\in\omega\}$ such that, for every $n\in\omega$, $\mathcal{H}_n$ is a non-empty finite subset of $\mathcal{F}_n$. For every $n\in\omega$, we define $f_n=\sum\limits_{f\in\mathcal{H}_n}|f|$. Then, for every $n\in\omega$,  $f_n\in C(\mathbf{X})$ and $Z(f_n)=Z_n$. This implies that the family $\{Z_n: n\in\omega\}$ is functionally accessible, so $\bigcap\limits_{n\in\omega}Z_n\neq\emptyset$. 
\end{proof}

The following theorem shows that a modification of $\mathbf{CHR}$ holds in $\mathbf{ZF}$.

\begin{theorem}
	\label{s6:t10}
	$[\mathbf{ZF}]$ A Tychonoff space $\mathbf{X}$ is realcompact if and only if every $z$-ultrafilter in $\mathbf{X}$ with the weak countable intersection property is fixed. 
\end{theorem}

\begin{proof}
	For a Tychonoff space $\mathbf{X}$, let $J=C^{\ast}(\mathbf{X})$, $\mathbf{Y}=\mathbb{R}^J$ and let $P$ be the closure in $\mathbf{Y}$ of $e_J[X]$. It is obvious that $e_J=e_{\beta^f}$, and, for the subspace $\mathbf{P}$ of $\mathbf{Y}$, $\mathbf{P}=\beta^{f}\mathbf{X}$.
	
	\textit{Sufficiency.} We consider any point $p\in P\setminus e_J[X]$ and define $\mathcal{F}=\{Z\in\mathcal{Z}(\mathbf{X}): p\in\cl_{\mathbf{Y}}e_J[Z]\}$. One can check that $\mathcal{F}$ is a free $z$-ultrafilter in $\mathbf{X}$. Assuming that every $z$-ultrafilter in $\mathbf{X}$ with the weak countable intersection property is fixed, we obtain that $\mathcal{F}$ does not have the weak countable intersection property. Hence, there exists a family $\{f_n: n\in\mathbb{N}\}$ of functions from $C^{\ast}(\mathbf{X})$ such that, for every $n\in\mathbb{N}$, the set $Z_n=f_n^{-1}[\{0\}]$ belongs to $\mathcal{F}$ but $\bigcap\limits_{n\in\mathbb{N}}Z_n=\emptyset$. We define $f\in C^{\ast}(\mathbf{X})$ by: $f(x)=\sum\limits_{n=1}^{+\infty}\frac{\min\{|f_n(x)|, 1\}}{2^n}$ for every $x\in X$. Then, for every $x\in X$, $f(x)>0$. Let $\tilde{f}\in C^{\ast}(\mathbf{P})$ be such that $\tilde{f}\circ e_J=f$. Then $\tilde{f}(p)=0$. Hence $\mathbf{X}$ is realcompact by Theorem \ref{s4:t16}(i).
	
	\textit{Necessity.} Now, we assume that $\mathbf{X}$ is realcompact and $\mathcal{U}$ is a free $z$-ultrafilter in $\mathbf{X}$. For every $j\in J$, let $\mathcal{U}_j=\{ \cl_{\mathbb{R}}\pi_j[e_J[U]]: U\in\mathcal{U}\}$. Let $j\in J$.  It is easily seen that $\mathcal{U}_j$ has the finite intersection property and all members of $\mathcal{U}_j$ are compact subsets of $\mathbb{R}$. Therefore, $\bigcap\mathcal{U}_j\neq\emptyset$. We can define $p_j=\max(\bigcap\mathcal{U}_j)$. Let $p\in Y$ be such that, for every $j\in J$, $p(j)=p_j$. Suppose that there exists $U_0\in\mathcal{U}$ such that $p\notin\cl_{\mathbf{Y}}(e_J[U_0])$. There exists a collection $\{V_j: j\in J\}$ of open subsets of $\mathbb{R}$ such that $p\in V=\prod\limits_{j\in J}V_j$, the set $K=\{j\in J: V_j\neq\mathbb{R}\}$ is finite, and $V\cap e_J[U_0]=\emptyset$. Then $K\neq\emptyset$. Let us consider an arbitrary $j\in K$. We can find real numbers $a_j, b_j$ such that $[a_j, b_j]\subseteq V_j$ and $a_j<p_j<b_j$. Let $W_j=\pi_j^{-1}[[a_j, b_j]]$ and let $U\in\mathcal{U}$. Since $(a_j, b_j)\cap \pi_j[e_J[U]]\neq\emptyset$, we have $W_j\cap e_J[U]\neq\emptyset$. Therefore, $e_J^{-1}[W_j]\cap U\neq\emptyset$. Since $e_J^{-1}[W_j]\in\mathcal{Z}(\mathbf{X})$ and $\mathcal{U}$ is a $z$-ultrafilter in $\mathbf{X}$, we have $e_J^{-1}[W_j]\in\mathcal{U}$. Since $\mathcal{U}$ is a $z$-ultrafilter, we obtain that $\bigcap\limits_{j\in K}e_J^{-1}[W_j]\in\mathcal{U}$. For every $j\in J\setminus K$, let $W_j=\mathbb{R}$. Let $W=\prod\limits_{j\in J}W_j$. We have shown that $e_J^{-1}[W]\in\mathcal{U}$, so $e_J^{-1}[W]\cap U_0\neq\emptyset$. Thus, $W\cap e_J[U_0]\neq\emptyset$ but this is impossible. The contradiction obtained shows that $p\in\bigcap\{\cl_{\mathbf{Y}}(e_j[U]): U\in\mathcal{U}\}$. Since $\mathcal{U}$ is free, we deduce that $p\in e_J X\setminus e_J[X]$. Since $\mathbf{X}$ is realcompact, it follows from Theorem \ref{s4:t16}(i) that there exists $f\in C^{\ast}(\mathbf{P})$ such that $f(p)=0$ and, for every $x\in X$, $f(e_J(x))>0$. Without loss of generality, we may assume that $f[P]\subseteq[0, 1]$. For every $n\in\mathbb{N}$, we define $Z_n=e_J^{-1}[f^{-1}[[0, \frac{1}{n}]]]$. In much the same way, as above, one can check that, for every $n\in\mathbb{N}$, $Z_n\in\mathcal{U}$. However, $\bigcap\limits_{n\in\mathbb{N}}Z_n=\emptyset$ and the family $\{Z_n: n\in\mathbb{N}\}$ is functionally accessible. Hence $\mathcal{U}$ does not have the weak countable intersection property. 
\end{proof}

The following corollary gives a partial solution to Problem \ref{s6:q7}.
\begin{corollary}
	\label{s6:c11}
	$[\mathbf{ZF}]$
	$\mathbf{CMC}$ implies $\mathbf{CHR}$.
\end{corollary}

We recall that a subspace $\mathbf{Y}$ of a topological space $\mathbf{X}$ is $z$-\emph{embedded} in $\mathbf{X}$ if, for every $Z\in\mathcal{Z}(\mathbf{Y})$, there exists $\tilde{Z}\in\mathcal{Z}(\mathbf{X})$ such that $Z=\tilde{Z}\cap Y$.  In \cite[Lemma 3.9]{Bl}, it is stated in $\mathbf{ZFC}$ that if a Tychonoff space $\mathbf{X}$ is a countable union of $z$-embedded realcompact subspaces, then $\mathbf{X}$ is realcompact. Let us show in the following theorem that, by applying our Theorem \ref{s6:t10} and, to a great extent, arguing similarly to \cite[the proof of Lemma 3.9]{Bl}, Lemma 3.9 of \cite{Bl} can be proved in $\mathbf{ZF+CMC}$: 

\begin{theorem}
	\label{s6:t12}
	$[\mathbf{ZF}]$ $\mathbf{CMC}$ implies that if a Tychonoff space $\mathbf{X}$ is a countable union of $z$-embedded realcompact subspaces, then $\mathbf{X}$ is realcompact.
\end{theorem}

\begin{proof}
	Let us assume $\mathbf{CMC}$. Let $\mathbf{X}$ be a Tychonoff space such that $X=\bigcup\limits_{n\in\mathbb{N}}X_n$ where, for every $n\in\mathbb{N}$, the subspace $\mathbf{X}_n$ of $\mathbf{X}$ is realcompact and $z$-embedded in $\mathbf{X}$. Let $\mathcal{U}$ be a $z$-ultrafilter in $\mathbf{X}$ such that $\mathcal{U}$ has the countable intersection property. Suppose that, for every $n\in\mathbb{N}$, the set $\mathcal{U}_n=\{U\in\mathcal{U}: U\cap X_n=\emptyset\}$ is non-empty. By $\mathbf{CMC}$, there exists $\psi\in\prod\limits_{n\in\mathbb{N}}([\mathcal{U}_n]^{<\omega}\setminus\{\emptyset\})$. Then, for every $n\in\mathbb{N}$, the set $U_n=\bigcap\{U: U\in\psi(n)\}$ is a member of $\mathcal{U}$ such that $U_n\cap X_n=\emptyset$. Since $\bigcap\limits_{n\in\mathbb{N}}U_n=\emptyset$, we obtain a contradiction with the countable intersection property of $\mathcal{U}$. Therefore, there exists $n_0\in\mathbb{N}$ such that $\mathcal{U}_{n_0}=\emptyset$. Let $\mathcal{F}=\{U\cap X_{n_0}: U\in\mathcal{U}\}$. Then $\mathcal{F}$ is a $z$-filter on $\mathbf{X}_{n_0}$. Since $\mathbf{X}_{n_0}$ is $z$-embedded in $\mathbf{X}$, $\mathcal{F}$ is a $z$-ultrafilter on $\mathbf{X}_{n_0}$. For every $n\in\mathbb{N}$, let $F_n\in\mathcal{F}$ and $\mathcal{W}_n=\{U\in\mathcal{U}: F_n=U\cap X_{n_0}\}$. Since $\mathcal{U}$ is closed under finite intersections, in much the same way,  as in the proof that, for some $n\in\mathbb{N}$,  $\mathcal{U}_n\neq\emptyset$, we can deduce from $\mathbf{CMC}$ that there exists a family $\{A_n: n\in\mathbb{N}\}$ of members of $\mathcal{U}$ such that, for every $n\in\mathbb{N}$, $F_n=A_n\cap X_{n_0}$. Let $A=\bigcap\limits_{n\in\mathbb{N}}A_n$.  It follows from $\mathbf{CMC}$ and the proof of Proposition \ref{s6:p9} that the family $\{A_n: n\in\mathbb{N}\}$ is functionally accessible. This implies that $A\in\mathcal{Z}(\mathbf{X})$.  Since the $z$-ultrafilter $\mathcal{U}$ has the countable intersection property, we infer that $A\in\mathcal{U}$. Hence $A\cap X_{n_0}\neq\emptyset$, so $\bigcap\limits_{n\in\mathbb{N}}F_n\neq\emptyset$. Therefore, $\mathcal{F}$ has the countable intersection property. Since $\mathbf{X}_{n_0}$ is realcompact, the filter $\mathcal{F}$ is fixed. Hence $\mathcal{U}$ is also fixed. In the light of Theorem \ref{s6:t10}, $\mathbf{X}$ is realcompact.
\end{proof}

To establish an accurate modification of Theorem  \ref{s6:t12} for $\mathbb{N}$-compactness, we need the following definition:

\begin{definition}
	\label{s6:d13}
	A subspace $\mathbf{S}$ of a topological space $\mathbf{X}$ will be called:
	\begin{enumerate}
		\item[(i)] $c$-\emph{embedded} in $\mathbf{X}$ if, for every $A\in\mathcal{CO}(\mathbf{S})$, there exists $B\in\mathcal{CO}(\mathbf{X})$ such that $A=B\cap S$;
		\item[(ii)] $c_{\delta}$-\emph{embedded} in $\mathbf{X}$ if, for every $A\in\mathcal{CO}_{\delta}(\mathbf{S})$, there exists $B\in\mathcal{CO}_{\delta}(\mathbf{X})$ such that $A=B\cap S$.
	\end{enumerate}
\end{definition}

\begin{proposition}
	\label{s6:p14}
	$[\mathbf{ZF}]$  $\mathbf{CMC}$ implies that, for every topological space $\mathbf{X}$, the following conditions are satisfied:
	\begin{enumerate}
		\item[(i)] for every set $A\subseteq X$, it holds that $A\in\mathcal{CO}_{\delta}(\mathbf{X})$ if and only if there exists $f\in U_{\aleph_0}(\mathbf{X})$ such that $A=Z(f)$; furthermore, we may demand that $0\leq f\leq 1$;  
		\item[(ii)]  for every ultrafilter $\mathcal{U}$ in $\mathcal{CO}_{\delta}(\mathbf{X})$, it holds that $\mathcal{U}$ has the countable intersection property if and only if $\mathcal{U}$ is closed under countable intersections; 
		\item[(iii)]   for every subspace $\mathbf{S}$ of $\mathbf{X}$, it holds that if $\mathbf{S}$ is $c$-embedded in $\mathbf{X}$, then it is also $c_{\delta}$-embedded in $\mathbf{X}$.
	\end{enumerate}
\end{proposition} 

\begin{proof}
	Let us assume $\mathbf{CMC}$ and fix a topological space $\mathbf{X}$. We recall that, by Theorem \ref{s1:t18}(b), $\mathbf{CMC}$ implies that $\mathcal{CO}_{\delta}(\mathbf{X})$ is closed under countable intersections.

	(i) Let $A\in\mathcal{CO}_{\delta}(\mathbf{X})$. Choose a family $\{A_n: n\in\mathbb{N}\}$ of clopen sets of $\mathbf{X}$ such that $A=\bigcap\limits_{n\in\mathbb{N}}A_n$. For every $n\in\mathbb{N}$, let $e_n: X\to\{0, 1\}$ be the characteristic function of $X\setminus A_n$, that is, for every $x\in X$,  $e_n(x)=0$ if and only if $x\in A_n$. For every $n\in\mathbb{N}$, let $f_n=\sum\limits_{i=1}^{n}\frac{e_i}{2^i}$. Then $f_n\in C(\mathbf{X}, \mathbb{R}_{disc})$. Let $f=\lim\limits_{n\to+\infty}f_n$. It follows from Theorem \ref{s1:t21}(c) that $f\in U_{\aleph_0}(\mathbf{X})$. Of course, $A=Z(f)$. On the other hand, if $g\in U_{\aleph_0}(\mathbf{X})$, then $Z(g)\in\mathcal{CO}_{\delta}(\mathbf{X})$ by Theorem \ref{s1:t21}(c). Hence (i) holds.
	
	(ii) Suppose that $\mathcal{U}$ is an ultrafilter in $\mathcal{CO}_{\delta}(\mathbf{X})$. Let $\{F_n: n\in\omega\}$ be a family of members of $\mathcal{U}$ and  let $F=\bigcap\limits_{n\in\omega}F_n$. Since $\mathcal{CO}_{\delta}(\mathbf{X})$ is closed under countable intersections,  $F\in\mathcal{CO}_{\delta}(\mathbf{X})$. Thus, assuming that  $F\notin\mathcal{U}$, we can fix $U_F\in\mathcal{U}$ such that $F\cap U_F=\emptyset$. Then $\{ U_F\}\cap\{F_n: n\in\omega\}$ witnesses that $\mathcal{U}$ does not have the countable intersection property. On the other hand, if an ultrafilter in $\mathcal{CO}_{\delta}(\mathbf{X})$ is closed under countable intersections, then it has the countable intersection property. 
	
	(iii) Suppose that $\mathbf{S}$ is a $c$-embedded subspace of $\mathbf{X}$. Let $C\in\mathcal{CO}_{\delta}(\mathbf{S})$. Let $\{C_n: n\in\omega\}$ be a family of members of $\mathcal{CO}(\mathbf{S})$ such that $C=\bigcap\limits_{n\in\omega}C_n$. Since $\mathbf{S}$ is $c$-embedded in $\mathbf{X}$, for every $n\in\omega$, the family $\mathcal{D}_n=\{D\in\mathcal{CO}(\mathbf{X}): C_n=D\cap S\}$ is non-empty. By $\mathbf{CMC}$, there exists $\psi\in\prod\limits_{n\in\omega}([\mathcal{D}_n]^{<\omega}\setminus\{\emptyset\})$. For every $n\in\omega$, let $D_n=\bigcap\psi(n)$. Let $D=\bigcap\limits_{n\in\omega}D_n$. Since, for every $n\in\omega$, $D_n\in\mathcal{CO}(\mathbf{X})$, we have $D\in\mathcal{CO}_{\delta}(\mathbf{X})$. It is obvious that $D\cap S=C$, so (iii) holds.
\end{proof}

\begin{remark}
	\label{s5:r015}
	Let us notice that, in proof of Proposition \ref{s6:p14}(i), to show that $f\in U_{\aleph_0}(\mathbf{X})$ one need not use $\mathbf{CMC}$. Indeed, for a topological space $\mathbf{X}$, given a sequence $(f_n)_{n\in\mathbb{N}}$ of functions from $C(\mathbf{X}, \mathbb{R}_{disc})$  such that $(f_n)_{n\in\mathbb{N}}$ is uniformly convergent to a function $f:\mathbf{X}\to\mathbb{R}$, one can check that it holds in $\mathbf{ZF}$ that $f\in U_{\aleph_0}(\mathbf{X})$ by using the following simple arguments instead of Theorem \ref{s1:t21}(c). For a fixed positive real number $\varepsilon$, there exists $n_0\in\mathbb{N}$ such that, for every $x\in X$, $|f(x)-f_{n_0}(x)|\leq\frac{\varepsilon}{4}$. For $q\in\mathbb{Q}$, we put $V_q=f_{n_0}^{-1}[[q-\frac{\varepsilon}{4}, q+\frac{\varepsilon}{4}]]$. Then, for every $q\in\mathbb{Q}$, $V_q\in\mathcal{CO}(\mathbf{X})$ and $\osc_{V_q}(f)\leq\varepsilon$; furthermore, $\bigcup\{ V_q: q\in\mathbb{Q}\}=X$. This shows that $f\in U_{\aleph_0}(\mathbf{X})$ in $\mathbf{ZF}$.
\end{remark}

We shall omit a simple proof of the following proposition:

\begin{proposition}
	\label{s6:p15}
	$[\mathbf{ZF}]$
	Let $\mathbf{S}$ be a $c_{\delta}$-embedded (respectively, $c$-embedded) subspace of a topological space $\mathbf{X}$. Let $\mathcal{F}$ be a a filter in $\mathcal{CO}_{\delta}(\mathbf{S})$ (respectively, in $\mathcal{CO}(\mathbf{S})$) and let $\mathcal{U}$ be an ultrafilter in $\mathcal{CO}_{\delta}(\mathbf{X})$ (respectively, in $\mathcal{CO}(\mathbf{X})$). Then:
	\begin{enumerate} 
		\item[(i)] the family 
		$$\mathcal{F}_X=\{V\in\mathcal{CO}_{\delta}(\mathbf{X}) \text{(respectively,} V\in\mathcal{CO}(\mathbf{X})\text{)}: V\cap S\in\mathcal{F}\}$$ is a filter in $\mathcal{CO}_{\delta}(\mathbf{X})$ (respectively, in $\mathcal{CO}(\mathbf{X})$);
		\item[(ii)] if, for every $U\in\mathcal{U}$, $U\cap S\neq\emptyset$, then the family 
		$$\mathcal{U}_S=\{U\cap S: U\in\mathcal{U}\}$$
		is an ultrafilter in $\mathcal{CO}_{\delta}(\mathbf{S})$ (respectively, in $\mathcal{CO}(\mathbf{S})$).  
	\end{enumerate} 
\end{proposition}

\begin{theorem}
	\label{s6:t16}
	$[\mathbf{ZF}]$ Let $\mathbf{X}$ be a zero-dimensional  $T_1$-space. Let $\{\mathbf{S}_n: n\in\omega\}$ be a family of $c_{\delta}$-embedded $\mathbb{N}$-compact subspaces of $\mathbf{X}$ such that $X=\bigcup\limits_{n\in\omega}S_n$. Then $\mathbf{CMC}$ implies that $\mathbf{X}$ is $\mathbb{N}$-compact. 
\end{theorem}
\begin{proof}
	Assuming $\mathbf{CMC}$, we consider any ultrafilter $\mathcal{U}$ in $\mathcal{CO}_{\delta}(\mathbf{X})$ such that $\mathcal{U}$ has the countable intersection property. It follows from Proposition \ref{s6:p14}(ii) that $\mathcal{U}$ is closed under countable intersections. In much the same way, as in the proof of Theorem \ref{s6:t12}, we can show that there exists $n_0\in\omega$ such that, for every $U\in\mathcal{U}$, $U\cap S_{n_0}\neq\emptyset$. By Proposition \ref{s6:p15}(ii), the family $\mathcal{U}_{n_0}=\{U\cap S_{n_0}: U\in\mathcal{U}\}$ is an ultrafilter in $\mathcal{CO}_{\delta}(\mathbf{S}_{n_0})$. Since $\mathcal{U}$ is closed under countable intersections, $\mathbf{CMC}$ implies that $\mathcal{U}_{n_0}$ is closed under countable intersections. Since $\mathbf{S}_{n_0}$ is $\mathbb{N}$-compact, it follows from Theorem \ref{s6:t5} that $\mathcal{U}_{n_0}$ is fixed. This implies that $\mathcal{U}$ is fixed, so $\mathbf{X}$ is $\mathbb{N}$-compact by Theorem \ref{s6:t5}.
\end{proof}

\begin{corollary}
	\label{s6:c17}
	$[\mathbf{ZF}]$ Let $\mathbf{X}$ be a zero-dimensional $T_1$-space. Let $\{\mathbf{S}_n: n\in\omega\}$ be a family of $c$-embedded $\mathbb{N}$-compact subspaces of $\mathbf{X}$ such that $X=\bigcup\limits_{n\in\omega}S_n$. Then $\mathbf{CMC}$ implies that $\mathbf{X}$ is $\mathbb{N}$-compact. 
\end{corollary}
\begin{proof}
	It suffices to apply Theorem \ref{s6:t16} and Proposition \ref{s6:p14}(iii). 
\end{proof}

In \cite[Fact 21]{nik} (see also \cite[Theorem 8.16]{gj} and \cite[Exercise 3.11 A]{en}), it was proved that it holds in $\mathbf{ZFC}$ that every Tychonoff space expressible as the union of a realcompact subspace and a compact subspace is realcompact. We can establish the following strengthening of \cite[Fact 21]{nik}.

\begin{theorem}
	\label{s6:t18}
	$[\mathbf{ZF}]$ Let $\mathbf{X}=\langle X, \tau\rangle$ be a $T_1$-space. Suppose that $X=S\cup K$ where the set $K$ is compact in $\mathbf{X}$. Then the following conditions are satisfied:
	\begin{enumerate}
		\item[(i)] if $\mathbf{X}$ is completely regular and the subspace $\mathbf{S}$ of $\mathbf{X}$ is realcompact, then $\mathbf{X}$ is realcompact;
		\item[(ii)] if $\mathbf{X}$ is zero-dimensional and the subspace $\mathbf{S}$ of $\mathbf{X}$ is $\mathbb{N}$-compact, then $\mathbf{X}$ is $\mathbb{N}$-compact. 
	\end{enumerate}
\end{theorem}

\begin{proof}
	
	Let $\mathbf{X}$ be a Tychonoff space (respectively, a zero-dimensional space). Suppose that $\mathcal{F}$ is a free ultrafilter in $\mathcal{Z}(\mathbf{X})$ (respectively, in $\mathcal{CO}(\mathbf{X})$). Put
	$$\mathcal{F}_S= \{S\cap F: F\in\mathcal{F}\}.$$
	\noindent Since $\bigcap\mathcal{F}=\emptyset$ and $K$ is compact in $\mathbf{X}$, we can fix $Z_0\in\mathcal{F}$ such that $K\cap Z_0=\emptyset$. Now, one can easily check that $\mathcal{F}_S$ is a filter in $\mathcal{Z}(\mathbf{S})$ (respectively, in $\mathcal{CO}(\mathbf{S})$). 
	Let us assume that $W\in\mathcal{Z}(\mathbf{S})$.
	
	(i) For the proof of (i), assuming that $\mathbf{X}$ is a Tychonoff space, by the compactness of $K$ in $\mathbf{X}$, we can choose $Z_1\in\mathcal{Z}(\mathbf{X})$ such that $K\subseteq \inter_{\mathbf{X}}(Z_1)$ and $Z_0\cap Z_1=\emptyset$. In much the same way, as in \cite[the proof of Fact 21]{nik}, one can show that $W\cap Z_0\in\mathcal{Z}(\mathbf{X})$. Suppose that, for every $T\in\mathcal{F}_S$, $W\cap T\neq\emptyset$. Then, for every $F\in\mathcal{F}$, $W\cap Z_0\cap F\neq\emptyset$, so $W\cap Z_0\in\mathcal{F}$. Hence $W\in\mathcal{F}_S$. This proves that $\mathcal{F}_S$ is an ultrafilter in $\mathcal{Z}(\mathbf{S})$. 
	
	Now, assume that $\mathcal{F}$ has the weak countable intersection property. Let us show that $\mathcal{F}_S$ has the weak countable intersection property. To this aim, we assume that  $\{f_n: n\in\omega\}$ is a family of functions from $C(\mathbf{S}, [0, 1])$ such that, for every $n\in\omega$, $F_n=\{s\in S: f_n(s)=0\}\in\mathcal{F}_S$. We can fix $g\in C(\mathbf{S}, [0, 1])$ such that $S\cap Z_1=\{ s\in S: g(s)=0\}$. Using the idea from \cite[the proof of Fact 21]{nik}, for every $n\in\omega$, we define a function $h_n\in C(\mathbf{X}, [0, 1])$ as follows:
	$$
	h_n(x)=\begin{cases} \frac{f_n(x)}{f_n(x)+g(x)} &\text{ if } x\in X\setminus \inter_{\mathbf{X}}Z_1;\\
		1 &\text{ if } x\in Z_1.\end{cases}
	$$
	We notice that, for every $n\in\omega$, $F_n\cap Z_0\in\mathcal{F}$ and $F_n\cap Z_0=\mathcal{Z}(h_n)$. Hence,  the family $\{F_n\cap Z_0: n\in\omega\}$ is a countable  admissible subfamily of $\mathcal{F}$. Since $\mathcal{F}$ has the weak countable intersection property, $\bigcap\limits_{n\in\omega}(F_n\cap Z_0)\neq\emptyset$. This proves that $\mathcal{F}_S$ has the weak countable intersection property. Since $\mathcal{F}_S$ is free, $\mathbf{S}$ is not realcompact by Theorem \ref{s6:t10}. This completes the proof of (i).
	
	(ii) For the proof of (ii), we assume that $\mathbf{X}$ is zero-dimensional, $\mathcal{F}$ is a free ultrafilter in $\mathcal{CO}(\mathbf{X})$, and $Z_0\in\mathcal{F}$ is such that $Z_0\cap K=\emptyset$.  We notice that if $W\in\mathcal{CO}(\mathbf{S})$, then,  $W\cap Z_0\in \mathcal{CO}(\mathbf{X})$, so, arguing in much the same way, as in the proof of (i), we can show that $\mathcal{F}_S$ is an ultrafilter in $\mathcal{CO}(\mathbf{S})$. Assuming that $\mathcal{F}$ has the countable intersection property, to check that $\mathcal{F}_S$ has the countable intersection property, it suffices to notice that if $\{E_n: n\in\omega\}$ is a collection of members of $\mathcal{F}_S$, then $\{E_n\cap Z_0: n\in\omega\}$ is a collection of members of $\mathcal{F}$. Hence, by Theorem \ref{s6:t1}, if $\mathbf{S}$ is $\mathbb{N}$-compact, so is $\mathbf{X}$.  
\end{proof}

\begin{definition}
	\label{s6:d19}
	A topological space $\mathbf{X}$ is called \emph{hereditarily realcompact} (respectively, \emph{hereditarily} $\mathbb{N}$-\emph{compact}) if every subspace of $\mathbf{X}$ is realcompact (respectively, $\mathbb{N}$-compact).
\end{definition}

Let us write down the following corollary to Theorem \ref{s6:t18}:

\begin{corollary}
	\label{s6:c20}
	$[\mathbf{ZF}]$ 
	Let $f: \mathbf{X}\to\mathbf{Y}$ be a continuous bijection of a Tychonoff space $\mathbf{X}$ onto a Tychonoff space $\mathbf{Y}$. Then the following conditions are satisfied:
	\begin{enumerate}
		\item[(i)] if $\mathbf{Y}$ is hereditarily realcompact, so is $\mathbf{X}$;
		\item[(ii)] if $\mathbf{X}$ is zero-dimensional and $\mathbf{Y}$ is hereditarily $\mathbb{N}$-compact, so is $\mathbf{X}$.
	\end{enumerate}
\end{corollary}
\begin{proof}
	We shall prove (ii). The proof of (i) is similar (see \cite[the proof of Fact 24]{nik}).
	
	Let us assume that $\mathbf{X}$ is zero-dimensional, and $\mathbf{Y}$ is hereditarily $\mathbb{N}$-compact.  Let $\tilde{f}: v_{\mathbb{N}}\mathbf{X}\to\mathbf{Y}$ be the continuous extension of $f$. Suppose that $p\in v_{\mathbb{N}}X\setminus X$ and $\tilde{f}(p)=y$. The subspace $Y\setminus\{y\}$ of $\mathbf{Y}$  is $\mathbb{N}$-compact. Let $S=v_{\mathbb{N}}X\setminus \tilde{f}^{-1}[\{y\}]$. By Proposition \ref{s4:p7}, the subspace $\mathbf{S}$ of $v_{\mathbb{N}}\mathbf{X}$ is $\mathbb{N}$-compact. The set $K=\{x\in X: f(x)=p\}$ is a singleton, so it is compact. Hence, the subspace $S\cup K$ of $v_{\mathbb{N}}\mathbf{X}$ is $\mathbb{N}$-compact. Since $X\subseteq S\cup K\subseteq v_{\mathbb{N}}X$, we infer that $X=v_{\mathbb{N}}X$. Therefore, $\mathbf{X}$ is $\mathbb{N}$-compact. If $\mathbf{Z}$ is a subspace of $\mathbf{X}$ then there is a continuous bijection of $\mathbf{Z}$ onto a subspace of $\mathbf{Y}$, so $\mathbf{Z}$ is also $\mathbb{N}$-compact.
\end{proof}

\begin{corollary}
	\label{s6:c21}
	$[\mathbf{ZF}]$ Let $\mathbf{X}=\langle X, \tau_X\rangle$ be a $T_1$-space. Then the following hold:
	
	\begin{enumerate}
		\item[(i)] if $\mathbf{X}$ is completely regular, then $\mathbf{X}$ is hereditarily realcompact if and only if, for every completely regular topology $\tau$ on $X$ such that $\tau_X\subseteq \tau$, the space $\langle X, \tau\rangle$  is realcompact;
		\item[(ii)] if $\mathbf{X}$ is zero-dimensional, then  $\mathbf{X}$ is hereditarily $\mathbb{N}$-compact if and only if, for zero-dimensional topology $\tau$ on $X$ such that $\tau_X\subseteq \tau$, the space $\langle X, \tau\rangle$  is $\mathbb{N}$-compact.
	\end{enumerate}
\end{corollary}

\begin{proof}
	Let us prove (ii) and omit a similar proof of (i).
	
	Suppose that, for every zero-dimensional topology $\tau$ on $X$ such that $\tau_X\subseteq \tau$, the space $\langle X, \tau\rangle$ is $\mathbb{N}$-compact. Fix a point $x\in X$ and let $\tau(x)=\tau\cup\{V\cup\{x_0\}: V\in\tau_X\}$. Then $\tau(x)$ is a zero-dimensional topology on $X$, and $x$ is an isolated point of $\langle X, \tau(x)\rangle$. For $S(x)= X\setminus \{x\}$, we have $\{V\cap S(x): V\in\tau_X\}=\{V\cap S: V\in \tau(x)\}$. By our hypothesis and by Proposition \ref{s4:p4}, the subspace $S(x)$ of $\langle X, \tau(x)\rangle$ is $\mathbb{N}$-compact. Hence, the subspace $\mathbf{S}(x)$ of $\mathbf{X}$ is $\mathbb{N}$-compact. If $S\subseteq X$, then $S=\bigcap\limits_{x\in X\setminus S}(X\setminus \{x\})$, so $\mathbf{X}$ is hereditarily $\mathbb{N}$-compact by Proposition \ref{s4:p6}. To complete the proof of (ii), it suffices to apply Corollary \ref{s6:c20}.
\end{proof}

\begin{corollary}
	\label{s6:c22}
	$[\mathbf{ZF}]$ If $\mathbf{X}$ is hereditarily realcompact, so is $(\mathbf{X})_z$. Similarly, if $\mathbf{X}$ is hereditarily $\mathbb{N}$-compact, so is $(\mathbf{X})_z$.
\end{corollary}

\begin{remark}
	\label{s6:r23} 
	In \cite[Remark (iii), p.  234]{lr}, it is said that every  Tychonoff topology stronger than a realcompact topology is realcompact in $\mathbf{ZFC}$. That the above-mentioned statement of \cite[Remark (iii), p. 234]{lr} is false, follows easily from Corollary \ref{s6:c21}.
\end{remark}

\begin{definition}
	\label{s6:d24}
	Let $\mathbf{X}$ and $\mathbf{Y}$ be topological spaces. A continuous surjection $f: \mathbf{X}\to\mathbf{Y}$ is called $z$-perfect (respectively, $c_{\delta}$-perfect if, for every $p\in Y$, the set $f^{-1}[\{p\}]$ is compact in $\mathbf{X}$  and, for every closed in $\mathbf{X}$ set $C$ and every $y\in Y\setminus f[C]$, there exists $Z\in \mathcal{Z}(\mathbf{Y})$ (respectively, $Z\in\mathcal{C}_{\delta}(\mathbf{Y})$) such that $y\in Z\subseteq Y\setminus f[C]$.
\end{definition}

\begin{lemma}
	\label{s6:l25}
	$[\mathbf{ZF}]$ Let $C$ be a subset of a topological space $\mathbf{X}$. Then $C\in CO_{\delta}(\mathbf{X})$  if and only if there exists $f\in C(\mathbf{X},\mathbb{N}(\infty))$ such that $C=f^{-1}[\{\infty\}]$.
\end{lemma}
\begin{proof}
	Clearly, if $f\in C(\mathbf{X}, \mathbb{N}(\infty))$ and $C=f^{-1}[\{\infty\}]$, then $C$ is a $c_{\delta}$-set of $\mathbf{X}$. On the other hand, if $C$ is clopen in $\mathbf{X}$, we define a function $g\in C(\mathbf{X}, \mathbb{N}(\infty))$ as follows:
	$$
	g(x)=\begin{cases} 1 &\text{ if } x\in X\setminus C;\\
		\infty &\text{ if } x\in C.\end{cases}
	$$
	\noindent Then $C=g^{-1}[\{\infty\}]$. 
	
	Let us assume that $C$ is not clopen in $\mathbf{X}$ but $C\in\mathcal{CO}_{\delta}(\mathbf{X})$. Then we can fix a family $\{C_n: n\in\omega\}$ of clopen sets of $\mathbf{X}$ such that  $C=\bigcap\limits_{n\in\omega}C_n$ where $C_0=X$ and, for every $n\in\omega$, $C_{n+1}\subsetneq C_n$. Let us define a function $h: X\to\mathbb{N}(\infty)$ as follows:
	$$
	h(x)=\begin{cases} n+1 &\text{ if $x\in C_n\setminus C_{n+1}$;}\\
		\infty &\text{ if $x\in C$.}\end{cases}
	$$
	Then $h\in C(\mathbf{X}, \mathbb{N}(\infty))$ and $C=h^{-1}[\{\infty\}]$. 
\end{proof}

\begin{theorem}
	\label{s6:t26}
	$[\mathbf{ZF}]$ Suppose that $\mathbf{X}$ and $\mathbf{Y}$ are Tychonoff (respectively, zero-dimension\-al $T_1$) spaces, and $f:\mathbf{X}\to\mathbf{Y}$ is a $z$-perfect (respectively $c_{\delta}$-perfect) mapping. Let $\tilde{f}: v_{\mathbb{R}}\mathbf{X}\to v_{\mathbb{R}}\mathbf{Y}$ (respectively, $\tilde{f}:v_{\mathbb{N}}\mathbb{X}\to v_{\mathbb{N}}\mathbb{Y}$) be the continuous extension of $f$. Then $\tilde{f}[v_{\mathbb{R}}X\setminus X]\subseteq v_{\mathbb{R}}Y\setminus Y$ (respectively, $\tilde{f}[v_{\mathbb{N}}X\setminus X]\subseteq v_{\mathbb{N}}Y\setminus Y$).
\end{theorem}
\begin{proof}
	
	Suppose that $p\in v_{\mathbb{R}}X\setminus X$ (respectively, $p\in v_{\mathbb{N}}X\setminus X$) is such that $\tilde{f}(p)\in Y$. There exists $x_p\in X$ such that $\tilde{f}(p)=f(x_p)$. The set $K=f^{-1}[f(x_p)]$ is compact in $\mathbf{X}$, so there exist sets $Z_1, Z_2\in\mathcal{Z}(v_{\mathbb{R}}\mathbf{X})$ (respectively, $Z_1, Z_2\in \mathcal{CO}(v_{\mathbb{N}}\mathbf{X})$) such that $p\in Z_1$, $K\subseteq Z_2$ and $Z_1\cap Z_2=\emptyset$. Since $f(x_p)\notin f[X\cap Z_1]$, there exists $Z_3\in\mathcal{Z}(\mathbf{Y})$ (respectively, $Z_3\in\mathcal{CO}_{\delta}(\mathbf{Y})$) such that $f(x_p)\in Z_3\subseteq Y\setminus f[X\cap Z_1]$. There exists $g\in C(v_{\mathbb{R}}\mathbf{Y}, [0, 1])$ (respectively, $G\in \mathcal{CO}_{\delta}(v_{\mathbb{N}}\mathbf{Y})$) such that $Z_3=g^{-1}[\{0\}]\cap Y$ (respectively, $Z_3=G\cap Y$). Let $Z_4=(g\circ \tilde{f})^{-1}[\{0\}]$ (respectively, $Z_4=\tilde{f}^{-1}[G]$).  Then $p\in Z_4\in\mathcal{Z}(v_{\mathbb{R}}\mathbf{X})$ (respectively, $p\in Z_4\in\mathcal{CO}_{\delta}(v_{\mathbb{N}}\mathbf{X})$) and $Z_4\cap Z_1\cap X=\emptyset$. Hence, for $Z_5=Z_4\cap Z_1$, we have $Z_5\in\mathcal{Z}(v_{\mathbb{R}}\mathbf{X})$ (respectively, $Z_5\in\mathcal{CO}_{\delta}(v_{\mathbb{N}}\mathbf{X})$) and $\emptyset\neq Z_5\subseteq v_{\mathbb{R}}X\setminus X$ (respectively, $\emptyset\neq Z_5\subseteq v_{\mathbb{N}}X\setminus X$).  To see that this is impossible, one can notice that the subspace $S=v_{\mathbb{R}}X\setminus Z_5$ (respectively, $S=v_{\mathbb{N}}X\setminus Z_5$) is realcompact (respectively, $\mathbb{N}$-compact) by Proposition \ref{s4:p7} (respectively, and Lemma \ref{s6:l25}); moreover, since $X\subseteq S\subseteq v_{\mathbb{R}}X$ (respectively, $X\subseteq S\subseteq v_{\mathbb{N}}X$), the equality $S=v_{\mathbb{R}}X$ (respectively, $S=v_{\mathbb{N}}X$) holds. 
\end{proof}

For a topology $\tau$ on $X$, let $\tau_{c_{\delta}}$ be the topology on $X$ generated by the family $\mathcal{CO}_{\delta}(\langle X, \tau\rangle)$. Obviously, $\tau_{c_{\delta}}\subseteq \tau_z$.

\begin{corollary}
	\label{s6:c27}
	$[\mathbf{ZF}]$ Let $\mathbf{X}=\langle X, \tau\rangle$ be a realcompact $($respectively, $\mathbb{N}$-compact$)$ space. Let $\tau^{\ast}$ be a topology on $X$ such that $\langle X, \tau^{\ast}\rangle$ is completely regular $($respectively, zero-dimensional$)$ and $\tau\subseteq\tau^{\ast}\subseteq \tau_{z}$ $($respectively,  $\tau\subseteq\tau^{\ast}\subseteq \tau_{c_{\delta}}$$)$. Then $\langle X, \tau^{\ast}\rangle$ is realcompact $($respectively, $\mathbb{N}$-compact$)$.
\end{corollary}
\begin{proof}
	It suffices to observe that the identity map $\text{id}_{X}:\langle X, \tau^{\ast}\rangle\to\langle X, \tau\rangle$ is $z$-perfect (respectively, $c_{\delta}$-perfect), so $\langle X, \tau^{\ast}\rangle$ is realcompact (respectively, $\mathbb{N}$-compact) by Theorem \ref{s6:t26}.
\end{proof}

\section{Baire sets in realcompact spaces}
\label{s7}
In what follows, for a set $X$, a \emph{field of subsets} on $X$ is a family $\mathfrak{M}\subseteq\mathcal{P}(X)$ such that $\emptyset\in \mathfrak{M}$ and, for all $A,B\in\mathfrak{M}$, $A\cup B\in\mathfrak{M}$ and $A\setminus B\in\mathfrak{M}$. If $\mathfrak{M}$ is a field on $X$ such that $\mathfrak{M}_{\sigma}\subseteq\mathfrak{M}$, then $\mathfrak{M}$ is called a $\sigma$-field of subsets of $X$. In \cite{hal}, fields on $X$ are called \emph{rings} (or \emph{Boolean rings}, and $\sigma$-fields are called $\sigma$-\emph{rings}).

If $\mathcal{A}$ is a family of subsets of a set $X$, then the intersection of all fields (respectively, $\sigma$-fields) on $X$ containing $\mathcal{A}$ is called the \emph{field} (respectively, $\sigma$-\emph{field}) \emph{generated by} $\mathcal{A}$.

For a topological space $\mathbf{X}$, we denote by $\mathfrak{Ba}(\mathbf{X})$ the $\sigma$-field of subsets of $X$ generated by $\mathcal{Z}(\mathbf{X})$. Members of $\mathfrak{Ba}(\mathbf{X})$ are called \emph{Baire sets} in $\mathbf{X}$. We denote by $\mathfrak{Ba}_0(\mathbf{X})$ the $\sigma$-field of subsets of $\mathbf{X}$ generated by $\mathcal{CO}(\mathbf{X})$. Members of $\mathfrak{Ba}_0(\mathbf{X})$ are called \emph{zero-Baire sets} of $\mathbf{X}$.

In \cite{neg}, Negrepontis proved that it holds in $\mathbf{ZFC}$ that if $\mathbf{X}$ is a realcompact space, then, for every $S\in\mathfrak{Ba}(\mathbf{X})$, the subspace $\mathbf{S}$ of $\mathbf{X}$ is realcompact. In this section, we discuss this result in the absence of $\mathbf{AC}$ and show its counterpart for $\mathbb{N}$-compact spaces.

\begin{definition}
	\label{s7:d1}
	Let $\mathbf{X}$ be a topological space and let $A\subseteq X$.
	\begin{enumerate} 
		\item[(i)] (Cf. \cite[Definition 2.4]{neg}.) The set $A$ (equivalently, the subspace $\mathbf{A}$ of $\mathbf{X}$) is called \emph{r-embedded} in $\mathbf{X}$ if, for every $p\in X\setminus A$, there exists $Z\in\mathcal{Z}(\mathbf{X})$ with $p\in Z$ and $Z\cap A=\emptyset$.
		\item[(ii)] The set $A$ (equivalently, the subspace $\mathbf{A}$ of $\mathbf{X}$) is called $r_{\mathbb{N}}$-\emph{embedded} in $X$ if, for every $p\in X\setminus A$, there exists $f\in C(\mathbf{X},\mathbb{N}(\infty))$ such that $p\in f^{-1}[\{\infty\}]$ and $f^{-1}[\{\infty\}]\cap A=\emptyset$.
	\end{enumerate}
\end{definition}

In the following definition, we modify the concept of an $r$-compactification from \cite[Definition 2.1]{neg} (cf. also \cite{mr0}).

\begin{definition}
	\label{s7:d2}
	Let $\mathbf{X}$ be a topological space.
	\begin{enumerate}
		\item[(i)] Suppose that $\mathcal{F}\subseteq C^{\ast}(\mathbf{X})$ and $\mathcal{F}\in\mathcal{E}(\mathbf{X})$. If $e_{\mathcal{F}}[X]$ is $r$-embedded in $e_{\mathcal{F}}\mathbf{X}$, then $e_{\mathcal{F}}\mathbf{X}$ will be called an $r$-\emph{extension} of $\mathbf{X}$.
		\item[(ii)] Suppose that $\mathcal{F}\subseteq C(\mathbf{X}, \mathbf{2})$ and $\mathcal{F}\in\mathcal{E}(\mathbf{X}, \mathbf{2})$. If  $e_{\mathcal{F}}[X]$ is $r_{\mathbb{N}}$-embedded in $e_{\mathcal{F}}\mathbf{X}$, then $e_{\mathcal{F}}\mathbf{X}$ will be called an $r_{\mathbb{N}}$-\emph{extension} of $\mathbf{X}$.
	\end{enumerate}
\end{definition}

The following proposition is a modification of \cite[Proposition 2.2]{neg}.

\begin{proposition} 
	\label{s7:p3}
	$[\mathbf{ZF}]$ A topological space is realcompact $($respectively, $\mathbb{N}$-compact$)$ if and only if it has an $r$-extension $($respectively, $r_{\mathbb{N}}$-extension$)$.
\end{proposition}
\begin{proof}
	If $\mathbf{X}$ is realcompact (respectively, $\mathbb{N}$-compact), then $\beta^f\mathbf{X}$ (respectively, $e^{\ast}_{2}\mathbf{X}$ is an $r$-extension (respectively, $r_{\mathbb{N}}$-extension) of $\mathbf{X}$ by Theorem \ref{s4:t16}. On the hand, given an $r$-extension (respectively, $r_{\mathbb{N}}$-extension) of $\mathbf{X}$, there is a continuous mapping $h: \beta^f\mathbf{X}\to e_{\mathcal{F}}\mathbf{X}$ (respectively, $h: e^{\ast}_{2}\mathbf{X}\to e_{\mathcal{F}}\mathbf{X}$) such that $h\circ e_{\beta^f}=e_{\mathcal{F}}$ (respectively, $h\circ e^{\ast}_{2}=e_{\mathcal{F}}$). For this mapping $h$, we have $h[\beta^{f}X\setminus e_{\beta^f}[X]]\subseteq e_{\mathcal{F}}X\setminus e_{\mathcal{F}}[X]$ (respectively, $h[e_{2}^{\ast}X\setminus e_{2}^{\ast}[X]]\subseteq e_{\mathcal{F}}X\setminus e_{\mathcal{F}}[X]$) (cf. \cite[Lemma 3.5.6]{en}). Now, one can easily deduce that if $\mathbf{X}$ has an $r$-extension (respectively, $r_{\mathbb{N}}$-extension), then $\beta^f\mathbf{X}$ (respectively, $e^{\ast}_{2}\mathbf{X}$) is an $r$-extension (respectively, $r_{\mathbb{N}}$-extension) of $\mathbf{X}$, so $\mathbf{X}$ is realcompact (respectively, $\mathbb{N}$-compact) by Theorem \ref{s4:t16}.
\end{proof}

\begin{corollary}
	\label{s7:c4}
	$[\mathbf{ZF}]$ Every $r$-embedded $($respectively, $r_{\mathbb{N}}$-embedded$)$ subspace of a realcompact $($respectively, $\mathbb{N}$-compact$)$ space is realcompact $($respectively, $\mathbb{N}$-compact$)$.
\end{corollary}

\begin{proof}
	Let $S$ be an $r$-embedded (respectively, $r_{\mathbb{N}}$-embedded) subset of a realcompact (respectively, $r_{\mathbb{N}}$-compact) space $\mathbf{X}$. Let $e\mathbf{S}$ be the closure of $e_{\beta^f}[S]$ (respectively, of $e^{\ast}_{2}[S]$) in $\beta^f\mathbf{X}$ (respectively, in $e^{\ast}_{2}\mathbf{X}$). Then $e\mathbf{S}$ is an $r$-extension (respectively, $r_{\mathbb{N}}$-extension) of $\mathbf{S}$. Therefore, $\mathbf{S}$ is realcompact (respectively, $\mathbb{N}$-compact) by Proposition \ref{s7:p3}.
\end{proof}

\begin{proposition}
	\label{s7:p6}
	$[\mathbf{ZF}]$ Let $A$ be a subset of a topological space $\mathbf{X}$.
	\begin{enumerate}
		\item[(i)] The set $A$ is $r_{\mathbb{N}}$-embedded in $\mathbf{X}$ if and only if, for every $p\in X\setminus A$,  there exists $C\in \mathcal{CO}_{\delta}(\mathbf{X})$ with $p\in C$ and $C\cap A=\emptyset$.
		
		\item[(ii)] If $\mathbf{X}$ is a zero-dimensional $T_1$-space, then $\mathbf{X}$ is $\mathbb{N}$-compact if and only if, for every $p\in e^{\ast}_{2}X\setminus e^{\ast}_{2}[X]$, there exists $C\in\mathcal{CO}_{\delta}(e^{\ast}_{2}\mathbf{X})$ such that $p\in C$ and $C\cap e^{\ast}_{2}[X]=\emptyset$.
	\end{enumerate}
\end{proposition}
\begin{proof}
	That (i) holds is a consequence of Lemma \ref{s6:l25}. It follows from (i) and Theorem \ref{s4:t16} that (ii) also holds.
\end{proof}

\begin{lemma}
	\label{s7:l7}
	$[\mathbf{ZF}]$ Let $\mathcal{A}$ be a non-empty family of $r$-embedded $($respectively, $r_{\mathbb{N}}$-embedded$)$ subsets of a topological space $\mathbf{X}$. Then the following conditions are satisfied:
	\begin{enumerate}
		\item[(i)]  every member of $\mathcal{A}_{\delta}$ is $r$-embedded (respectively, $r_{\mathbb{N}}$-embedded) in $\mathbf{X}$;
		\item[(ii)] $\mathbf{CMC}$ implies that every member of $\mathcal{A}_{\sigma}$ is $r$-embedded $($respectively, $r_{\mathbb{N}}$-embedded$)$ in $\mathbf{X}$.
	\end{enumerate}
\end{lemma}

\begin{proof}
	(i) Arguing similarly to the proof of Lemma 3.4 in \cite{neg} and applying our Proposition \ref{s7:p6}(i), one can easily show that (i) holds.
	
	(ii) Now, assume that $\{A_n: n\in\omega\}$ is a subfamily of $\mathcal{A}$ such that, for every $n\in\omega$, the set $A_n$ is $r_{\mathbb{N}}$-embedded in $\mathbf{X}$. Let $A=\bigcup\limits_{n\in\omega}A_n$ and let $p\in X\setminus A$. For every $n\in\omega$, let $\mathcal{U}_n=\{U\in \mathcal{CO}_{\delta}(\mathbf{X}): p\in U\wedge U\cap A_n=\emptyset\}$. It follows from Proposition \ref{s7:p6}(i) that, for every $n\in\omega$, $\mathcal{U}_n\neq\emptyset$. By $\mathbf{CMC}$, there is a family $\{\mathcal{V}_n: n\in\omega\}$ such that, for every $n\in\omega$, $\mathcal{V}_n$ is a non-empty finite subset of $\mathcal{U}_n$. For every $n\in\omega$, let $C_n=\bigcap\mathcal{V}_n$. Clearly, for every $n\in\omega$, $C_n\in\mathcal{CO}_{\delta}(\mathbf{X})$, $p\in C_n$ and $C_n\cap A_n=\emptyset$. Let $C=\bigcap\limits_{n\in\omega}C_n$. By $\mathbf{CMC}$ and Lemma \ref{s6:l4}, $C\in\mathcal{CO}_{\delta}(\mathbf{X})$. Obviously, $p\in C$ and $C\cap A=\emptyset$. In view of Proposition \ref{s7:p6}(i), $A$ is $r_{\mathbb{N}}$-embedded in $\mathbf{X}$. Let us omit a similar proof that $A$ is $r$-embedded in $\mathbf{X}$ when, for every $n\in\omega$, $A_n$ is $r$-embedded in $\mathbf{X}$.
\end{proof}

\begin{definition}
	\label{s7:d8}
	Let $X$ be a set and let  $\emptyset\neq \mathcal{A}\subseteq\mathcal{P}(X)$. For every $\alpha\in\omega_1$, we define, by transfinite induction, the sets $\mathcal{F}_{\alpha}(\mathcal{A})$ as follows. We put  $\mathcal{F}_0(\mathcal{A})=\mathcal{A}$ and, for every non-zero ordinal $\alpha\in\omega_1$, we put:
	$$
	\mathcal{F}_{\alpha}(\mathcal{A})=\begin{cases} [\bigcup\limits_{\gamma\in\alpha}\mathcal{F}_{\gamma}(\mathcal{A})]_{\delta} &\text{ when $\alpha$ is even;}\\
		[\bigcup\limits_{\gamma\in\alpha} \mathcal{F}_{\gamma}(\mathcal{A})]_{\sigma} &\text{ when $\alpha$ is odd.}\end{cases}
	$$
	Furthermore, $\sigma(\mathcal{A})=\bigcup\limits_{\alpha\in\omega_1}\mathcal{F}_{\alpha}(\mathcal{A})$.
\end{definition}

Let us omit a standard, relatively simple proof of the following lemma:
\begin{lemma}
	\label{s7:l9}
	$[\mathbf{ZF}]$ Suppose that $\mathcal{A}$ is a family of subsets of a set $X$ such that $\emptyset\in \mathcal{A}$ and, for every $A\in\mathcal{A}$, $X\setminus A\in\mathcal{A}$. If $\omega_1$ is regular, then $\sigma(\mathcal{A})$ is a $\sigma$-field of subsets of $X$.
\end{lemma}

\begin{theorem}
	\label{s7:t10}
	It is true in $\mathbf{ZF}$ that $\mathbf{CMC}$ implies that, for every topological space $\mathbf{X}$, the following conditions are satisfied:
	\begin{enumerate}
		\item[(i)] if $\mathbf{X}$ is realcompact, then, for every $S\in\mathfrak{Ba}(\mathbf{X})$, the subspace $\mathbf{S}$ of $\mathbf{X}$ is realcompact $($cf. \cite[Theorem 3.8]{neg}$)$;
		\item[(ii)] if $\mathbf{X}$ is $\mathbb{N}$-compact, then, for every $S\in\mathfrak{Ba}_0(\mathbf{X})$, the subspace $\mathbf{S}$ of $\mathbf{X}$ is $\mathbb{N}$-compact.
	\end{enumerate}
\end{theorem}
\begin{proof}
	Let us assume $\mathbf{ZF+CMC}$.  Then $\omega_1$ is regular (see \cite{hr1}). Given a non-empty topological space $\mathbf{X}$, we put $\mathcal{A}=\mathcal{Z}(\mathbf{X})\cup\mathcal{Z}^{c}(\mathbf{X})$. It follows from Lemma \ref{s7:l9} that $\sigma(\mathcal{A})=\mathfrak{Ba}(\mathbf{X})$ and $\sigma(\mathcal{CO}(\mathbf{X}))=\mathfrak{Ba}_0(\mathbf{X})$.
	
	(i) Suppose that $\mathbf{X}$ is realcompact. It follows from Lemma \ref{s7:l7} that every member of $\mathcal{A}$ is $r$-embedded in $\mathbf{X}$.  Using transfinite induction and applying Lemmas \ref{s7:l7} and \ref{s7:l9}, we can show that every Baire set in $\mathbf{X}$ is $r$-embedded in $\mathbf{X}$.
	
	(ii) Suppose that $\mathbf{X}$ is $\mathbb{N}$-compact. Then every clopen set in $\mathbf{X}$ is $r_{\mathbb{N}}$-embedded in $\mathbf{X}$.  Using transfinite induction and Proposition \ref{s7:p6}, we infer from Lemmas \ref{s7:l7} and \ref{s7:l9} that every zero-Baire set in $\mathbf{X}$ is $r_{\mathbb{N}}$-embedded in $\mathbf{X}$.
	
	To conclude the proof, it suffices to apply Corollary \ref{s7:c4}.
\end{proof}

Statement (i) of Theorem \ref{s7:t10} was shown to be true in $\mathbf{ZFC}$ in \cite[Theorem 3.8]{neg}, so we have weakened the assumption of $\mathbf{AC}$ to $\mathbf{CMC}$ to get the main result of Section 3 of \cite{neg}.  That (ii) of Theorem \ref{s7:t10} holds in $\mathbf{ZF+CMC}$ is a quite new result. In $\mathbf{ZFC}$, a measure-theoretic proof of statement (i) of our Theorem \ref{s7:t10} is given in \cite[pp. 487--488, Chapter VIII.4, the proof of Corollary 1]{nag} but we believe that this proof is not complete. In the forthcoming section, by modifying the idea of the proof of Corollary 1 of Chapter VIII.4 in \cite{nag}, we give a detailed measure-theoretic alternative proof of statement (ii) of Theorem \ref{s7:t10} in $\mathbf{ZF+CAC}$, as well as a measure-theoretic proof of statement (i) of Theorem \ref{s7:t10} in $\mathbf{ZF+CMC}$. 

\section{Measure-theoretic aspects of $\mathbb{N}$-compactness and realcompactness}
\label{s8}

It is known, for instance, from \cite[Chapter VIII.4]{nag} that some topological properties, among them also realcompactness (see \cite[Theorem VIII.14]{nag}) can be characterized in terms of measures in $\mathbf{ZFC}$. Up to now, there has been no attempt to investigate this measure-theoretic approach to realcompactness in the absence of the Axiom of Choice. To fill this gap, a characterization of realcompactness and $\mathbb{N}$-compactness by 2-valued measures is given in $\mathbf{ZF}$ in this section. 

For a family $\mathcal{A}$ of subsets of a set $X$, let $\sigma_f(\mathcal{A})$ denote the $\sigma$-field generated by $\mathcal{A}$. Let $\mathfrak{F}(\mathcal{A})$ be the field generated by $\mathcal{A}$.

\begin{remark}
	\label{s8:r1}
	Let $\mathcal{A}$ be a family of subsets of a set $X$ such that $\emptyset\in\mathcal{A}$ and, for every $A\in\mathcal{A}$, $X\setminus A\in\mathcal{A}$. It follows from Lemma \ref{s7:l9} that it holds in $\mathbf{ZF}$ that if $\omega_1$ is regular, then $\sigma(\mathcal{A})=\sigma_f(\mathcal{A})$.
\end{remark}

\begin{definition}
	\label{s8:d2}
	\begin{enumerate}
		\item[(a)] Let $\mathfrak{M}$ be a field of subsets of a set $X$. 
		\begin{enumerate}
			\item[(i)] A measure on $\mathfrak{M}$ is a function $\mu: \mathfrak{M}\to [0, +\infty)$  such that , for every pair $A,B$ of disjoint members of $\mathfrak{M}$, $\mu(A\cup B)=\mu(A)+\mu(B)$ (that is, $\mu$ is \emph{finitely additive}).
			\item[(ii)] A measure $\mu$ on $\mathfrak{M}$ is called \emph{countably additive} or $\sigma$-\emph{additive} if, for every disjoint family $\{A_n: n\in\omega\}$ of members of $\mathfrak{M}$ such that $\bigcup\limits_{n\in\omega}A_n\in\mathfrak{M}$, the following equality holds: $\mu(\bigcup\limits_{n\in\omega}A_n)=\sum\limits_{n=0}^{+\infty}\mu(A_n)$. 
			\item[(iii)] A measure $\mu$ on $\mathfrak{M}$ is called the \emph{Dirac measure determined by a point} $x\in X$ if:
			$$
			\mu(A)=\begin{cases} 1 &\text{ if } x\in A\in\mathfrak{M};\\
				0 & \text{ if } x\notin A\in\mathfrak{M}.\end{cases}
			$$
			A \emph{Dirac measure} on $\mathfrak{M}$ is the Dirac measure on $\mathfrak{M}$ determined by some point of $X$. 
			\item[(iv)] A 2-valued measure on $\mathfrak{M}$ is a measure $\mu$ on $\mathfrak{M}$ such that $\mu[\mathfrak{M}]=\{0, 1\}$. 
			\item[(v)] If $\mu$ is a measure on $\mathfrak{M}$, then $\mu^{\ast}:\mathcal{P}(X)\to [0,+\infty)$ is a function defined as follows:  for every $A\subseteq X$, 
			$$\mu^{\ast}(A)=\inf\{\sum\limits_{n=0}^{+\infty}\mu(C_n): ((\forall n\in\omega)C_n\in\mathfrak{M}(\mathbf{X}))\wedge A\subseteq\bigcup\limits_{n\in\omega}C_n\}.$$
		\end{enumerate}
		\item[(b)] Let $\mathbf{X}$ be a topological space. 
		\begin{enumerate}
			\item[(vi)] A $z$-\emph{measure} on $\mathbf{X}$ is a measure on $\mathfrak{F}(\mathcal{Z}(\mathbf{X}))$. A \emph{Dirac} $z$-\emph{measure} on $\mathbf{X}$ is a Dirac measure on $\mathfrak{F}(\mathcal{Z}(\mathbf{X}))$. A 2-valued $z$-measure $\mu$ on $\mathbf{X}$ is a $z$-measure on $\mathbf{X}$ whose values are from the set $\{0, 1\}$ and $\mu(X)=1$. 
			\item[(vii)] A $c$-\emph{measure} on $\mathbf{X}$ is a measure on $\mathcal{CO}(\mathbf{X})$. A \emph{Dirac} $c$-\emph{measure} on $\mathbf{X}$ is a Dirac measure on $\mathcal{CO}(\mathbf{X})$. A 2-valued $c$-measure $\mu$ on $\mathbf{X}$ is a $c$-measure on $\mathbf{X}$ whose values are from the set $\{0, 1\}$ and $\mu(X)=1$. 
			\item[(viii)] (Cf. \cite[p. 483]{nag}.) A $z$-measure $\mu$ on $\mathbf{X}$ is called \emph{regular} if, for every $A\in \mathfrak{F}(\mathcal{Z}(\mathbf{X}))$, the following equality holds:
			$$\mu(A)=\inf\{\mu(V): A\subseteq V\in\mathcal{Z}^c(\mathbf{X})\}.$$
		\end{enumerate} 
	\end{enumerate}
\end{definition}

It is obvious that the following proposition holds:

\begin{proposition}
	\label{s8:p3}
	$[\mathbf{ZF}]$ A $z$-measure $\mu$ on a topological space $\mathbf{X}$ is regular if and only if, for every $A\in \mathfrak{F}(\mathcal{Z}(\mathbf{X}))$, the following equality holds:
	$$\mu(A)=\sup\{\mu(Z): Z\subseteq A\wedge Z\in \mathcal{Z}(\mathbf{X})\}.$$
\end{proposition}

The following proposition can be deduced from standard proofs of well-known facts in measure theory (see, e.g., \cite[pp. 54--57]{hal} and \cite{yh}), so we omit its proof.

\begin{proposition}
	\label{s8:p4}
	$[\mathbf{ZF}]$ Let $\mu$ be a measure on a field $\mathfrak{M}$ of subsets of a set $X$. For every $E\subseteq X$, let $\mathfrak{C}(E)$ be the family of all $C\in\mathfrak{M}^{\omega}$ such that $E\subseteq\bigcup\limits_{i\in\omega}C(i)$ and, for every pair $i,j$ of distinct members of $\omega$, $C(i)\cap C(j)=\emptyset$.
	\begin{enumerate}
		\item[(i)] $\mu$ is countably additive if and only if, for every family $\{A_n: n\in\mathbb{N}\}$ of members of $\mathfrak{M}$ such that $\bigcap\limits_{n\in\mathbb{N}}A_n=\emptyset$ and, for every $n\in\mathbb{N}$, $A_{n+1}\subseteq A_n$, we have $\lim\limits_{n\to+\infty}\mu(A_n)=0$.
		\item[(ii)] If $\mu$ is countably additive then, for every $E\subseteq X$, 
		$$\mu^{\ast}(E)=\inf\{\sum\limits_{i=0}^{+\infty}\mu(C(i)): C\in\mathfrak{C}(E)\}.$$
		\item[(iii)] If $\mu$ is a countably additive  2-valued measure, then, for every $E\subseteq X$, $\mu^{\ast}(E)\in\{0, 1\}$.
		\item[(iv)] $\mathbf{CAC}$ implies that $\mu^{\ast}$ is an outer measure, and $\mu^{\ast}\upharpoonright \sigma_f(\mathfrak{M})$ is a countably additive measure on the $\sigma$-field $\sigma_f(\mathfrak{M})$ such that, for every $A\in\mathfrak{M}$, $\mu^{\ast}(A)=\mu(A)$.
	\end{enumerate}
\end{proposition}

The following proposition is strictly related to \cite[Proposition A, Chapter VIII.4, p. 484]{nag}.
\begin{proposition}
	\label{s8:p5}
	$[\mathbf{ZF}]$ Let $\mu$ be a $z$-measure on a topological space $\mathbf{X}$. Let (i)--(iii) denote the following conditions, respectively:
	\begin{enumerate}
		\item[(i)] $\mu$ is countably additive;
		\item[(ii)] $\mu$ is regular and, for every increasing sequence $(U_n)_{n\in\omega}$ of co-zero sets of $\mathbf{X}$ such that $\bigcup_{n\in\omega}U_n=X$, $\lim\limits_{n\to+\infty}\mu(U_n)=\mu(X)$;
		\item[(iii)]  $\mu$ is regular and, for every decreasing sequence $(Z_n)_{n\in\omega}$ of zero-sets of $\mathbf{X}$ such that $\bigcap_{n\in\omega}Z_n=\emptyset$, $\lim\limits_{n\to+\infty}\mu(Z_n)=0$.
	\end{enumerate}
	Then (i) and (iii) are equivalent, (i) implies both (ii) and (iii). Furthermore, $\mathbf{CMC}$ implies that (i) follows from (iii) (so also from (ii)).
\end{proposition}
\begin{proof}
	It is obvious that (ii) and (iii) are equivalent. Assuming (i), in much the same way, as in \cite[p. 484]{nag}, we can prove that $\mu$ is regular. Hence (i) implies both (ii) and (iii) by Proposition \ref{s8:p4}. 
	
	Assuming that both (iii) and  $\mathbf{CMC}$ hold, we show that (i) holds. To this aim, we fix a decreasing sequence $(A_n)_{n\in\omega}$  of members of $\mathfrak{F}(\mathcal{Z}(\mathbf{X}))$ such that $\bigcap_{n\in\omega}A_n=\emptyset$. In view of Proposition \ref{s8:p4}, it suffices to prove that $\lim\limits_{n\to+\infty}\mu(A_n)=0$. We fix a positive real number $\varepsilon$.  For every $n\in\omega$, let $\mathcal{Z}_n=\{Z\in\mathcal{Z}(\mathbf{X}): Z\subseteq A_n\wedge \mu(A_n)\leq\mu(Z)+\frac{\epsilon}{2^{n+1}}\}$. By the regularity of $\mu$, for every $n\in\omega$, $\mathcal{Z}_n\neq\emptyset$. It follows from $\mathbf{CMC}$ that there exists a family $\{\mathcal{C}_n: n\in\omega\}$ of non-empty finite sets such that, for every $n\in\omega$, $\mathcal{C}_n\subseteq\mathcal{Z}_n$. For a given $n\in\omega$, let $Z_n=\bigcup\mathcal{C}_n$. Then $Z_n\in\mathcal{Z}(\mathbf{X})$, $Z_n\subseteq A_n$ and $\mu(A_n)\leq\mu(Z_n)+\frac{\varepsilon}{2^{n+1}}$. We put $E_n=\bigcap_{k\in n+1}Z_k$. Then $\mu(A_n\setminus E_n)\leq\sum_{k\in n+1}\mu(A_k\setminus Z_k)\leq\varepsilon$. Hence $\mu(A_n)\leq\mu (E_n)+\varepsilon$. Since $\bigcap_{n\in\omega}E_n=\emptyset$ and, for every $n\in\omega$, $E_n\in\mathcal{Z}(\mathbf{X})$ and $E_{n+1}\subseteq E_n$, it follows from (iii) that $\lim\limits_{n\to+\infty}\mu(E_n)=0$. In consequence, $\lim\limits_{n\to+\infty}\mu(A_n)\leq\varepsilon$. This implies that $\lim\limits_{n\to+\infty}\mu(A_n)=0$. It follows from Proposition \ref{s8:p4} that $\mu$ is countably additive.
\end{proof}

\begin{definition}
	\label{s8:d6}
	We say that a $z$-measure $\mu$ on a topological space $\mathbf{X}$ is \emph{weakly} $\sigma$-\emph{additive} if $\mu$ is regular and, for every functionally accessible family $\{Z_n: n\in\omega\}$ of zero-sets of $\mathbf{X}$ such that $\bigcap_{n\in\omega}Z_n=\emptyset$ and, for every $n\in\omega$, $Z_{n+1}\subseteq Z_n$, we have $\lim\limits_{n\to+\infty}\mu(Z_n)=0$.
\end{definition}

\begin{proposition}
	\label{s8:p7}
	$[\mathbf{ZF}]$ Let $\mathcal{F}$ be a $z$-ultrafilter on a non-empty topological space $\mathbf{X}$. Let $\mu_{\mathcal{F}}:\mathfrak{F}(\mathcal{Z}(\mathbf{X}))\to\{0,1\}$ be the function defined as follows. For every $A\in\mathfrak{F}(\mathcal{Z}(\mathbf{X}))$:
	$$
	\mu_{\mathcal{F}}(A)=\begin{cases} 1 &\text{ if there exists } Z\in\mathcal{F} \text{ such that } Z\subseteq A;\\
		0 &\text{ otherwise}. \end{cases}
	$$
	Then $\mu_{\mathcal{F}}$ is a regular 2-valued $z$-measure on $\mathbf{X}$. Furthermore, if $\mathcal{F}$ has the weak countable intersection property, then the measure $\mu_{\mathcal{F}}$ is weakly $\sigma$-additive.
\end{proposition}
\begin{proof}
	That $\mu_{\mathcal{F}}$ is a regular $z$-measure can be shown in much that same way, as Proposition D in \cite[p. 485]{nag}. 
	
	Suppose that $\mathcal{F}$ has the weak countable intersection property. Fix a functionally accessible family $\{Z_n: n\in\omega\}$ of zero-sets of $\mathbf{X}$ such that $\bigcap_{n\in\omega}Z_n=\emptyset$ and, for every $n\in\omega$, $Z_{n+1}\subseteq Z_n$. Since $\mathcal{F}$ has the weak countable intersection property and $\bigcap_{n\in\omega}Z_n=\emptyset$, there exists $n_0\in\omega$ such that $Z_{n_0}\notin\mathcal{F}$. Then, for every $n\in\omega$ with $n_0\in n$, we have $\mu_{\mathcal{F}}(Z_n)=0$, so $\lim\limits_{n\to+\infty}\mu_{\mathcal{F}}(Z_n)=0$.
\end{proof}

The following theorem is a $\mathbf{ZF}$-modification of Theorem VIII.4 in \cite[p. 487]{nag}:

\begin{theorem}
	\label{s8:t8}
	$[\mathbf{ZF}]$ A Tychonoff space $\mathbf{X}$ is realcompact if and only if every weakly $\sigma$-additive 2-valued $z$-measure on $\mathbf{X}$ is a Dirac $z$-measure on $\mathbf{X}$.
\end{theorem} 

\begin{proof}
	Suppose that $\mathbf{X}$ is a realcompact space, and $\mu$ is a weakly $\sigma$-additive 2-valued $z$-measure on $\mathbf{X}$. Arguing in much the same way, as in the proof of Theorem VIII.4 in \cite{nag}, we can show that $\mathcal{F}=\{ Z\in\mathcal{Z}(\mathbf{X}): \mu(Z)=1\}$ is a $z$-ultrafilter in $\mathbf{X}$ such that $\mathcal{F}$ has the weak countable intersection property. By Theorem \ref{s6:t10}, there exists $x\in\bigcap\mathcal{F}$. Using the same arguments, as in the proof of Theorem VIII.4 in \cite{nag}, we can show that $\mu$ is the Dirac $z$-measure on $\mathbf{X}$ determined by $x$.
	
	Now suppose that $\mathbf{X}$ is a Tychonoff space such that every 2-valued weakly $\sigma$-additive $z$-measure on $\mathbf{X}$ is a Dirac $z$-measure on $\mathbf{X}$. Let $\mathcal{F}$ be an arbitrary $z$-ultrafilter in $\mathbf{X}$ such that $\mathcal{F}$ has the weak countable intersection property. Then the measure $\mu_{\mathcal{F}}$, defined in Proposition \ref{s8:p7}, is a weakly $\sigma$-additive 2-valued $z$-measure on $\mathbf{X}$. Suppose that $x\in X$ is such that $\mu$ is the Dirac $z$-measure determined by $x$ on $\mathbf{X}$. Then $x\in\bigcap\mathcal{F}$. Hence $\mathbf{X}$ is realcompact by Theorem \ref{s6:t10}.
\end{proof}

\begin{corollary}
	\label{s8:c9}
	$[\mathbf{ZF}]$ If $\mathbf{X}$ is a realcompact space, then every 2-valued weakly $\sigma$-additve $z$-measure on $\mathbf{X}$ is countably additive.
\end{corollary}

\begin{problem}
	\label{s8:q10}
	May a weakly $\sigma$-additive $z$-measure on a realcompact space $\mathbf{X}$ fail to be countably additive in a model of $\mathbf{ZF}$?
\end{problem}

The following corollary follows directly from  Proposition \ref{s8:p5}, Corollary \ref{s6:c11} and the proof of Theorem \ref{s8:t8}:

\begin{corollary}
	\label{s8:c11}
	$[\mathbf{ZF}]$ $\mathbf{CMC}$ implies that a Tychonoff space $\mathbf{X}$ is realcompact if and only if every 2-valued countably additive $z$-measure on $\mathbf{X}$ is a Dirac $z$-measure on $\mathbf{X}$.
\end{corollary}

Now, let us pass to a characterization of $\mathbb{N}$-compactness in terms of 2-valued measures.

\begin{theorem}
	\label{s8:t12}
	$[\mathbf{ZF}]$ For every non-empty zero-dimensional $T_1$-space $\mathbf{X}$, it holds that $\mathbf{X}$ is $\mathbb{N}$-compact if and only if every countably additive 2-valued  $c$-measure  on $\mathbf{X}$ is a Dirac $c$-measure on $\mathbf{X}$.
\end{theorem}
\begin{proof}
	Let $\mathbf{X}$ be a non-empty zero-dimensional $T_1$-space. Suppose that $\mu$ is a countably additive 2-valued $c$-measure on $\mathbf{X}$. One can easily verify that
	$$\mathcal{F}=\{ C\in\mathcal{CO}(\mathbf{X}): \mu(C)=1\}$$
	is an ultrafilter in $\mathcal{CO}(\mathbf{X})$. 
	To show that $\mathcal{F}$ has the countable intersection property, suppose that $\{C_n: n\in\omega\}$ is a family of members of $\mathcal{F}$ such that, for every $n\in\omega$, $C_{n+1}\subseteq C_n$. If $\bigcap\limits_{n\in\omega}C_n=\emptyset$, by Proposition \ref{s8:p4}(i), it follows from the countable additivity of $\mu$ that $0=\lim\limits_{n\to+\infty}\mu(C_n)$, which is impossible because, for every $n\in\omega$, $\mu(C_n)=1$. This shows that $\mathcal{F}$ has the countable intersection property. Assuming that $\mathbf{X}$ is $\mathbb{N}$-compact, we deduce from Theorem \ref{s6:t1} that $\mathcal{F}$ is fixed. Hence, there exists $p\in X$ such that $\bigcap\mathcal{F}=\{p\}$. This implies that $\mu$ is the Dirac $c$-measure on $\mathbf{X}$ determined by $p$.
	
	Conversely, suppose that $\mathbf{X}$ is such that every countably additive 2-valued $c$-measure on $\mathbf{X}$ is a Dirac $c$-measure on $\mathbf{X}$.  Let $\mathcal{U}$ be an ultrafilter in $\mathcal{CO}(\mathbf{X})$ with the countable intersection property. We define a mapping $\mu_{\mathcal{U}}:\mathcal{CO}(\mathbf{X})\to\{0, 1\}$ as follows:
	$$
	\mu_{\mathcal{U}}(A)=\begin{cases} 1 &\text{ if $A\in\mathcal{U}$;}\\
		0 &\text{ if $X\setminus A\in\mathcal{U}$.}\end{cases}
	$$
	Clearly, $\mu_{\mathcal{U}}$ is a 2-valued $c$-measure on $\mathbf{X}$. To show that $\mu_{\mathcal{U}}$ is countably additive, we consider any family $\{E_n: n\in\omega\}$ of clopen subsets of $\mathbf{X}$ such that $\bigcap\limits_{n\in\omega}E_n=\emptyset$ and, for every $n\in\omega$, $E_{n+1}\subseteq E_n$. Since $\mathcal{U}$ has the countable intersection property, there exists $n_0\in\omega$ such that $\mu_{\mathcal{U}}(E_{n_0})=0$ and, in consequence, $\lim\limits_{n\to+\infty}\mu_{\mathcal{U}}(E_n)=0$. Hence $\mu_{\mathcal{U}}$ is countably additive by Proposition \ref{s8:p4}(i). By our hypothesis, there exists $q\in X$ such that $\mu_{\mathcal{U}}$ is the Dirac $c$-measure determined by $q$. Then $q\in\bigcap\mathcal{U}$, so $\mathcal{U}$ is fixed. By Theorem \ref{s6:t1}, $\mathbf{X}$ is $\mathbb{N}$-compact.
\end{proof}

Although $\mathbf{CAC}$ implies $\mathbf{CMC}$, to apply the idea of the proof of Corollary 1 of Chapter VIII.4 in \cite{nag}, let us demonstrate in $\mathbf{ZF+CAC}$ an alternative measure-theoretic proof of statement (ii) of our Theorem \ref{s7:t10}. 

\begin{theorem}
	\label{s8:t13}
	$[\mathbf{ZF+CAC}]$
	Every zero-Baire set in an $\mathbb{N}$-compact space is $\mathbb{N}$-compact.
\end{theorem}
\begin{proof}
	In what follows, our set-theoretic assumption is $\mathbf{ZF+CAC}$. Let $S$ be a non-empty zero-Baire set in an $\mathbb{N}$-compact space $\mathbf{X}$. Let $\nu$ be a countably additive  2-valued $c$-measure on the subspace $\mathbf{S}$ of $\mathbf{X}$. We define a function $\mu:\mathcal{CO}(\mathbf{X})\to\{0,1\}$ as follows: for every $A\in\mathcal{CO}(\mathbf{X})$, 
	$$\mu(A)=\nu(A\cap S).$$
	One can easily check that $\mu$ is a countably additive  2-valued $c$-measure on $\mathbf{X}$. By Theorem \ref{s8:t12}, there exists $p\in X$ such that $\mu$ is the Dirac $c$-measure $\delta_p$ on $\mathbf{X}$ determined by $p$.  By Proposition \ref{s8:p4}(iv), for $\mathfrak{M}=\mathcal{CO}(\mathbf{X})$,  $\mu^{\ast}$ is an outer measure, and $\mu^{\ast}\upharpoonright \mathfrak{Ba}_0(\mathbf{X})$ is a countably additive measure on $\mathfrak{Ba}_{0}(\mathbf{X})$ such that, for every $A\in\mathfrak{Ba}_0(\mathbf{X})$, $\mu^{\ast}(A)\in\{0, 1\}$ and, for every $B\in\mathcal{CO}(\mathbf{X})$, $\mu^{\ast}(B)=\mu(B)$.
	
	Consider any $A\in\mathfrak{Ba}_0(\mathbf{X})$. If $p\in A$, it follows from Proposition \ref{s8:p4}(ii) that $\mu^{\ast}(A)=1$. Suppose that $p\notin A$. It has been shown in the proof of Theorem \ref{s7:t10} that every zero-Baire set of $\mathbf{X}$ is $r_{\mathbb{N}}$-embedded in $\mathbf{X}$. Hence, by Proposition \ref{s7:p6}(i), there exists $C\in\mathcal{CO}_{\delta}(\mathbf{X})$ such that $p\in C\subseteq X\setminus A$. Then $\mu^{\ast}(X\setminus A)=1$, so $\mu^{\ast}(A)=0$.  This implies that, for every $A\in\mathfrak{Ba}_0(\mathbf{X})$, the following equality holds:
	$$
	\mu^{\ast}(A)=\begin{cases} 1 &\text{ if $p\in A$;}\\
		0 &\text{ if $p\notin A$}.\end{cases}
	$$
	Furthermore, if $p\notin S$, then there exists a disjoint family $\{C_n: n\in\omega\}$ of clopen sets of $\mathbf{X}$ such that $S\subseteq\bigcup_{n\in\omega}C_n$ and $p\notin\bigcup_{n\in\omega}C_n$. Then $\mu^{\ast}(\bigcup_{n\in\omega}C_n)=0=\sum\limits_{n=0}^{+\infty}\mu(C_n)=\sum\limits_{n=0}^{+\infty}\nu(C_n\cap S)=\nu(S)=1$. The contradiction obtained proves that $p\in S$.
	
	Suppose that $U\in\mathcal{CO}(\mathbf{S})$, $\nu(U)=1$ but $p\notin U$. Then there exists $V\in\mathcal{CO}(\mathbf{X})$ such that $p\in V\cap S\subseteq S\setminus U$. But this is impossible because $\nu$ is a 2-valued measure on $\mathcal{CO}(\mathbf{S})$, and $\nu(S)=\nu(U)+\nu(S\setminus U)\geq \nu(U)+\mu(V)=2$. Hence, for every $U\in\mathcal{CO}(\mathbf{S})$, it holds that $\nu(U)=1$ if and only if $p\in U$. This implies that $\nu$ is the Dirac measure on $\mathcal{CO}(\mathbf{S})$ determined by $p$. It follows from Theorem \ref{s8:t12} that $\mathbf{S}$ is $\mathbb{N}$-compact.
\end{proof}

Since $\mathbf{CMC}$ follows from $\mathbf{CAC}$ in $\mathbf{ZF}$, Theorem \ref{s8:t13} is a weaker version of Theorem \ref{s7:t10}(i).

\begin{remark}
	\label{s8:r14}
	Assume $\mathbf{ZF+CMC}$. Let $S$ be a non-empty Baire set in a realcompact space $\mathbf{X}$. Let $\nu$ be a countably additive  2-valued $z$-measure on the subspace $\mathbf{S}$ of $\mathbf{X}$. Let us consider the function $\mu:\mathfrak{F}(\mathcal{Z}(\mathbf{X}))\to\{0,1\}$ defined as follows: for every $A\in\mathfrak{F}(\mathcal{Z}(\mathbf{X}))$, 
	$$\mu(A)=\nu(A\cap S).$$
	It follows from Theorem \ref{s8:t8} that there exists $p\in X$ such that $\mu$ is the Dirac $z$-measure on $\mathbf{X}$ determined by $p$. Since every Baire set in a realcompact space is $r$-embedded, the set $S$ is $r$-embedded in $\mathbf{X}$. This, together with the fact that $\mu(S)=1$, implies that $p\in S$. If $A\in\mathfrak{F}(\mathcal{Z}(\mathbf{S}))$ and $p\notin A$, one can show in $\mathbf{ZF+CMC}$ that there exists $C\in\mathfrak{F}(\mathcal{Z}(\mathbf{X}))$ with $p\in C\subseteq X\setminus A$. Hence $\nu$ is the Dirac $z$-measure on $\mathbf{S}$, so $\mathbf{S}$ is realcompact by Theorem \ref{s8:t8} and Proposition \ref{s8:p5}. This is an alternative proof of our Theorem \ref{s7:t10}(i), based on the idea of the $\mathbf{ZFC}$-proof of Corollary 1 of Chapter VIII.4 of \cite{nag}.
\end{remark}

\section{Characters on $U_{\aleph_0}(\mathbf{X})$ when $\mathbf{X}$ is $\mathbb{N}$-compact}
\label{s9}

In what follows, let $\mathbf{X}$ be a non-empty topological space. We recall that, for a real number $c$, we denote by $\mathbf{c}$ the constant function from $C(\mathbf{X})$ such that $\mathbf{c}[X]=\{c\}$. Let us also recall the definition of a character on a subring of $C(\mathbf{X})$ and a definition of a real ideal of a subring of $C(\mathbf{X})$.

\begin{definition}
	\label{s9:d1}
	Let $\mathbb{H}$ be a subring of $C(\mathbf{X})$ which contains all constant functions from $C(\mathbf{X})$.
	\begin{enumerate}
		\item[(a)] A \emph{character} on $\mathbb{H}$ is a function $\chi:\mathbb{H}\to\mathbb{R}$ which satisfies the following conditions:
		\begin{enumerate}
			\item[(i)] $(\forall f,g\in \mathbb{H})\text{ } \chi(f+g)=\chi(f)+\chi(g)$;
			\item[(ii)] $(\forall f,g\in \mathbb{H})\text{ } \chi(f\cdot g)=\chi(f)\cdot\chi(g)$;
			\item[(iii)] $(\forall c\in\mathbb{R})\text{ } \chi(\mathbf{c})=c$.
		\end{enumerate}
		\item[(b)] For $w\in X$, the character on $\mathbb{H}$ determined by $w$ is the function $\chi_w: \mathbb{H}\to\mathbb{R}$ defined as follows:
		$$(\forall f\in\mathbb{H})\text{ } \chi_w(f)=f(w).$$
		\item[(c)] An ideal $M$ of $\mathbb{H}$ is called a \emph{real ideal} of $\mathbb{H}$ if the quotient ring $\mathbb{H}/M$ is isomorphic with the field $\mathbb{R}$.
		\item[(d)] An ideal $M$ of $\mathbb{H}$ is called \emph{fixed} if there exists $p\in X$ such that $M=\{f\in\mathbb{H}: f(p)=0\}$.
	\end{enumerate}
\end{definition} 

A very elegant, elementary proof of the following theorem of $\mathbf{ZF}$  is given in \cite{bo}:

\begin{theorem} 
	\label{s9:t2}(Cf. \cite[Theorem 1]{bo}.) $[\mathbf{ZF}]$ Let $\mathbf{X}$ be a non-empty realcompact space and let $\chi$ be a character on $C(\mathbf{X})$. Then there exists $w\in X$ such that, for every $f\in C(\mathbf{X})$, $\chi(f)=f(w)$.
\end{theorem}

\begin{remark}
	\label{s9:r03}
	One can easily observe that if $\mathbf{X}$ is a functionally Hausdorff space, $\chi$ is a character on $C(\mathbf{X})$, and $w_1, w_2\in X$ are such that, for every $f\in C(\mathbf{X})$, $\chi(f)=f(w_1)=f(w_2)$, then $w_1=w_2$.
\end{remark}

Interesting newer results on characters relevant to $\mathbb{N}$-compactness can be found, for instance, in \cite{ol}. In this section, we are going to prove in $\mathbf{ZF}$ that if $\mathbf{X}$ is a non-empty $\mathbb{N}$-compact space, then every character on $C(\mathbf{X}, \mathbb{R}_{disc})$ is determined by a point of $X$ and, furthermore, $\mathbf{CMC}$ implies that every character on $U_{\aleph_0}(\mathbf{X})$ is also determined by a point of $X$. To do this, we need the following simple lemma:

\begin{lemma}
	\label{s9:l3}
	$[\mathbf{ZF}]$ For a non-empty topological space $\mathbf{X}$, let $\mathbb{H}$ be either $C(\mathbf{X}, \mathbb{R}_{disc})$ or $U_{\aleph_0}(\mathbf{X})$. Let $\chi$ be a character on $\mathbb{H}$. Then the following conditions are all satisfied:
	\begin{enumerate}
		\item[(i)] $(\forall f, g\in\mathbb{H})\text{ } (f\leq g\rightarrow \chi(f)\leq\chi(g))$;
		\item[(ii)] $(\forall f\in\mathbb{H})\text{ } \chi(|f|)=|\chi(f)|$;
		\item[(iii)] $\chi$ is a continuous function from the subspace $\mathbb{H}$ of $C_{u}(\mathbf{X})$ to $\mathbb{R}$ with the natural topology;
		\item[(iv)] if $\mathbb{H}=C(\mathbf{X}, \mathbb{R}_{disc})$, $f\in \mathbb{H}$ and $\chi(f)=0$, then there exists $p\in X$ such that $f(p)=0$.
	\end{enumerate}
\end{lemma}
\begin{proof}
	First, we notice that if $h\in \mathbb{H}$ and, for every $x\in X$, $0\leq h(x)$ then, since the function $\sqrt\square:[0,+\infty)\to\mathbb{R}$ is uniformly continuous, $\sqrt h\in\mathbb{H}$. Hence $\chi(h)=\chi(\sqrt h)\cdot\chi(\sqrt h)$, so $0\leq\chi(h)$. This implies (i), (ii) and (iii). 
	
	To prove (iv), we assume that $\mathbb{H}=C(\mathbf{X}, \mathbb{R}_{disc})$. Let $f\in\mathbb{H}$ be such that $\chi(f)=0$. Suppose that, for every $x\in X$, $f(x)\neq 0$. Then $\frac{1}{f}\in\mathbb{H}$ and $1=\chi(\mathbf{1})=\chi(f)\cdot\chi(\frac{1}{f})=0$ -a contradiction. Hence, there exists $p\in X$ with $f(p)=0$. 
\end{proof}

\begin{remark}
	\label{s9:r4} 
	For a non-empty topological space $\mathbf{X}$, let $\mathbb{H}$ be a subring of $C(\mathbf{X})$ such that every constant function from $C(\mathbf{X})$ belongs to $\mathbb{H}$ and, for every $f\in\mathbb{H}$ with $0\leq f$, $\sqrt f\in\mathbb{H}$.
	\begin{enumerate}
		\item[(i)] One can easily check that every character $\chi$ on $\mathbb{H}$ has properties (i)-(iii) of Lemma \ref{s9:l3}. Furthermore, if a function $\psi: \mathbb{H}\to\mathbb{R}$ has properties (i)-(ii) of Definition \ref{s9:d1}(a), then $\psi$ is a character on $\mathbb{H}$. 
		\item[(ii)] Let $M$ be a real ideal of $\mathbb{H}$. Then $M$ is a maximal ideal. Given an isomorphism $\phi$ of the field $\mathbb{H}/M$ onto the field $\mathbb{R}$,  we can define a mapping $\chi_M:\mathbb{H}\to\mathbb{R}$ as follows: 
		$$(\forall f\in\mathbb{H})\text{ } \chi_M(f)=\phi(f+M).$$
		Then $\chi_M$ is a character on $\mathbb{H}$, and $M=\{f\in\mathbb{H}: \chi_M(f)=0\}$, so $M$ is the kernel of the homomorphism $\chi_M$. 
		\item[(iii)] For every $p\in X$, the set $M^p=\{f\in\mathbb{H}: f(p)=0\}$ is a real ideal of $\mathbb{H}$ and, obviously, $M^{p}$ is fixed.
	\end{enumerate}
\end{remark}

\begin{theorem}
	\label{s9:t5}
	$[\mathbf{ZF}]$ Let $\mathbf{X}$ be a non-empty $\mathbb{N}$-compact space and let $\chi$ be a character on $C(\mathbf{X}, \mathbb{R}_{disc})$. Then there exists a unique $w\in X$ such that, for every $h\in C(\mathbf{X}, \mathbb{R}_{disc})$, $\chi(h)=h(w)$. 
\end{theorem}
\begin{proof}
	There exists a non-empty set $J$ such that $\mathbf{X}$ is homeomorphic to a closed subspace of $\mathbb{N}^J$. For simplicity, we may assume that $\mathbf{X}$ is a closed subspace of the Tychonoff product $\mathbb{R}_{disc}^J$. For every $j\in J$, let $\pi_j: \mathbb{R}^J\to \mathbb{R}$ be the standard $j$-th projection, and let $\psi_j=\pi_j\upharpoonright X$. It is obvious that, for every $j\in J$, $\psi_j\in C(\mathbf{X}, \mathbb{R}_{disc})$. Therefore, we can define a point $w\in\mathbb{R}^J$ as follows:
	$$(\forall j\in J)\text{ } w(j)=\chi(\psi(j)).$$
	Let us prove that $w\in X$. To this aim, suppose that $w\notin X$. Since $X$ is closed in $\mathbb{R}_{disc}^J$, there exists a non-empty finite subset $K$ of $J$ such that 
	$$ X\cap \bigcap\limits_{j\in K}\pi_j^{-1}[\{w(j)\}]=\emptyset.$$
	We define a function $f_0\in C(\mathbf{X}, \mathbb{R}_{disc})$ as follows:
	$$f_0=\sum\limits_{j\in K}|\psi_j-w(j)|.$$
	By Lemma \ref{s9:l3}(ii), $\chi(f_0)=0$. Therefore, it follows from Lemma \ref{s9:l3}(iv) that there exists $z_0\in X$ such that $f_0(z_0)=0$. Then, by the definition of $f_0$, for every $j\in K$, $\psi_j(z_0)=w(j)$. This implies that $z_0\in X\cap \bigcap\limits_{j\in K}\pi_j^{-1}[\{w(j)\}]$, which is impossible. The contradiction obtained proves that $w\in X$. 
	
	Now, consider an arbitrary  $h\in C(\mathbf{X},\mathbb{R}_{disc})$ and show that $\chi(h)=h(w)$. To this aim, we notice that, since $h:\mathbf{X}\to\mathbb{R}_{disc}$ is continuous, there exists a non-empty finite set $H\subseteq J$ such that, if
	$$U=X\cap\bigcap\limits_{j\in H}\pi_j^{-1}[\{w(j)],$$
	then, for every $t\in U$, $h(t)=h(w)$. Let $g\in C(\mathbf{X}, \mathbb{R}_{disc})$ be defined as follows:
	$$ g= |h-\chi(h)|+\sum\limits_{j\in H}|\psi_j-w(j)|.$$
	We infer from Lemma \ref{s9:l3}(ii) that $\chi(g)=0$. By Lemma \ref{s9:l3}(iv), there exists $s_0\in X$ such that $g(s_0)=0$. Then, by the definition of $g$, $h(s_0)=\chi(h)$ and, moreover, for every $j\in H$, $\psi_j(s_0)=w(j)$. This implies that $s_0\in U$ and, in consequence, $\chi(h)=h(s_0)=h(w)$. This shows that $\chi$ is the character $\chi_w$ on $C(\mathbf{X}, \mathbb{R}_{disc})$ determined by $w$. 
	
	Suppose that $v\in X\setminus\{w\}$ is such that $\chi=\chi_v$. Since $\mathbf{X}$ is a zero-dimensional $T_1$-space, there exists $C\in\mathcal{CO}(\mathbf{X})$ such that $v\in C$ and $w\notin C$. Let $f_C\in C(\mathbf{X}, \mathbb{R}_{disc})$ be  defined as follows:
	$$
	f_C(x)=\begin{cases} 1 &\text{ if } x\in X\setminus C;\\
		0 &\text{ if } x\in C.\end{cases}
	$$
	We notice that $\chi(f_C)=f_C(v)=1=f_C(w)=0$. This is impossible. Hence, if $v,w\in X$ are such that $\chi_v=\chi_w$, then $v=w$. 
\end{proof}

\begin{theorem}
	\label{s9:t6}
	$[\mathbf{ZF}]$ Let $\mathbf{X}$ be a non-empty $\mathbb{N}$-compact space and let $\chi$ be a character on $U_{\aleph_0}(\mathbf{X})$. Then $\mathbf{CMC}$ implies that there exists a unique $w\in X$ such that, for every $h\in U_{\aleph_0}(\mathbf{X})$, $\chi(h)=h(w)$. 
\end{theorem}
\begin{proof}
	Since $C(\mathbf{X},\mathbb{R}_{disc})\subseteq U_{\aleph_0}(\mathbf{X})$, it makes sense to define a mapping $\psi$ by the equality: $\psi=\chi\upharpoonright C(\mathbf{X}, \mathbb{R}_{disc})$. Clearly, $\psi$ is a character on $C(\mathbf{X}, \mathbb{R}_{disc})$. By Theorem \ref{s9:t5}, there exists a unique $w\in X$ such that, for every $h\in C(\mathbf{X}, \mathbb{R}_{disc})$, $\psi(h)=h(w)$. Let $\chi_w$ be the character on $U_{\aleph_0}(\mathbf{X})$ determined by the point $w$. Then, for every $h\in C(\mathbf{X}, \mathbb{R}_{disc})$, $\chi_w(h)=\chi(h)$. By Theorem \ref{s1:t21}(c), $\mathbf{CMC}$ implies that the set $C(\mathbf{X}, \mathbb{R}_{disc})$ is dense in $U_{\aleph_0}(\mathbf{X})$. In consequence, it follows from Lemma \ref{s9:l3}(iii) that $\mathbf{CMC}$ implies that  $\chi=\chi_w$. If $v\in X$ and $\chi=\chi_v$, it follows from Theorem \ref{s9:t5} that $v=w$.
\end{proof}

Unfortunately, we are unable to answer the following question:
\begin{question}
	\label{s9:q7}
	Is it provable in $\mathbf{ZF}$ that, for every non-empty $\mathbb{N}$-compact space $\mathbf{X}$, every character on $U_{\aleph_0}(\mathbf{X})$ is determined by a point of $X$?
\end{question}

A well-known result of $\mathbf{ZFC}$ is that a Tychonoff space $\mathbf{X}$ is realcompact if and only if every real ideal of $C(\mathbf{X})$ is fixed. In \cite{com}, it was shown that this characterization of realcompactness by maximal ideals holds also in $\mathbf{ZF}$. In the following theorem we give, in a sense, an analogous characterization of $\mathbb{N}$-compactness. A relevant result on $\mathbb{N}$-compactness is given in \cite[Theorem 2.2]{ol}.

\begin{theorem}
	\label{s9:t8}
	$[\mathbf{ZF}]$
	Let $\mathbf{X}$ be a non-empty Tychonoff space.
	\begin{enumerate} 
		\item[(i)] (See \cite{com}.) $\mathbf{X}$ is realcompact if and only  every real ideal of $C(\mathbf{X})$ is fixed.
		\item[(ii)] If $\mathbf{X}$ is zero-dimensional, then $\mathbf{X}$ is $\mathbb{N}$-compact if and only if every real ideal of $C(\mathbf{X}, \mathbb{R}_{disc})$ is fixed.
		\item[(iii)] If $\mathbf{X}$ is zero-dimensional, then $\mathbf{CMC}$ implies that $\mathbf{X}$ is $\mathbb{N}$-compact if and only if every real ideal of $U_{\aleph_0}(\mathbf{X})$ is fixed.
	\end{enumerate}
\end{theorem}
\begin{proof}
	Let us start with the proof of (iii) which, to some extent, is similar to the proof of Theorem 2.2 in \cite{ol}.
	
	Assume that $\mathbf{X}$ is a non-empty $\mathbb{N}$-compact space, and $M$ is a real ideal of $U_{\aleph_0}(\mathbf{X})$. Let $\psi: U_{\aleph_0}(\mathbf{X})/M\to\mathbb{R}$ be an isomorphism. For every $f\in U_{\aleph_0}(\mathbf{X})$, let $\chi_M(f)=\psi(f+M)$. By Remark \ref{s9:r4}(ii), $\chi_M$ is a character of $U_{\aleph_0}(\mathbf{X})$ whose kerner is equal to $M$. Assuming $\mathbf{CMC}$, we infer from Theorem \ref{s9:t6} that there exists $w\in X$ such that, for every $f\in U_{\aleph_0}(\mathbf{X})$, $\chi_M(f)=f(w)$. Then $M=\{f\in U_{\aleph_0}(\mathbf{X}): f(w)=0\}$, so $M$ is fixed.
	
	Now, assume that $\mathbf{X}$ is a zero-dimensional space that is not $\mathbb{N}$-compact. For every $f\in U_{\aleph_0}(\mathbf{X})$, let $\tilde{f}\in C(v_{\mathbb{N}}\mathbf{X})$ be such that $\tilde{f}\upharpoonright X=f$. Let us assume $\mathbf{CMC}$. We deduce from Theorem \ref{s5:t4}(i) that the mapping $\phi: U_{\aleph_0}(\mathbf{X})\to U_{\aleph_0}(v_{\mathbb{N}}\mathbf{X})$ defined by:
	$$ (\forall f\in U_{\aleph_0}(\mathbf{X}))\text{ } \phi(f)=\tilde{f}$$
	is an isomorphism of the ring $U_{\aleph_0}(\mathbf{X})$ onto the ring $U_{\aleph_0}(v_{\mathbb{N}}\mathbf{X})$. Since $\mathbf{X}$ is not $\mathbb{N}$-compact, we can fix a point $p\in v_{\mathbb{N}}X\setminus X$. Let $M^{p}=\{g\in U_{\aleph_0}(v_{\mathbb{N}}\mathbf{X}): g(p)=0\}$. Then, by Remark \ref{s9:r4}(iii), $M^p$ is a real ideal of  $U_{\aleph_0}(v_{\mathbb{N}}\mathbf{X})$. Hence, the ideal $\phi^{-1}[M^p]=\{f\in U_{\aleph_0}(\mathbf{X}): \tilde{f}\in M^p\}$ is a real ideal of $U_{\aleph_0}(\mathbf{X})$ which is not fixed. This completes the proof of (iii). 
	
	Using Theorem \ref{s9:t2} (respectively, Theorem \ref{s9:t5}) and similar (but accurately modified) arguments to the ones in the proof of (iii), one can prove that (i) (respectively, (ii)) holds in $\mathbf{ZF}$. 
\end{proof}

Let us note that Comfort's argument given in \cite{com} that (i) of Theorem \ref{s9:t8} above holds in $\mathbf{ZF}$ is distinct from ours.

The following question pops up at this moment:

\begin{question}
	\label{s9:q9}
	Is it provable in $\mathbf{ZF}$ that a non-empty zero-dimensional $T_1$-space is $\mathbb{N}$-compact if and only if every real ideal of $U_{\aleph_0}(\mathbf{X})$ is fixed?
\end{question}

\section{When can $U_{\aleph_0}(\mathbf{X})$ be isomorphic with $C(\mathbf{Y})$?}
\label{s10}

Our goal is to prove by applying characters that the following theorem  holds in $\mathbf{ZF}$:

\begin{theorem}
	\label{s10:t1}
	$[\mathbf{ZF}]$ For every non-empty $\mathbb{N}$-compact space $\mathbf{X}$ whose Banaschewski compactification exists, $\mathbf{CMC}$ implies that $\mathbf{X}$ is strongly zero-dimensional if and only if there exists a Tychonoff space $\mathbf{Y}$ such that the rings $U_{\aleph_0}(\mathbf{X})$ and $C(\mathbf{Y})$ are isomorphic.
\end{theorem}

First, we recall that, by Theorem 1.19 (c)(iv), it holds in $\mathbf{ZF}$ that if $\mathbf{X}$ is a Tychonoff space, then  $\mathbf{CMC}$ implies that $U_{\aleph_0}(\mathbf{X})=C(\mathbf{X})$ if and only if $\mathbf{X}$ is strongly zero-dimensional. Hence, we have the following proposition:

\begin{proposition}
	\label{s10:p2}
	$[\mathbf{ZF}]$ Let $\mathbf{X}$ be a non-empty strongly zero-dimensional $T_1$-space. Then $\mathbf{CMC}$ implies that there exists a non-empty Tychonoff space $\mathbf{Y}$ such that the rings $U_{\aleph_0}(\mathbf{X})$ and $C(\mathbf{Y})$ are isomorphic.
\end{proposition}

Our proof of Theorem \ref{s10:t1} will be completed if we show that the following theorem holds:

\begin{theorem}
	\label{s10:t3}
	$[\mathbf{ZF+CMC}]$ Let $\mathbf{X}$ be a non-empty $\mathbb{N}$-compact space such that $\beta_0\mathbf{X}$ exists. Suppose that there exists a non-empty Tychonoff space $\mathbf{Y}$ such that the rings $U_{\aleph_0}(\mathbf{X})$ and $C(\mathbf{Y})$ are isomorphic. Then $\mathbf{X}$ is strongly zero-dimensional.
\end{theorem}
\begin{proof}
	Assuming $\mathbf{ZF+CMC}$, let us fix a non-empty Tychonoff space $\mathbf{Y}$ such that the rings $U_{\aleph_0}(\mathbf{X})$ and $C(\mathbf{Y})$ are isomorphic. By Theorem \ref{s2:t29},  $\beta_{T}\mathbf{Y}$ exists and is homeomorphic with $\beta_0\mathbf{X}$. Therefore, $\beta_T\mathbf{Y}$ is strongly zero-dimensional. By Theorem 1.21, the extension $\beta^{f}\mathbf{Y}$ of $\mathbf{Y}$ is compact and equivalent to $\beta_T\mathbf{Y}$. Hence $\beta^{f}\mathbf{Y}$ is zero-dimensional. By Theorem \ref{s2:t6}, $\mathbf{Y}$ is strongly zero-dimensional. It follows from Theorem \ref{s1:t21} that $C(\mathbf{Y})=U_{\aleph_0}(\mathbf{Y})$.  Without loss of generality, we may assume that $\mathbf{Y}$ is realcompact because we may replace $\mathbf{Y}$ with $v_{\mathbb{R}}\mathbf{Y}$.
	
	Let $\psi: C(\mathbf{Y})\to U_{\aleph_0}(\mathbf{X})$ be a ring isomorphism. Similarly to \cite{ol}, for every $x\in X$, we define a character $\phi_x: C(\mathbf{Y})\to\mathbb{R}$ as follows: 
	$$(\forall f\in C(\mathbf{Y}))\text{ } \phi_x(f)=\psi(f)(x).$$
	Since $\mathbf{Y}$ is realcompact, it follows from Theorem \ref{s9:t2} and Remark \ref{s9:r03} that there exists a mapping $\pi: X\to Y$ such that:
	$$(\forall x\in X)(\forall f\in C(\mathbf{Y})) \text{ } \phi_x(f)=f(\pi(x)).$$
	In much the same way, as in \cite{ol}, we can show that $\pi$ is a homeomorphic embedding of $\mathbf{X}$ into $\mathbf{Y}$ such that $\pi[X]$ is dense in $\mathbf{Y}$. Furthermore, we can show that, for every $f\in U_{\aleph_0}(\pi[X])$, there exists $T\in U_{\aleph_0}(\mathbf{Y})$ with $T\circ\pi= f\circ\pi$. We deduce that if $f\in C(\pi[X], \mathbb{N})$, then there exists $T\in C(\mathbf{Y}, \mathbb{N})$ such that $T\circ\pi=f\circ\pi$. This, together with the $\mathbb{N}$-compactness of $\pi[X]$, implies that $\pi$ is a surjection onto $\mathbf{Y}$. Hence, the spaces $\mathbf{X}$ and $\mathbf{Y}$ are homeomorphic and, in consequence, $\mathbf{X}$ is strongly zero-dimensional as required.
\end{proof}

\begin{corollary}
	\label{s10:c4}
	$[\mathbf{ZF+CMC}]$ 
	Let $\mathbf{X}$ be a non-empty $\mathbb{N}$-compact space for which $\beta_0\mathbf{X}$ exists. Suppose that $\mathbf{Y}$ is a realcompact space such that the rings $U_{\aleph_0}(\mathbf{X})$ and $C(\mathbf{Y})$ are isomorphic. Then $\mathbf{X}$ and $\mathbf{Y}$ are homeomorphic strongly zero-dimensional spaces.
\end{corollary}

\section{The shortlist of open problems}
\label{s11}

Finally, in this section,  for the convenience of readers, we collect the open problems posed in the previous sections and add several new ones to show possible directions for future research in this field.
\begin{enumerate}
	\item  (Problem \ref{s2:q7}.) Is there a model $\mathcal{M}$ of $\mathbf{ZF}$ in which there is a strongly zero-dimensional $T_1$-space $\mathbf{X}$ for which $\beta_0\mathbf{X}$ exists in $\mathcal{M}$ but $\beta\mathbf{X}$ does not exist in $\mathcal{M}$? 
	
	\item  (Problem \ref{s2:q8}.) Is there a model $\mathcal{M}$ of $\mathbf{ZF}$ in which there is a Tychonoff space $\mathbf{X}$ for which $\beta_T\mathbf{X}$ exists in $\mathcal{M}$ but $\beta\mathbf{X}$ does not exist in $\mathcal{M}$? 
	
	\item (Problem \ref{s2:q9}.) Can a Cantor cube fail to be strongly zero-dimensional in a model of $\mathbf{ZF}$?
	
	\item Is it provable in $\mathbf{ZF}$ that, for every strongly zero-dimensional space $\mathbf{X}$, the equality $U_{\aleph_0}(\mathbf{X})=C(\mathbf{X})$ holds? (See Remark \ref{s2:r14}.)
	
	\item Is there a model of $\mathbf{ZF+\neg CMC}$ in which the following statement is true:  ``There exists a non-empty zero-dimensional $T_1$-space $\mathbf{X}$ for which $\beta_0\mathbf{X}$ exists but the spaces $\beta_0\mathbf{X}$ and $\BMax(U_{\aleph_0}(\mathbf{X}))$ are not homeomorphic'' ? (Cf. Theorem \ref{s2:t22}.)
	
	\item Is it provable in $\mathbf{ZF}$ that, for every non-empty Tychonoff space $\mathbf{Y}$ and every non-empty zero-dimensional $T_1$-space $\mathbf{X}$ such that $\beta_0\mathbf{X}$ exists, if the rings $C(\mathbf{Y})$ and $U_{\aleph_0}(\mathbf{X})$ are isomorphic, then $\beta_T\mathbf{Y}$ exists and is homeomorphic with $\beta_0\mathbf{X}$? (See Theorem \ref{s2:t29}(ii).)
	
	\item (Problem \ref{s3:q7}.) Is there a model of $\mathbf{ZF}$ in which there exists a Tychonoff space $\mathbf{X}$ of countable pseudocharacter for which $(\mathbf{X})_z$ is not discrete?
	
	\item Can a locally compact Hausdorff space $\mathbf{E}$ fail to be compactly $\mathbf{E}$-Urysohn in a model of $\mathbf{ZF}$?
	
	\item (Problem \ref{s3:q8}.) Is there a model of $\mathbf{ZF}$ in which there exists a Tychonoff space $\mathbf{X}$ for which $(\mathbf{X})_z\neq (\mathbf{X})_{\delta}$? 
	
	\item (Problem \ref{s5:q8}.) Let $\mathbf{E}$ be a Tychonoff space such that $\mathbb{R}_{disc}$ is $\mathbf{E}$-compact. 
	\begin{enumerate}
		\item[(i)]  If $\mathbf{X}$ is an $\mathbf{E}$-completely regular $P$-space, may $v_{\mathbf{E}}\mathbf{X}$ fail to be a $P$-space in a model of $\mathbf{ZF}$?
		\item[(ii)] If $\mathbf{X}$ is a non-empty $\mathbf{E}$-completely regular space, may the rings $U_{\aleph_0}(\mathbf{X})$ and $U_{\aleph_0}(v_{\mathbf{E}}\mathbf{X})$ fail to be isomorphic in a model of $\mathbf{ZF}$?
	\end{enumerate}
	
	\item Is there a model of $\mathbf{ZF+\neg CMC}$ in which there exists a Tychonoff, not realcompact space $\mathbf{X}$ such that every $z$-ultrafilter in $\mathbf{X}$ with the countable intersection property is fixed? (See Problem \ref{s6:q7} and Corollary \ref{s6:c11}.)
	
	\item Is there a model $\mathcal{M}$ of $\mathbf{ZF+\neg CMC}$ in which there exists a Tychonoff, not realcompact space $\mathbf{X}$ which can be expressible in $\mathcal{M}$ as a countable union of $z$-embedded subspaces of $\mathbf{X}$? (See Theorem \ref{s6:t12}.)
	
	\item Is there a model of $\mathbf{ZF+\neg CMC}$ in which it is true that a $c$-embedded subspace of a zero-dimensional $T_1$-space $\mathbf{X}$ need not be $c_{\delta}$-embedded in $\mathbf{X}$? (See Proposition \ref{s6:p14}.)
	
	\item  If a zero-dimensional $T_1$-space $\mathbf{X}$ is expressible as a ountable union of its $c_{\delta}$-embedded (or $c$-embedded) $\mathbb{N}$-compact subspaces, may $\mathbf{X}$ fail to be $\mathbb{N}$-compact in a model of $\mathbf{ZF+\neg CMC}$? (See Theorem \ref{s6:t16} and Corollary \ref{s6:c17}.)
	
	\item Can a Baire set in a realcompact space fail to be realcompact in a model of $\mathbf{ZF+\neg CMC}$? (See Theorem \ref{s7:t10}(i).)
	
	\item Can a zero-Baire set in an $\mathbb{N}$-compact space fail to be $\mathbb{N}$-compact in a model of $\mathbf{ZF+\neg CMC}$? (See Theorem \ref{s7:t10}(ii).)
	
	\item Can a weakly $\sigma$-additive $z$-measure on a topological space fail to be countably additive in a model of $\mathbf{ZF+\neg CMC}$? (See Proposition \ref{s8:p5} and Definition \ref{s8:d6}.)
	
	\item (Problem \ref{s8:q10}.) May a weakly $\sigma$-additive $z$-measure on a realcompact space $\mathbf{X}$ fail to be countably additive in a model of $\mathbf{ZF}$?
	
	\item  (Question \ref{s9:q7}.) Is it provable in $\mathbf{ZF}$ that, for every non-empty $\mathbb{N}$-compact space $\mathbf{X}$, every character on $U_{\aleph_0}(\mathbf{X})$ is determined by a point of $X$?
	
	\item (Question \ref{s9:q9}.) Is it true in $\mathbf{ZF}$ that a non-empty zero-dimensional $T_1$-space is $\mathbb{N}$-compact if and only if every real ideal of $U_{\aleph_0}(\mathbf{X})$ is fixed?
	
	\item Is there a model $\mathcal{M}$ of $\mathbf{ZF+\neg CMC}$ in which there exists a non-empty strongly zero-dimensional $T_1$-space such that, for every non-empty Tychonoff space $\mathbf{Y}$ in $\mathcal{M}$, the rings $C(\mathbf{Y})$ and $U_{\aleph_0}(\mathbf{X})$ are not isomorphic in $\mathcal{M}$? (See Proposition \ref{s10:p2}.)
	
	\item Is there a model $\mathcal{M}$ of $\mathbf{ZF+\neg CMC}$ in which there exists a non-empty,  $\mathbb{N}$-compact, not strongly zero-dimensional space $\mathbf{X}$ and a non-empty Tychonoff space $\mathbf{Y}$, such that the rings $C(\mathbf{Y})$ and $U_{\aleph_0}(\mathbf{X})$ are isomorphic in $\mathcal{M}$ and $\mathbf{X}$ has its Banaschewski compactification in $\mathcal{M}$? (See Theorem \ref{s10:t3}.)
\end{enumerate}

Of course, many other relevant open problems could be added to this shortlist.

\end{document}